\documentclass[11pt, twoside]{article}

\usepackage{amssymb}
\usepackage{amsmath}
\usepackage{amsthm}
\usepackage{amsfonts}
\usepackage{mathrsfs}
\usepackage{color}

\allowdisplaybreaks

\def\rr{{\mathbb R}}
\def\rn{{{\rr}^n}}

\def\zz{{\mathbb Z}}
\def\cc{{\mathbb C}}
\def\nn{{\mathbb N}}
\def\cp{{\mathcal P}}
\def\cq{{\mathcal Q}}
\def\cf{{\mathcal F}}
\def\cl{{\mathcal L}}
\def\cm{{\mathcal M}}
\def\cn{{\mathcal N}}
\def\mj{{\mathrm J}}

\def\fz{\infty}
\def\az{\alpha}
\def\bz{\beta}

\def\ez{\epsilon}
\def\gz{{\gamma}}

\def\lz{\lambda}

\def\boz{{\Omega}}

\def\epz{\epsilon}

\def\vz{\varphi}

\def\lf{\left}
\def\r{\right}

\def\la{\langle}
\def\ra{\rangle}
\def\hs{\hspace{0.26cm}}

\def\ls{\lesssim}

\def\noz{\nonumber}
\def\wz{\widetilde}
\def\wh{\widehat}
\def\st{\subset}
\def\com{\complement}
\def\bh{\backslash}

\def\cs{{\mathcal S}}

\def\dt{\,\frac{dt}{t}}
\def\dytn{\,\frac{dy\,dt}{t^{n+1}}}

\def\dist{\mathop\mathrm{\,dist\,}}
\def\supp{\mathop\mathrm{\,supp\,}}

\def\loc{{\mathop\mathrm{\,loc\,}}}

\def\q1{\wz q}
\def\Q1{q_1}
\def\lv{{L^{p(\cdot)}(\rn)}}
\def\wlv{W\!L^{p(\cdot)}(\rn)}
\def\hv{{H^{p(\cdot)}(\rn)}}
\def\ha{{H_{\mathrm{atom}}^{p(\cdot),q,s}(\rn)}}
\def\whv{{W\!H^{p(\cdot)}(\rn)}}
\def\wha{{W\!H_{\mathrm{atom}}^{p(\cdot),q,s}(\rn)}}
\def\whm{{W\!H_{\mathrm{mol}}^{p(\cdot),q,s,\epsilon}(\rn)}}
\def\whz{{W\!H_{\mathrm{atom}}^{p(\cdot),\fz,s}(\rn)}}
\def\Bij{{B_{i,j}}}
\def\bij{{b_{i,j}}}
\def\bije{{b_{i,j}^\ez}}
\def\lij{{\lz_{i,j}}}
\def\aij{{a_{i,j}}}
\def\mij{{m_{i,j}}}
\def\xij{{x_{i,j}}}
\def\rij{{r_{i,j}}}
\def\qij{{Q_{i,j}}}
\def\wqij{{\wz Q_{i,j}}}

\def\supp{{\mathop\mathrm{\,supp\,}}}

\def\dist{{\mathop\mathrm{\,dist\,}}}
\def\loc{{\mathop\mathrm{loc\,}}}

\newtheorem{thm}{Theorem}[section]
\newtheorem{prop}[thm]{Proposition}
\newtheorem{lem}[thm]{Lemma}
\newtheorem{cor}[thm]{Corollary}
\theoremstyle{definition}
\newtheorem{defn}[thm]{Definition}
\newtheorem{rem}[thm]{Remark}

\numberwithin{equation}{section}
\numberwithin{equation}{section}

\begin{document}

\arraycolsep=1pt

\title{\vspace{-2cm}\bf\Large Variable Weak Hardy Spaces and Their Applications
\footnotetext{\hspace{-0.35cm} 2010 {\it
Mathematics Subject Classification}. Primary 42B30;
Secondary 42B25, 42B20, 42B35, 46E30.
\endgraf {\it Key words and phrases.} Hardy space, variable exponent,
maximal function, Littlewood-Paley function, atom, molecule, Calder\'on-Zygmund operator.
\endgraf This project is supported by the National
Natural Science Foundation of China
(Grant Nos. 11571039, 11361020 and 11471042).}}
\author{Xianjie Yan, Dachun Yang\,\footnote{Corresponding author},\ \  Wen Yuan and Ciqiang Zhuo}
\date{ }
\maketitle

\vspace{-0.8cm}

\begin{center}
\begin{minipage}{13cm}
{\small {\bf Abstract}\quad
Let $p(\cdot):\ \mathbb R^n\to(0,\infty)$ be a variable exponent
function satisfying the globally log-H\"older continuous condition.
In this article, the authors first introduce the variable weak Hardy space
on $\mathbb R^n$, $W\!H^{p(\cdot)}(\mathbb R^n)$,
via the radial grand maximal function, and then establish its radial or
non-tangential maximal function characterizations.
Moreover, the authors also obtain various equivalent characterizations of $W\!H^{p(\cdot)}(\mathbb R^n)$,
respectively, by means of atoms, molecules,
the Lusin area function, the Littlewood-Paley $g$-function or $g_{\lambda}^\ast$-function.
As an application, the authors establish the boundedness of convolutional
$\delta$-type and non-convolutional $\gamma$-order Calder\'on-Zygmund
operators from $H^{p(\cdot)}(\mathbb R^n)$ to
$W\!H^{p(\cdot)}(\mathbb R^n)$ including the critical case
when $p_-={n}/{(n+\delta)}$ or when $p_-=n/(n+\gamma)$,
where $p_-:=\mathop\mathrm{ess\,inf}_{x\in \rn}p(x).$}
\end{minipage}
\end{center}

\section{Introduction\label{s-intro}}
\hskip\parindent
The main purpose of this article is to introduce and to investigate the
variable weak Hardy spaces on $\rn$. It is well known that
the classical weak Hardy spaces appear naturally in critical cases
of the study on the boundedness of operators.
Indeed,  the classical
weak Hardy space $W\!H^1(\rn)$ was originally introduced by
Fefferman and Soria \cite{fs86} when they tried to find out
the biggest space from which the Hilbert transform
is bounded to the weak Lebesgue space $W\!L^1(\rn)$. Via
establishing the $\fz$-atomic characterization of $W\!H^1(\rn)$, they obtained
the boundedness of some Calder\'on-Zygmund
operators from $W\!H^1(\rn)$ to $W\!L^1(\rn)$.
Moreover, it is also well known that, when studying the boundedness of some
singular integral operators, $H^p(\rn)$ is a good substitute of
the Lebesgue space $L^p(\rn)$ with $p\in(0,1]$; while
when studying the boundedness of operators
in the critical case, the Hardy spaces $H^p(\rn)$ are usually further replaced
by the weak Hardy space $W\!H^p(\rn)$. For example, if $\delta\in(0,1]$ and $T$
is a convolutional $\delta$-type Calder\'on-Zygmund operator with $T^*(1)=0$, where
$T^*$ denotes the \emph{adjoint operator} of $T$, then $T$ is
bounded on $H^p(\rn)$ for all $p\in({n}/{(n+\delta)},1]$ (see \cite{am86}),
but  may not be bounded on $H^{{n}/{(n+\delta)}}(\rn)$. For such an endpoint case,
Liu \cite{liu91} proved that $T$ is bounded from
$H^{{n}/{(n+\delta)}}(\rn)$ to $W\!H^{{n}/{(n+\delta)}}(\rn)$
via establishing the $\fz$-atomic characterization of the weak Hardy space $W\!H^p(\rn)$.

Furthermore, when studying the real interpolation between the Hardy space
$H^p(\rn)$ and the space $L^\fz(\rn)$, Fefferman et al. \cite{frs74} proved
that the weak Hardy spaces $W\!H^p(\rn)$ also naturally appear as the
intermediate spaces, which is another main motivation to
develop a real-variable theory of $W\!H^p(\rn)$. Recently, He \cite{h13}
and Grafakos and He \cite{gh14} further investigated
{vector-valued weak Hardy spaces} $H^{p,\fz}(\rn,\ell^2)$ with
$p\in(0,\fz)$.
Very recently, Liang et al. \cite{lyj} introduced a kind of generalized
weak Hardy spaces of Musielak-Orlicz type $W\!H^\varphi(\rn)$,
which covers both weak Hardy spaces
$W\!H^p(\rn)$ and weighted weak Hardy spaces $W\!H_w^p(\rn)$ from \cite{qy00}.
Various equivalent characterizations
of $W\!H^\varphi(\rn)$ by means of maximal functions, atoms, molecules
and Littlewood-Paley functions, and
the boundedness of Calder\'on-Zygmund operators in the critical case were obtained in \cite{lyj}.
For more related history and properties about $W\!H^p(\rn)$,
we refer the reader to \cite{at07,frs74,fs86,gh14,h13,liu91,lu95,qy00}
and their references.

On the other hand, based on the variable Lebesgue space,
several variable function spaces are developed
rapidly in recent years (see, for example, \cite{aa16,ah10,cw14,dhr09,ns12,
Xu081,yzy15,yzy151}).
Recall that the variable Lebesgue space $\lv$,
with a variable exponent function $p(\cdot):\ \rn\to(0,\fz)$,
is a generalization of the classical Lebesgue space
$L^p(\rn)$, via replacing the constant exponent $p$ by the exponent function
$p(\cdot)$, which consists of all functions $f$
such that $\int_{\rn}|f(x)|^{p(x)}\,dx<\fz$.
The study of variable Lebesgue spaces can be traced back to
Orlicz \cite{or31}, however, they
have been the subject of more intensive study since the early 1990s because
of their intrinsic interest for applications into harmonic analysis, partial
differential equations and variational integrals with nonstandard
growth conditions (see \cite{cfbook,dhr11,ins14} and their references).

Particularly, Nakai and Sawano \cite{ns12} introduced the variable
Hardy spaces $\hv$ and established their atomic
characterizations which were further applied to consider
their dual spaces and the boundedness of singular integral operators on $\hv$.
Later, in \cite{s10}, Sawano extended the atomic characterization
of $\hv$, which also improves the corresponding result in
\cite{ns12}, and gave out more applications of $\hv$, including the boundedness
of several operators on $\hv$. Moreover, Zhuo et al. \cite{zyl}
established equivalent characterizations of $\hv$ via
intrinsic square functions including the intrinsic
Lusin-area function, the intrinsic Littlewood-Paley $g$-function
or $g_\lz^*$-function.
Independently, Cruz-Uribe and Wang \cite{cw14}
also investigated the variable Hardy space $\hv$ with $p(\cdot)$
satisfying some conditions slightly weaker
than those used in \cite{ns12}. In \cite{cw14},
equivalent characterizations of $\hv$ by means of radial or non-tangential
maximal functions or atoms were established.
Very recently, in \cite{yzn15}, Yang et al.
characterized $\hv$ via Riesz transforms with $p(\cdot)$
satisfying the same conditions as in \cite{cw14}.

In this article, via combining some ideas from the theories
of the aforementioned classical weak Hardy spaces
and the variable Hardy spaces from \cite{ns12, cw14}, with the same assumptions
on $p(\cdot)$ as in Nakai and Sawano \cite{ns12}, we introduce and investigate
the variable weak Hardy spaces $\whv$. These spaces are first defined
via the radial grand maximal function
and then characterized by means of radial or non-tangential maximal
functions. Various equivalent characterizations of $\whv$ by means of atoms,
molecules and square functions, including the Lusin area function, the
Littlewood-Paley $g$-function and $g_\lz^*$-function, are also obtained.
As an application, we establish the boundedness of
convolutional $\delta$-type and non-convolutional
$\gz$-order Calder\'on-Zygmund operators from $\hv$ to $\whv$ including the critical
case when $p_-={n}/{(n+\delta)}$ or when $p_-=n/(n+\gamma)$, with $p_-$ as in \eqref{2.1x} below, which is of
special interest. These results further complete the theory
of variable Hardy-type spaces developed by Nakai and Sawano \cite{ns12} (see also
Cruz-Uribe and Wang \cite{cw14}).

To be precise, this article is organized as follows.

In Section \ref{s-pre}, we first recall some notation and notions, and state
some basic properties about variable Lebesgue spaces.
The variable weak Hardy space $\whv$ is also defined in this section
via the radial grand maximal function.

Section \ref{s-max} is devoted to characterizing the space $\whv$
by means of the radial maximal function corresponding to certain
Schwartz function or the non-tangential maximal function
corresponding to Poisson kernels (see Theorem \ref{mthm1} below).
To this end, we first establish a vector-valued inequality of
the Hardy-Littlewood maximal operator $\cm$ on the variable weak
Lebesgue space $\wlv$ in Proposition \ref{mlmveq} below.
Then, by borrowing some ideas from those used in the proofs of
\cite[p.\,91, Theorem 1]{stein93} and \cite[Theorem 2.1.4]{Gra14},
we give out the proof of Theorem \ref{mthm1}.
To prove Proposition \ref{mlmveq}, an interpolation theorem of
sublinear operators on the space $\wlv$ is obtained
(see Theorem \ref{mp1} below), which further induces the boundedness of
$\cm$ on $\wlv$ and may be of independent interest.
We point out that Proposition \ref{mlmveq} also plays an important
role in Section \ref{s-palay} when establishing
the Littlewood-Paley function characterizations of $\whv$.

In Section \ref{s-atom}, by borrowing some ideas from \cite{c77} and
a modified technic based on \cite{lyj},
we establish the atomic characterization of $\whv$.
Indeed, we first introduce the
variable weak atomic Hardy space
$\wha$ in Definition \ref{atd2} below and then
prove $\whv=\wha$ with equivalent quasi-norms (see Theorem \ref{atthm1} below).
To prove that $\wha$ is continuous embedded into $\whv$,
we mainly use a key lemma obtained by Sawano in
\cite[Lemma 4.1]{s10} (also restated as in Lemma 4.5 below), which reduces
some estimates related to $L^{p(\cdot)}(\rn)$ norms for some
series of functions into dealing with the $L^q(\rn)$ norms
of the corresponding functions,
and also the Fefferman-Stein vector-valued inequality of the Hardy-Littlewood
maximal operator $\cm$ on $\lv$
from \cite[Corollary 2.1]{cf06} (also restated as in Lemma 2.4 below).
The proof for the converse embedding is different from that used in the proof for
the corresponding embedding of variable Hardy spaces $H^{p(\cdot)}(\rn)$.
Recall that $L_\loc^1(\rn)\cap H^{p(\cdot)}(\rn)$
is dense in $H^{p(\cdot)}(\rn)$. Hence, to obtain an atomic
decomposition of any distribution $f\in H^{p(\cdot)}(\rn)$,
by a dense argument, it suffices to assume that
$f$ is a function in $L_\loc^1(\rn)$, which makes it
convenient to construct the desired atomic decomposition
(see \cite[Theorem 4.6]{ns12} and \cite[Theorem 7.1]{cw14}).
However, this standard procedure is invalid for the space $\whv$ due to
its lack of a dense function subspace.
To overcome this difficulty, we adopt a strategy used in \cite{lyj},
originated from Calder\'on \cite{c77},
to directly obtain an atomic decomposition of distributions in $\whv$
instead of some dense function subspace.

In Section \ref{s-mole}, we characterize the space $\whv$ via molecules
in Theorem \ref{mothm1} below.
Since each atom of $\whv$ is also a molecule of $\whv$,
due to Theorem \ref{atthm1}, to prove Theorem \ref{mothm1},
it suffices to show that the
variable weak molecular Hardy space
$W\!H_{\rm mol}^{p(\cdot),q,s,\epsilon}(\rn)$ is
continuously embedded into $\whv$.
To this end, the main step is to prove that a
$(p(\cdot),q,s,\epsilon)$-molecule
can be divided into an infinite linear combination of
$(p(\cdot),q,s)$-atoms. Here we use some ideas similar to those used in the proof of
\cite[Theorem 4.13]{hyy} (see also \cite{tg80}).

Section \ref{s-palay} is devoted to establishing some square function
characterizations of the space $\whv$, including characterizations via
the Lusin area function,
the Littlewood-Paley $g$-function or $g_\lz^\ast$-function,
respectively, in Theorems \ref{lpthm1}, \ref{lpthm2} and \ref{10.10.x} below.
We first prove Theorem \ref{lpthm1}, the Lusin area function
characterization of $\whv$, by borrowing some ideas from those
used in the proof of \cite[Theorem 4.5]{lyj} in which
Liang et al. established the Lusin area function characterization of
$W\!H^p(\rn)$ with $p\in(0,1]$ as a special case.
To obtain the Littlewood-Paley $g$-function characterization of $\whv$,
we make full use of an approach initiated by Ullrich \cite{u12}
and further developed by Liang et al. \cite{lsuyy},
which, via a key and technical lemma (see Lemma \ref{lm-12.7} below) and
an auxiliary function $g_{a,\ast}(f)$ (see \eqref{1.21-x} below), gives one
way to control the Littlewood-Paley $g$-function by the Lusin area function.

As an application of the space $\whv$, in Section \ref{s-bou},
we establish the boundedness of the convolutional $\delta$-type
and the non-convolutional $\gamma$-order Calder\'on-Zygmund operators
from $\hv$ to $\whv$ in the critical case when $p_-=n/{(n+\delta)}$ or when $p_-=n/(n+\gamma)$
(see Theorems \ref{bdnthm2}, respectively, \ref{bdnthm3} below).
In this case, any convolutional $\delta$-type or any non-convolution $\gamma$-order
Calder\'on-Zygmund operator may not be bounded
on $H^{p(\cdot)}(\rn)$ even when $p(\cdot)\equiv {\rm constant}\in(0,1]$.
In this sense, the space $\whv$ is a proper substitution of
$H^{p(\cdot)}(\rn)$ in the critical case for the study on the boundedness
of some operators,
which is one of the main motivations to study the variable weak Hardy space $\whv$.

We point out that the approach used in the proofs of Theorems \ref{bdnthm2} and \ref{bdnthm3}
is different from that of \cite[Theorem 1]{liu91}
(see also \cite[p.\,110, Theorem 4.2]{lu95})
in which the boundedness of the operator $T$ from the classical
Hardy space $H^{n/{(n+\delta)}}(\rn)$ to the weak
Hardy space $W\!H^{n/{(n+\delta)}}(\rn)$ was obtained,
and that of \cite[Theorem 5.2]{lyj} in which Liang et al.
established the boundedness of $T$ from the Musielak-Orlicz
Hardy space $H^{\vz}(\rn)$ to the weak Musielak-Orlicz
Hardy space $W\!H^{\vz}(\rn)$ in the critical case, where
$\vz:\ \rn\times [0,\fz)\to[0,\fz)$ is a Musielak-Orlicz growth
function satisfying that, for any given $x\in\rn$,
$\vz(x,\cdot)$ is an Orlicz function and $\vz(\cdot,t)$
is a Muckenhoupt $A_{1}(\rn)$ weight uniformly in $t\in(0,\fz)$
(see \cite[Definition 2.2]{lyj}).
Indeed, when $\vz(x,t):=t^p$ with $p\equiv {\rm constant}\in(0,\fz)$ for
all $x\in\rn$ and $t\in[0,\fz)$, or $\vz$ is as in \cite[Theorem 5.2]{lyj},
the fact that there exists a positive constant $C$ such that,
for any $\bz\in[1,\fz)$, $t\in(0,\fz)$ and any ball $B\st\rn$,
$$\int_{\bz B}\vz(x,t)\,dx\le C \bz^n \int_{ B}\vz(x,t)\,dx$$
plays a crucial role in the proofs of \cite[Theorem 5.2]{lyj} and also
\cite[p.\,110, Theorem 4.2]{lu95}. However, it may not be true
when $\vz(x,t):=t^{p(x)}$ for all $x\in\rn$ and $t\in(0,\fz)$ with
$p(\cdot)$ being a general variable exponent (see \cite[Remark 2.23]{yyz13}),
and hence the methods used in the proofs of \cite[Theorem 5.2]{lyj}
or \cite[p.\,110, Theorem 4.2]{lu95} are invalid in the present article.
To overcome this difficulty, we establish a weak-type
vector-valued inequality of the Hardy-Littlewood
maximal operator $\cm$ on $\lv$ for $p_-=1$ in
Proposition \ref{1.14.x3} below, via an extrapolation theorem
obtained in \cite[Theorem 5.24]{cfbook}.

These variable weak Hardy spaces $\whv$ might also be useful in the study
on the real interpolation between the variable Hardy spaces $\hv$, which is
the main subject of another forthcoming article, to limit the length of this article.
More applications of these variable weak Hardy spaces $\whv$ (for example, in the study
on the endpoint boundedness of operators) are expectable.

Finally, we make some conventions on notation. Let $\nn:=\{1,2,\dots\}$ and
$\zz_+:=\nn\cup\{0\}$. We denote by $C$ a \emph{positive constant}
which is independent of the main parameters,
but may vary from line to line. We use $C_{(\az,\dots)}$
to denote a positive constant depending on the indicated
parameters $\az,\, \dots$. The \emph{symbol}
$A\ls B$ means $A\le CB$. If $A\ls B$ and $B\ls A$, we then write $A\sim B$.
If $E$ is a subset of $\rn$, we denote by $\chi_E$ its
\emph{characteristic function} and by $E^\complement$
the set $\rn\backslash E$.
For all $r\in(0,\fz)$ and $x\in\rn$, denote by $B(x,r)$ the ball
centered at $x$ with the radius $r$, namely,
$B(x,r):=\{y\in\rn:\ |x-y|<r\}.$
For any ball $B$, we use $x_B$ to denote its center and $r_B$ its radius,
and denote by $\lz B$ for any $\lz\in(0,\fz)$ the ball concentric with
$B$ having the radius $\lz r_B$.

\section{Preliminaries\label{s-pre}}
\hskip\parindent
In this section, we aim to introduce the variable weak Hardy space via
the radial grand maximal function.
To this end, we first recall some notation and notions on
variable Lebesgue spaces and then state some of their basic conclusions
to be used in this article.
For an exposition of these concepts, we refer the reader to the monographs
\cite{cfbook,dhr11}.

\subsection{Variable Lebesgue spaces\label{s2.1}}
\hskip\parindent
A measurable function $p(\cdot):\ \rn\to(0,\fz)$ is called a
\emph{variable exponent}.
Denote by $\cp(\rn)$ the \emph{collection of all variable exponents}
$p(\cdot)$ satisfying
\begin{align}\label{2.1x}
0<p_-:=\mathop\mathrm{ess\,inf}_{x\in \rn}p(x)\le
\mathop\mathrm{ess\,sup}_{x\in \rn}p(x)=:p_+<\fz.
\end{align}

In what follows, for any $p(\cdot)\in\cp(\rn)$, we use $p^*(\cdot)$
to denote its \emph{conjugate variable exponent}, namely, for all $x\in\rn$,
$\frac{1}{p(x)}+\frac{1}{p^*(x)}=1$.

For a measurable function $f$ on $\rn$ and $p(\cdot)\in\cp(\rn)$,
the \emph{modular functional} (or, simply, the \emph{modular})
$\varrho_{p(\cdot)}$, associated with $p(\cdot)$, is defined by setting
$$\varrho_{p(\cdot)}(f):=\int_\rn|f(x)|^{p(x)}\,dx$$ and the
\emph{Luxemburg} (also known as the \emph{Luxemburg-Nakano})
\emph{quasi-norm} is given by setting
\begin{equation*}
\|f\|_{\lv}:=\inf\lf\{\lz\in(0,\fz):\ \varrho_{p(\cdot)}(f/\lz)\le1\r\}.
\end{equation*}
Then the \emph{variable Lebesgue space} $\lv$ is defined to be the
set of all measurable functions $f$ such that $\varrho_{p(\cdot)}(f)<\fz$,
equipped with the quasi-norm $\|f\|_{\lv}$.

\begin{rem}\label{r-vlp}
 Let $p(\cdot)\in\cp(\rn)$.

\begin{enumerate}
\item[(i)] It is easy to see that, for all $s\in (0,\fz)$ and $f\in\lv$,
$$\lf\||f|^s\r\|_{\lv}=\|f\|_{L^{sp(\cdot)}(\rn)}^s.$$
Moreover, for all $\lz\in{\mathbb C}$ and $f,\ g\in\lv$,
$\|\lz f\|_{\lv}=|\lz|\|f\|_{\lv}$ and
$$\|f+g\|_{\lv}^{\underline{p}}\le \|f\|_{\lv}^{\underline{p}}
+\|g\|_{\lv}^{\underline{p}},$$
here and hereafter,
\begin{align}\label{2.1y}
\underline{p}:=\min\{p_-,1\}
\end{align}
with $p_-$ as in \eqref{2.1x}.
Particularly, when $p_-\in[1,\fz)$,
$\lv$ is a Banach space (see \cite[Theorem 3.2.7]{dhr11}).

\item[(ii)] For any non-trivial
function $f\in \lv$, by \cite[Proposition 2.21]{cfbook},
we know that $\varrho_{p(\cdot)}(f/\|f\|_{\lv})=1$
and, if $\|f\|_{\lv}\le1$, then $\varrho_{p(\cdot)}(f)\le\|f\|_{\lv}$
(see \cite[Corollary 2.22]{cfbook}).

\item[(iii)] If there exist $\delta,\,c\in(0,\fz)$ such that  $\int_\rn[|f(x)|/\delta]^{p(x)}\,dx\le c$, then it is easy
to see that $\|f\|_{\lv}\le C\delta$, where $C$ is a positive constant
independent of $\delta$, but depending on $p_-$ (or $p_+$) and $c$.

\item[(iv)] If $p_+\in(0,1)$, then it is easy to see that,
for all non-negative functions $f,\ g\in\lv$, the following reverse
Minkowski inequality holds true:
$$\|f\|_{\lv}+\|g\|_{\lv}\le\|f+g\|_{\lv}.$$
\end{enumerate}
\end{rem}

A function $p(\cdot)\in\cp(\rn)$ is said to satisfy the
\emph{globally log-H\"older continuous condition}, denoted by $p(\cdot)\in C^{\log}(\rn)$,
if there exist positive constants $C_{\log}(p)$ and $C_\fz$, and
$p_\fz\in\rr$ such that, for all $x,\ y\in\rn$,
\begin{equation}\label{elog}
|p(x)-p(y)|\le \frac{C_{\log}(p)}{\log(e+1/|x-y|)}
\end{equation}
and
\begin{equation}\label{edecay}
|p(x)-p_\fz|\le \frac{C_\fz}{\log(e+|x|)}.
\end{equation}

For any measurable set $E\subset\rn$ and $r\in(0,\fz)$, let $L^r(E)$
be the set of all measurable functions $f$ such that
$$\|f\|_{L^r(E)}:=\lf[\int_E|f(x)|^r\,dx\r]^{1/r}<\fz.$$
For $r\in(0,\fz)$, denote by $L_{\rm loc}^r(\rn)$ the set of all
$r$-locally integrable functions
on $\rn$.
Recall that the \emph{Hardy-Littlewood maximal operator}
$\cm$ is defined by setting,
for all $f\in L_{\rm loc}^1(\rn)$ and $x\in\rn$,
\begin{align}\label{2.2x}
\cm(f)(x):=\sup_{B\ni x}\frac1{|B|}\int_B |f(y)|\,dy,
\end{align}
where the supremum is taken over all balls $B$ of $\rn$ containing $x$.

\begin{rem}\label{r-hlb}
Let $p(\cdot)\in C^{\log}(\rn)$ and $1<p_-\le p_+<\fz$.
For any $r\in[1,\fz)$, it is easy to see that $rp(\cdot)\in C^{\log}(\rn)$
and hence, for all $f\in L^{rp(\cdot)}(\rn)$,
$$\|\cm(f)\|_{L^{rp(\cdot)}(\rn)}\le C\|f\|_{L^{rp(\cdot)}(\rn)},$$
where $C$ is a positive constant independent of $f$
(see, for example, \cite[Theorem 3.16]{cfbook}).
\end{rem}

The following result is just \cite[Lemma 2.6]{zyl}
(For the case when $p_-\in(1,\fz)$, see also \cite[Corollary 3.4]{Iz10}).

\begin{lem}\label{zhuolemma}
Let $p(\cdot)\in C^{\log}(\rn)$.
Then there exists a positive constant $C$ such that, for all balls $B_1$, $B_2$
of $\rn$ with $B_1\subset B_2$,
$$C^{-1}\lf(\frac{|B_1|}{|B_2|}\r)^{\frac 1{p_-}}
\le\frac{\lf\|\chi_{B_1}\r\|_{L^{p(\cdot)}(\rn)}}
{\lf\|\chi_{B_2}\r\|_{L^{p(\cdot)}(\rn)}}
\le C\lf(\frac{|B_1|}{|B_2|}\r)^{\frac 1{p_+}}.$$
\end{lem}

The following Fefferman-Stein vector-valued inequality of
the maximal operator $\cm$ on the variable Lebesgue space $\lv$
was obtained in \cite[Corollary 2.1]{cf06}.

\begin{lem}\label{mlm1}
Let $r\in(1,\fz)$.
Assume that $p(\cdot)\in C^{\log}(\rn)$ satisfies $1<p_-\le p_+<\fz$.
Then there exists a positive
constant $C$ such that, for all
sequences $\{f_j\}_{j=1}^\fz$ of measurable functions,
$$\lf\|\lf\{\sum_{j=1}^\fz
\lf[\cm(f_j)\r]^r\r\}^{1/r}\r\|_{\lv}
\le C\lf\|\lf(\sum_{j=1}
^\fz|f_j|^r\r)^{1/r}\r\|_{\lv},$$
where $\cm$ denotes the Hardy-Littlewood maximal operator as in \eqref{2.2x}.
\end{lem}

\begin{rem}\label{2.5.y}
Let $p(\cdot)\in C^{\log}(\rn)$  and $\beta\in[1,\fz)$. Then, by Lemma \ref{mlm1}
and the fact that, for all balls $B\subset\rn$ and $r\in(0,\underline{p})$,
$\chi_{\beta B}\le\beta^{\frac{n}{r}}[\cm(\chi_B)]^{\frac{1}{r}}$,
we conclude that there exists a positive
constant $C$ such that, for any sequence $\{B_j\}_{j\in\nn}$ of balls of $\rn$,
$$\lf\|\sum_{j\in\nn}\chi_{\beta B_j}\r\|_{\lv}\le
C\beta^{\frac{n}{r}}\lf\|\sum_{j\in\nn}\chi_{B_j}\r\|_{\lv}.$$
\end{rem}

\subsection{Variable weak Hardy spaces $W\!H^{p(\cdot)}(\rn)$}
\hskip\parindent
In this subsection, we introduce the variable weak Hardy space
via the radial grand maximal function. To this end,
we first recall the definition of the variable weak Lebesgue space $\wlv$, which
is a special case of the variable Lorentz space
$L_{p(\cdot),q(\cdot)}(\rn)$ studied by Kempka and
Vyb\'iral in \cite{kv14}.

\begin{defn}
Let $p(\cdot)\in\cp(\rn)$.
The \emph{variable weak Lebesgue space} $W\!L^{p(\cdot)}(\rn)$
is defined to be the set of all measurable functions $f$ such that
\begin{align}\label{wvlp}
\|f\|_{WL^{p(\cdot)}(\rn)}:=\sup_{\az\in(0,\fz)}\az\lf
\|\chi_{\{x\in\rn:\ |f(x)|>\az\}}\r\|_{L^{p(\cdot)}(\rn)}<\fz.
\end{align}
\end{defn}

\begin{rem}
\begin{enumerate}
\item[(i)] We point out that the variable weak Lebesgue space
$W\!L^{p(\cdot)}(\rn)$ is a suitable substitute
of the variable Lebesgue space $\lv$ when studying the boundedness of
the Hardy-Littlewood maximal operator $\cm$ on $\lv$ when $p_-=1$.
Indeed, if $p_-=1$, then $\cm$ is not bounded on $\lv$
(see \cite[Theorem 3.19]{cfbook}).
However, if $p(\cdot)\in C^{\log}(\rn)$ with $p_-=1$, then $\cm$
is bounded from $\lv$ to $\wlv$ (see \cite[Theorem 3.16]{cfbook}).

\item[(ii)] As a special case of the variable Lorentz space
$L_{p(\cdot),q(\cdot)}(\rn)$ in \cite{kv14},
the space $\wlv$ also naturally appears when considering the
{real interpolation} between $\lv$ and $L^\fz(\rn)$.
More precisely, if $p(\cdot)\in\cp(\rn)$ and
$\theta\in(0,1)$, then it was proved by Kempka and Vyb\'iral in
\cite[Theorem 4.1]{kv14} that
$$(\lv,L^{\fz}(\rn))_{\theta,\fz}=WL^{\wz p(\cdot)}(\rn),$$
where $\frac1{\wz p(\cdot)}:=\frac{1-\theta}{p(\cdot)}$
and $(\lv,L^{\fz}(\rn))_{\theta,\fz}$ denotes the
real interpolation between $\lv$ and $L^{\fz}(\rn)$.
\end{enumerate}
\end{rem}

\begin{rem}\label{10.13.x}
In \cite{lyj}, Liang et al. introduced the \emph{weak Musielak-Orlicz space}
$W\!L^{\vz}(\rn)$, with $\vz:\ \rn\times[0,\fz)\to[0,\fz)$ being a Musielak-Orlicz
growth function (see \cite[Definition 2.2]{lyj}), which is defined as
the set of all measurable functions $f$ such that
$$\|f\|_{WL^\vz(\rn)}:=\inf\lf\{\lz\in(0,\fz):\ \sup_{\az\in(0,\fz)}
\int_{\{x\in\rn:\ |f(x)|>\az\}}\vz\lf(x,\frac{\az}{\lz}\r)\,dx\le1\r\}<\fz.$$
Here we claim that, for any $p(\cdot)\in\cp(\rn)$,
when $\vz(x,t):=t^{p(x)}$ for all $x\in\rn$ and $t\in(0,\fz)$,
$W\!L^\vz(\rn)=W\!L^{p(\cdot)}(\rn)$.

To see this, it suffices to show that, for any measurable function $f$,
\begin{align}\label{1.23x}
\|f\|_{\wlv}
&=\inf\lf\{\lz\in(0,\fz):\ \sup_{\az\in(0,\fz)}
\int_{\{x\in\rn:\ |f(x)|>\az\}}\lf(\frac{\az}{\lz}\r)^{p(x)}\,dx\le1\r\}\noz\\
&=:\|f\|_{\wlv}^\ast.
\end{align}
Indeed, for any  $\az\in(0,\fz)$, let $E_\az:=\{x\in\rn:\ |f(x)|>\az\}$.
Then, by Remark \ref{r-vlp}(ii), we know that, for any $\az\in(0,\fz)$,
\begin{align*}
\int_{E_\az}\lf[\frac{\az}{\|f\|_{\wlv}}\r]^{p(x)}\,dx
&\le\int_{E_\az}\lf[\frac{\az}{\az\lf
\|\chi_{E_\az}\r\|_{\lv}}\r]^{p(x)}\,dx\\
&=\int_{E_\az}\lf[\frac{1}{\lf
\|\chi_{E_\az}\r\|_{\lv}}\r]^{p(x)}\,dx=1,
\end{align*}
which shows that $\|f\|_{\wlv}^\ast\le \|f\|_{\wlv}$.
On the other hand, for any $\az\in(0,\fz)$, we easily find that
\begin{align*}
\alpha\|\chi_{E_\alpha}\|_{\lv}
&=\inf\lf\{\lz\in(0,\fz):\ \int_{E_\alpha}\lf(\frac{\alpha}{\lz}\r)
^{p(x)}\,dx\le1\r\}\\
&\le\inf\lf\{\lz\in(0,\fz):\ \sup_{\alpha\in(0,\fz)}\int_{E_\alpha}\lf(\frac{\alpha}{\lz}\r)
^{p(x)}\,dx\le1\r\}=\|f\|_{\wlv}^\ast,
\end{align*}
which implies that $\|f\|_{\wlv}\le\|f\|_{\wlv}^\ast$.
Therefore, \eqref{1.23x} holds true. This finishes
the proof of the above claim.
\end{rem}

Next we present some properties of the variable weak Lebesgue space $\wlv$.

\begin{lem}\label{mlm2}
Let $p(\cdot)\in \cp(\rn)$. Then $\|\cdot\|_{\wlv}$
defines a quasi-norm on $\wlv$, namely,
\begin{enumerate}
\item[{\rm (i)}] $\|f\|_{\wlv}=0$ if and only if $f=0$ almost everywhere;

\item[{\rm (ii)}] for all $\lz\in\mathbb C$ and
$f\in\wlv$, $\|\lz f\|_{\wlv}=|\lz|\|f\|_{\wlv}$;

\item[{\rm (iii)}] for any $f,\ g\in\wlv$,
$$\|f+g\|_{\wlv}^{\underline{p}}\le 2^{\underline{p}}
\lf[\|f\|_{\wlv}^{\underline{p}}+\|g\|_{\wlv}^{\underline{p}}\r],$$
where $\underline{p}$ is as in \eqref{2.1y}.
\end{enumerate}
\end{lem}

\begin{proof}
We only give the proofs of (ii) and (iii), since
(i) is obviously true, the details being omitted.

To prove (ii), without loss of generality, we may assume that
 $\lz\neq0$. By the definition of $\|\cdot\|_{\wlv}$ in \eqref{wvlp}, we have
\begin{align*}
\|\lz f\|_{\wlv}&=\sup_{\az\in(0,\fz)}\az\lf
\|\chi_{\{x\in\rn:\ |\lz f(x)|>\az\}}\r\|_{\lv}\\
&=|\lz|\sup_{\az\in (0,\fz)}\frac{\az}{|\lz|}\lf
\|\chi_{\{x\in\rn:\ |f(x)|>\frac{\az}{|\lz|}\}}\r\|_{\lv}
=|\lz|\|f\|_{\wlv}.
\end{align*}
Thus, (ii) holds true.

To show (iii), for any $f,\ g\in\wlv$, by \eqref{wvlp} and Remark \ref{r-vlp}(i),
we find that
\begin{align*}
&\|f+g\|_{\wlv}^{\underline{p}}\\
&\hs=\sup_{\az\in(0,\fz)}\az^{\underline{p}}\lf
\|\chi_{\{x\in\rn:\ |f(x)+g(x)|>\az\}}\r\|_{\lv}^{\underline{p}}\\
&\hs\le\sup_{\az\in(0,\fz)}\az^{\underline{p}}\lf[\lf
\|\chi_{\{x\in\rn:\ |f(x)|>\frac{\az}{2}\}}\r\|_{\lv}^{\underline{p}}
+\lf\|\chi_{\{x\in\rn:\ |g(x)|>\frac{\az}{2}\}}\r\|_{\lv}^{\underline{p}}\r]\\
&\hs\le\sup_{\az\in(0,\fz)}\az^{\underline{p}}\lf
\|\chi_{\{x\in\rn:\ |f(x)|>\frac{\az}{2}\}}\r\|_{\lv}^{\underline{p}}
+\sup_{\az\in(0,\fz)}\az^{\underline{p}}\lf\|
\chi_{\{x\in\rn:\ |g(x)|>\frac{\az}{2}\}}\r\|_{\lv}^{\underline{p}}\\
&\hs\le 2^{\underline{p}}\lf[\|f\|_{\wlv}^{\underline{p}}
+\|g\|_{\wlv}^{\underline{p}}\r],
\end{align*}
namely, (iii) holds true. This finishes the proof of Lemma \ref{mlm2}.
\end{proof}

\begin{rem}\label{r-ar}
Let $p(\cdot)\in\cp(\rn)$.
Then, by the Aoki-Rolewicz theorem (see \cite{ta42, sr57}
and also \cite[Exercise 1.4.6]{g09-1}),
we find that there exists a positive constant $v\in(0,1)$
such that, for all $R\in\nn$ and $\{f_j\}_{j=1}^R$,
$$\lf\|\sum_{j=1}^R |f_j|\r\|_{\wlv}^{v}
\le 4\lf[\sum_{j=1}^R \|f_j\|_{\wlv}^{v}\r].$$
\end{rem}

\begin{lem}\label{mlmim}
Let $p(\cdot)\in \cp(\rn)$.
Then, for all $f\in\wlv$ and $s\in(0,\fz)$, it holds true that
$$\lf\||f|^s\r\|_{\wlv}=\|f\|_{WL^{sp(\cdot)}(\rn)}^s.$$
\end{lem}

\begin{proof}
By \eqref{wvlp} and Remark \ref{r-vlp}(i), we find that
\begin{align*}
\lf\||f|^s\r\|_{\wlv}&=\sup_{\az\in(0,\fz)}\az\lf\|
\chi_{\{x\in\rn:\ |f(x)|^s>\az\}}\r\|_{\lv}\\
&=\sup_{\bz\in(0,\fz)}\bz^s\lf\|
\chi_{\{x\in\rn:\ |f(x)|>\bz\}}\r\|_{\lv}\\
&=\sup_{\bz\in(0,\fz)}\bz^s\lf\|
\chi_{\{x\in\rn:\ |f(x)|>\bz\}}\r\|_{L^{sp(\cdot)}(\rn)}^s
=\|f\|_{WL^{sp(\cdot)}(\rn)}^s.
\end{align*}
This finishes the proof of Lemma \ref{mlmim}.
\end{proof}

From the Fatou lemma of $L^{p(\cdot)}(\rn)$ (see \cite[Theorem 2.61]{cfbook}),
we easily deduce the following Fatou lemma of $\wlv$, the details being omitted.

\begin{lem}\label{10.24.x1}
Let $p(\cdot)\in\cp(\rn)$ and $\{f_k\}_{k\in\nn}\subset\wlv$.
If $f_k\to f$ as $k\to\fz$ pointwise almost everywhere in $\rn$ and
$\liminf_{k\to\fz}\|f_k\|_{\wlv}$ is finite,
then $f\in\wlv$ and
$$\|f\|_{\wlv}\le\liminf_{k\to\fz}\|f_k\|_{\wlv}.$$
\end{lem}

In what follows, denote by $\cs(\rn)$ the \emph{space of all
Schwartz functions} and $\cs'(\rn)$ its \emph{topological dual space}
equipped with the weak-$*$ topology.
For any $N\in\nn$, let
\begin{align}\label{2.4x}
\cf_N(\rn):=\lf\{\psi\in\cs(\rn):\ \sum_{\bz\in\zz_+^n,\,|\bz|\le N}
\sup_{x\in\rn}\lf[(1+|x|)^N|D^\bz\psi(x)|\r]\le1\r\},
\end{align}
where, for any $\bz:=(\bz_1,\dots,\bz_n)\in\zz_+^n$,
$|\bz|:=\bz_1+\cdots+\bz_n$ and
$D^\bz:=(\frac\partial{\partial x_1})^{\bz_1}
\cdots(\frac\partial{\partial x_n})^{\bz_n}$.
Then, for all $f\in\cs'(\rn)$, the \emph{radial grand maximal function}
$f^\ast_{N,+}$
of $f$ is defined by setting, for all $x\in\rn$,
\begin{equation}\label{2.8x}
f_{N,+}^\ast(x):=\sup\lf\{|f\ast\psi_t(x)|:\
t\in(0,\fz)\ {\rm and}\ \psi\in\cf_N(\rn)\r\},
\end{equation}
where, for all $t\in(0,\fz)$ and $\xi\in\rn$,
$\psi_t(\xi):=t^{-n}\psi(\xi/t)$.

Now we introduce the variable weak Hardy space.

\begin{defn}
Let $p(\cdot)\in C^{\log}(\rn)$ and
$N\in(\frac{n}{\underline{p}}+n+1,\fz)$ be a positive integer,
where $\underline{p}$ is as in \eqref{2.1y}.
The \emph{variable weak Hardy space} $\whv$ is defined to be the set of all
$f\in\cs'(\rn)$ such that
$f_{N,+}^\ast\in\wlv$, equipped with the {quasi-norm}
$$\|f\|_{\whv}:=\|f_{N,+}^\ast\|_{\wlv}.$$
\end{defn}

\begin{rem}\label{1.29.x2}
\begin{enumerate}
\item[(i)] If $p(\cdot)\equiv p\in(0,\fz)$,
then the space $\whv$ is just the classical weak Hardy space
$W\!H^p(\rn)$ studied in \cite{frs74,fs86,liu91}. By Theorem \ref{mthm1}
below, we find that the space $\whv$ is independent of the choice of
$N\in(\frac{n}{\underline{p}}+n+1,\fz)$.

\item[(ii)] Very recently, Almeida et al. \cite{abr16} introduced the
variable anisotropic Hardy-Lorentz spaces on
$\rn$ associated with an expansive matrix $A$,
$H^{p(\cdot),q(\cdot)}(\rn, A)$, via the variable Lorentz spaces
$\cl^{p(\cdot),q(\cdot)}(\rn)$ in \cite{eks08}, where
$$p(\cdot),\ q(\cdot):\ (0,\fz)\to(0,\fz)$$
are bounded measurable functions. As was pointed out in
\cite[Remark 2.6]{kv14}, the space $\cl^{p(\cdot),q(\cdot)}(\rn)$
in \cite{eks08} never goes back to the space $\lv$,
since the variable exponent $p(\cdot)$ in $\cl^{p(\cdot),q(\cdot)}(\rn)$
is only defined on $(0,\fz)$ but not on $\rn$.

On the other hand, the space $\whv$ introduced in this article
is defined via the variable Lorentz space
$L_{p(\cdot),q(\cdot)}(\rn)$ from \cite{kv14} but with $q(\cdot)\equiv\fz$,
which is not covered by the space $H^{p(\cdot),q(\cdot)}(\rn, A)$
in \cite{abr16}.
Moreover, as was mentioned in \cite[p.\,5]{abr16},
the main procedure of \cite{abr16} requires the fact that
the set $L_\loc^1(\rn)\cap H^{p(\cdot),q(\cdot)}(\rn, A)$
is dense in $H^{p(\cdot),q(\cdot)}(\rn, A)$.
Thus, the method used in \cite{abr16} does not work for
$\whv$ in the present article,
due to the lack of a dense function subspace of $\whv$ even when
$p(\cdot)\equiv {\rm constant}\in(0,\fz)$.

\item[(iii)] Recall that Liang et al. \cite{lyj} introduced the weak
Musielak-Orlicz Hardy space $W\!H^\varphi(\rn)$ with a
Musielak-Orlicz function
$\varphi:\ \rn\times[0,\fz)\rightarrow[0,\fz)$. Observe that, when
\begin{equation}\label{liang1}
\vz(x,t):=t^{p(x)}\quad\mathrm{for\ all}\quad x\in\rn\quad\mathrm{and}\quad
t\in (0,\fz),
\end{equation}
then $W\!H^\varphi(\rn)=\whv$ (see also Remark \ref{10.13.x}).
However, a general Musielak-Orlicz
function $\varphi$ satisfying all the assumptions
in \cite{lyj} may not have the form as in \eqref{liang1}.
On the other hand, it was proved in \cite{yyz13} that
there exists a variable exponent function $p(\cdot)$
satisfying \eqref{elog} and \eqref{edecay} which were required in this article,
but $t^{p(\cdot)}$ is not a uniformly Muckenhoupt
weight which was required in \cite{lyj}. Thus, the weak Musielak-Orlicz
Hardy space $W\!H^\varphi(\rn)$ in \cite{lyj} and the variable weak Hardy space $\whv$ in this article can not cover each other.
\end{enumerate}
\end{rem}

\section{Maximal function characterizations of $\whv$\label{s-max}}

\hskip\parindent
In this section, we aim to characterize $W\!H^{p(\cdot)}(\rn)$
via radial or non-tangential maximal functions.
To this end, we first establish an interpolation theorem on $\wlv$ in
Subsection \ref{s3.1}. Via applying such an interpolation theorem,
the maximal function characterizations of $\whv$ are
obtained in Subsection \ref{s3.2}.

\subsection{An interpolation theorem on $\wlv$\label{s3.1}}

\hskip\parindent
In what follows, let $\mathscr{M}_0(\rn)$ be the linear space of
all almost everywhere finite measurable functions on $\rn$.
Let $T$ be an operator defined on $\mathscr{M}_0(\rn)$. Then $T$ is
called a \emph{sublinear operator} if,
for all $f,\ g\in \mathscr{M}_0(\rn)$ and all $\lz\in\cc$,
$$|T(f+g)|\le |T(f)|+|T(g)|\quad
{\rm and}\quad |T(\lz f)|=|\lz||T(f)|.$$
For $q_1(\cdot),\ q_2(\cdot)\in\cp(\rn)$, let
$$L^{q_1(\cdot)}(\rn)+L^{q_2(\cdot)}(\rn)
:=\lf\{f\in \mathscr{M}_0(\rn):
\ f=f_1+f_2,\ f_k\in L^{q_k(\cdot)}(\rn),\ k\in\{1,2\}\r\}.$$

The main result of this section is stated as follows.

\begin{thm}\label{mp1}
Let $p(\cdot)\in \cp(\rn)$ with $1<p_-\le p_+<\fz$,
$p_1\in(\frac{1}{p_-},1)$ and $p_2\in(1,\fz)$,
where $p_-$ and $p_+$ are as in \eqref{2.1x}.
Assume that $T$ is a sublinear operator defined on
$L^{p_1p(\cdot)}(\rn)+L^{p_2p(\cdot)}(\rn)$ satisfying that
there exist positive constants $C_1$ and $C_2$ such that,
for all $i\in\{1,2\}$, $f\in L^{p_ip(\cdot)}(\rn)$ and $\beta\in(0,\fz)$,
\begin{equation}\label{interp-1}
\beta\lf\|\chi_{\{x\in\rn:\ |T(f)(x)|>\beta\}}\r\|_{L^{p_ip(\cdot)}(\rn)}
\le C_i\|f\|_{L^{p_ip(\cdot)}(\rn)}.
\end{equation}
Then $T$ is bounded on $\wlv$ and there exists a positive constant
$C$ such that, for all $f\in \wlv$,
\begin{equation*}
\|T(f)\|_{\wlv}\le C\|f\|_{\wlv}.
\end{equation*}
\end{thm}

To prove Theorem \ref{mp1}, we need the following lemma, whose proof is quite easy,
the details being omitted.

\begin{lem}\label{r-2.1x}
Let $p(\cdot)\in \cp(\rn)$.
Then, for any $t\in(0,\fz)$ and $x\in\rn$, it holds true that
$$t^{p(x)}\sim\int_0^tr^{p(x)}\frac{dr}{r},$$
where the implicit equivalent positive
constants are independent of $t$ and $x$.
\end{lem}

\begin{proof}[Proof of Theorem \ref{mp1}]
Let $f\in \wlv$ and
$$\lz:=\|f\|_{\wlv}=\sup_{\bz\in(0,\fz)}\bz
\lf\|\chi_{\{x\in\rn:\ |f(x)|>\bz\}}\r\|_{\lv}.$$
Then, by Remark \ref{r-vlp}(ii), we easily know that, for all $\bz\in(0,\fz)$,
\begin{align}\label{3.1x}
\int_{\{x\in\rn:\ |f(x)|>\bz\}}\lf(\frac{\bz}{\lz}\r)^{p(x)}\,dx\le1.
\end{align}
Next we show that, for all $\az\in(0,\fz)$,
\begin{equation*}
\az\lf\|\chi_{\{x\in\rn:\ \lf|T(f)(x)\r|>\az\}}\r\|_{\lv}\ls\lz
\end{equation*}
with the implicit positive constant independent of $\az$ and $f$.

To this end, for any $\az\in(0,\fz)$, let
$$f_{\az,1}:=f\chi_{\{x\in\rn:\ |f(x)|>\az\}}\ \ \ \mbox{and}\ \ \ f_{\az,2}
:=f\chi_{\{x\in\rn:\ |f(x)|\le\az\}}.$$
We claim that, for $i\in\{1,2\}$,
\begin{align}\label{emb}
\int_{\rn}\lf[\frac{\lf|
f_{\az,i}(x)\r|/{\az}}{(\lz/{\az})
^{1/{p_i}}}\r]^{p_ip(x)}\,dx\ls 1.
\end{align}
Assuming that this claim holds true for the time being, then,
by Remark \ref{r-vlp}(iii), we find that,
for $i\in\{1,2\}$,
$$\lf\|f_{\az,i}/{\az}\r\|_{L^{p_ip(\cdot)}(\rn)}
\ls (\lz/{\az})^{1/{p_i}},$$
which shows that $f_{\az,i}\in L^{p_ip(\cdot)}(\rn)$
and ${\az}^{1-p_i}\lf\|f_{\az,i}\r\|_
{L^{p_ip(\cdot)}(\rn)}^{p_i}\ls\lz<\fz$.
From this and the fact that $T$ is sublinear, Remark \ref{r-vlp}(i)
and \eqref{interp-1}, we deduce that, for any $\az\in(0,\fz)$,
\begin{align*}
&\az\lf\|\chi_{\{x\in\rn:\ |T(f)(x)|>\az\}}\r\|_{\lv}\\
&\hs\ls\az\lf\|\chi_{\{x\in\rn:\
|T(f_{\az,1})(x)|>\az/2\}}\r\|_{\lv}+
\az\lf\|\chi_{\{x\in\rn:\
|T(f_{\az,2})(x)|>\az/2\}}\r\|_{\lv}\\
&\hs\sim\sum_{i=1}^2\az\lf\|\chi_{\{x\in\rn:\
|T(f_{\az,i})(x)|>\az/2\}}\r\|_{L^{p_ip(\cdot)}(\rn)}^{p_i}
\ls\sum_{i=1}^2{\az}^{1-p_i}\lf\|f_{\az,i}\r\|
_{L^{p_ip(\cdot)}(\rn)}^{p_i}\ls\lz.
\end{align*}
This further implies that $\|T(f)\|_{\wlv}\ls\|f\|_{\wlv}$,
which is the desired conclusion.

Therefore, to complete the proof of Theorem \ref{mp1},
it remains to prove the above claim.

To this end, when $i=1$,
by Lemma \ref{r-2.1x} and $p(x)\ge p_->1$ for almost every $x\in\rn$, we have
\begin{align}\label{emd}
&\int_{\rn}\lf[\frac{\lf|f_{\az,1}(x)\r|/{\az}}
{(\lz/{\az})^{1/{p_1}}}\r]^{p_1p(x)}\,dx\noz\\
&\hs=\int_{\rn}\frac{1}{p(x)}\int_0^\frac{[|f_{\az,1}(x)|/
{\az}]^{p_1}}{\lz/{\az}}t^{p(x)}\frac{dt}{t}\,dx
\le \frac{1}{p_-}\int_0^{\fz}\int
_{\{x\in\rn:\ [|f_{\az,1}(x)|/{\az}]
^{p_1}>\frac{t\lz}{\az}\}}
t^{p(x)}\frac{dx\,dt}{t}\noz\\
&\hs\le\int_0^{\az/\lz}
\int_{\{x\in\rn:\ [|f_{\az,1}(x)|
/{\az}]^{p_1}>\frac{t\lz}
{\az}\}}t^{p(x)}\frac{dx\,dt}{t}
+\int_{\az/\lz}^{\fz}\int_{\{x\in\rn:\
[|f_{\az,1}(x)|/{\az}]
^{p_1}>\frac{t\lz}{\az}\}}\cdots\noz\\
&\hs=:{\rm I}_1+{\rm I}_2.
\end{align}
By the definition of $f_{\az,1}$ and \eqref{3.1x}, we conclude that
\begin{align}\label{emd1}
{\rm I}_1&\le\int_0^{\az/\lz}
\int_{\{x\in\rn:\ |f(x)|>\az\}}\lf(\frac{\az}
{\lz}\r)^{p(x)}\lf(\frac{t\lz}{\az}\r)
^{p_-}\frac{dx\,dt}{t}\noz\\
&=\int_0^{\az/\lz}\lf(\frac{t\lz}{\az}\r)
^{p_-}\frac{dt}{t}
\int_{\{x\in\rn:\ |f(x)|>\az\}}\lf(\frac{\az}
{\lz}\r)^{p(x)}\,dx\noz\\
&\le\int_0^{\az/\lz}\lf(\frac{t\lz}{\az}\r)
^{p_-}\,\frac{dt}{t}\sim 1.
\end{align}
On the other hand, from the fact that $p_1\in(\frac{1}{p_-},1)$,
the definition of $f_{\az,1}$ and \eqref{3.1x} again, we deduce that
\begin{align}\label{emd2}
\hs{\rm I}_2&\le\int_{\az/\lz}^
{\fz}\int_{\{x\in\rn:\ |f_{\az,1}(x)|
>\az\lf(\frac{t\lz}{\az}\r)^{1/{p_1}}\}}
\lf[\frac{\az}{\lz}\lf
(\frac{t\lz}{\az}\r)^{1/{p_1}}\r]^{p(x)}
\lf(\frac{t\lz}{\az}\r)^
{(1-1/p_1)p_-}\,dx\frac{dt}{t}\noz\\
&\le\int_{\az/\lz}^{\fz}
\lf(\int_{\{x\in\rn:\ |f(x)|>\az\lf(\frac{t\lz}{\az}\r)^{1/{p_1}}\}}
\lf[\frac{\az}{\lz}\lf
(\frac{t\lz}{\az}\r)^{1/{p_1}}\r]^{p(x)}\,dx\r)
\lf(\frac{t\lz}{\az}\r)^
{(1-1/p_1)p_-}\frac{dt}{t}\noz\\
&\le\int_{\az/\lz}^
{\fz}\lf(\frac{t\lz}{\az}\r)^
{(1-1/p_1)p_-}\,\frac{dt}{t}
\sim 1.
\end{align}
Thus, by \eqref{emd}, \eqref{emd1}
and \eqref{emd2}, we obtain \eqref{emb} when $i=1$.

When $i=2$, by Lemma \ref{r-2.1x} and $p(x)\ge p_->1$ for almost every $x\in\rn$, we know that
\begin{align}\label{eme}
&\int_{\rn}\lf[\frac{|f_{\az,2}(x)|/{\az}}
{(\lz/{\az})^{1/{p_2}}}\r]^{p_2p(x)}\,dx\noz\\
&\hs=\int_{\rn}\frac1{p(x)}\int_0^
\frac{[|f_{\az,2}(x)|
/{\az}]^{p_2}}
{\lz/{\az}}t^{p(x)}\,\frac{dt}{t}\,dx
\le \int_0^{\fz}\int_
{\{x\in\rn:\ [|f_{\az,2}(x)|
/{\az}]^{p_2}>\frac{t\lz}{\az}\}}
t^{p(x)}\,dx\,\frac{dt}{t}\noz\\
&\hs=\int_0^{\az/\lz}\int_{\{x\in\rn:\
[|f_{\az,2}(x)|/{\az}]
^{p_2}>\frac{t\lz}{\az}\}}t^{p(x)}\,dx\,\frac{dt}{t}
+\int_{\az/\lz}^{\fz}\int_{\{x\in\rn:\
[|f_{\az,2}(x)|/{\az}]
^{p_2}>\frac{t\lz}{\az}\}}\cdots\noz\\
&\hs=:{\rm II}_1+{\rm II}_2.
\end{align}
By the definition of $f_{\az,2}$, \eqref{3.1x} and the fact that
$p_2\in(1,\fz)$, we find that
\begin{align}\label{eme1}
\hs\hs{\rm II}_1&\le\int_0^{\az/\lz}
\int_{\{x\in\rn:\ |f_{\az,2}(x)|
>\az\lf(\frac{t\lz}{\az}\r)^{1/{p_2}}\}}
\lf[\frac{\az}{\lz}\lf(\frac{t\lz}{\az}\r)^{1/{p_2}}\r]
^{p(x)}\lf(\frac{t\lz}{\az}\r)^
{(1-1/p_2)p_-}\,dx\,\frac{dt}{t}\noz\\
&\le\int_0^{\az/\lz}
\lf(\int_{\{x\in\rn:\ |f(x)|>\az\lf(\frac{t\lz}{\az}\r)^{1/{p_2}}\}}
\lf[\frac{\az}{\lz}\lf(\frac{t\lz}{\az}\r)^{1/{p_2}}\r]
^{p(x)}\,dx\r)\lf(\frac{t\lz}{\az}\r)^
{(1-1/p_2)p_-}\,\frac{dt}{t}\noz\\
&\ls\int_0^{\az/\lz}\lf(\frac{t\lz}{\az}\r)
^{(1-1/p_2)p_-}\,\frac{dt}{t}
\sim 1.
\end{align}
Observe that, when $t\in(\frac{\az}{\lz}, \fz)$,
$(|f_{\az,2}|/{\az})^{p_2}<1<\frac{t\lz}{\az}$
and hence ${\rm II}_2=0$.
From this, \eqref{eme} and \eqref{eme1},
we deduce that \eqref{emb} holds true when $i=2$,
which implies that the above claim \eqref{emb}
holds true.
This finishes the proof of Theorem \ref{mp1}.
\end{proof}

As a simple consequence of Theorem \ref{mp1} and Remark \ref{r-hlb},
we immediately obtain the following boundedness of $\cm$ on $\wlv$,
which is of independent  interest, the details being omitted.

\begin{cor}\label{mc1}
Let $p(\cdot)\in C^{\log}(\rn)$ satisfy $1<p_-\le p_+<\fz$, where
$p_-$ and $p_+$ are as in \eqref{2.1x}.
Then the Hardy-Littlewood maximal operator $\cm$ is bounded on $\wlv$.
\end{cor}

Moreover, using Theorem \ref{mp1} and Lemma \ref{mlm1}, we now establish the
following vector-valued inequality of the Hardy-Littlewood maximal
operator $\cm$ on variable weak Lebesgue spaces.

\begin{prop}\label{mlmveq}
Let $p(\cdot)\in C^{\log}(\rn)$ satisfy $1<p_-\le p_+<\fz$,
with $p_-$ and $p_+$ as in \eqref{2.1x}, and $r\in(1,\fz)$.
Then there exists a positive
constant $C$ such that, for all sequences $\{f_j\}_{j\in\nn}$
of measurable functions,
$$\lf\|\lf\{\sum_{j\in\nn}\lf[\cm (f_j)\r]^r\r\}^{1/r}\r\|_{\wlv}
\le C\lf\|\lf(\sum_{j\in\nn}|f_j|^r\r)^{1/r}\r\|_{\wlv},$$
where $\cm$ denotes the Hardy-Littlewood maximal operator as in \eqref{2.2x}.
\end{prop}

\begin{proof}
To prove this proposition, let $\{f_j\}_{j\in\nn}$ be a given arbitrary sequence of
measurable functions and, for any measurable function
$g$ and $x\in\rn$, define
$$A(g)(x):=\lf\{\sum_{j\in\nn}
\lf[\cm(g\eta_j)(x)\r]^r\r\}^{\frac{1}{r}},$$
where $r\in(1,\fz)$ and, for any $y\in\rn$,
$$\eta_j(y):=\frac{f_j(y)}
{[\sum_{j\in\nn}|f_j(y)|^r]^{1/r}}\quad
{\rm if}\quad \lf[\sum_{j\in\nn}\lf|f_j(y)\r|^r\r]^{1/r}\neq0,$$
and $\eta_j(y):=0$ otherwise.
Then, by the Minkowski inequality, we find that, for any measurable functions
$g_1$, $g_2$ and $x\in\rn$,
\begin{align*}
&A(g_1+g_2)(x)\\
&\hs=\lf\{\sum_{j\in\nn}\lf[
\cm\lf([g_1+g_2]\eta_j\r)(x)\r]^r\r\}^{\frac{1}{r}}
\le\lf\{\sum_{j\in\nn}\lf[\cm(g_1\eta_j)(x)+
\cm(g_2\eta_j)(x)\r]^r\r\}^{\frac{1}{r}}\\
&\hs\le\lf\{\sum_{j\in\nn}
\lf[\cm(g_1\eta_j)(x)\r]^r\r\}^{\frac{1}{r}}
+\lf\{\sum_{j\in\nn}\lf[\cm(g_2\eta_j)(x)\r]^r\r\}
^{\frac{1}{r}}
=A(g_1)(x)+A(g_2)(x).
\end{align*}
Thus, $A$ is sublinear. Moreover, by Lemma \ref{mlm1},
we know that, for any $p_1\in(\frac{1}{p_-},1)$, $p_2\in(1,\fz)$ and
measurable function $h$,
\begin{align*}
\lf\|A(h)\r\|_{L^{p_ip(\cdot)}(\rn)}
&=\lf\|\lf\{\sum_{j\in\nn}\lf[\cm(h\eta_j)\r]^r\r\}^{\frac{1}{r}}\r\|
_{L^{p_ip(\cdot)}(\rn)}\\
&\ls\lf\|\lf(\sum_{j\in\nn}|h\eta_j|^r\r)^{\frac{1}{r}}\r\|
_{L^{p_ip(\cdot)}(\rn)}
\sim\|h\|_{L^{p_ip(\cdot)}(\rn)},
\end{align*}
which implies that the operator $A$ is bounded on
$L^{p_ip(\cdot)}(\rn)$, where $i\in\{1,2\}$.
Now, letting $g:=(\sum_{j=1}^\fz|f_j|^r)^{1/r}$,
then, by Theorem \ref{mp1}, we conclude that
\begin{align*}
\lf\|\lf\{\sum_{j\in\nn}\lf[\cm(f_j)\r]^r\r\}^{\frac{1}{r}}\r\|_{\wlv}
&=\lf\|A(g)\r\|_{\wlv}\\
&\ls\|g\|_{\wlv}
\sim\lf\|\lf(\sum_{j\in\nn}\lf|f_j\r|^r\r)^{1/r}\r\|_{\wlv},
\end{align*}
which completes the proof of Proposition \ref{mlmveq}.
\end{proof}

\subsection{Maximal function characterizations of $\whv$\label{s3.2}}

\hskip\parindent
We begin with the following definitions of the radial maximal
function and the non-tangential maximal function.

\begin{defn}
Let $\psi\in\cs(\rn)$ and $\int_\rn \psi(x)\,dx\neq0$.
Let $f\in\cs'(\rn)$.
The \emph{radial maximal function} of $f$
associated to $\psi$ is defined by setting, for all $x\in\rn$,
\begin{align}\label{3.5x}
\psi_{+}^\ast(f)(x):=\sup_{t\in(0,\fz)}\lf|f*\psi_t(x)\r|
\end{align}
and, for any $a\in(0,\fz)$, the \emph{non-tangential maximal function} of $f$
associated to $\psi$ is defined by setting, for all $x\in\rn$,
$$\psi_{\triangledown,a}^\ast(f)(x)
:=\sup_{t\in(0,\fz),\,|y-x|<at}\lf|f*\psi_t(y)\r|.$$
When $a=1$, we simply use $\psi_{\triangledown}^\ast(f)$ to denote
$\psi_{\triangledown,a}^\ast(f)$.
\end{defn}

In what follows, for any $N\in\nn$ and $a\in(0,\fz)$,
the \emph{non-tangential grand maximal function} of $f\in\cs'(\rn)$
is defined by setting, for all $x\in\rn$,
$$f_{N,\triangledown,a}^{\ast}(x):=\sup_{\psi\in\cf_N(\rn)}
\sup_{t\in(0,\fz),\,|y-x|<at}\lf|f*\psi_t(y)\r|,$$
where $\cf_N(\rn)$ is as in \eqref{2.4x}.
When $a=1$, we simply denote $f_{N,\triangledown,a}^{\ast}$ by
$f_{N,\triangledown}^\ast$.

\begin{rem}\label{r-m-fun}
Let $f\in\cs'(\rn)$.
\begin{enumerate}
\item[(i)] From the definitions of $f_{N,+}^\ast$ and $f_{N,\triangledown}^\ast$,
and \cite[Proposition 2.1]{zyl} (see also \cite[Lemma 7.9]{cw14}),
we deduce that there exists a positive constant
$C$ such that, for all $x\in\rn$,
$$C^{-1}f_{N,\triangledown}^\ast(x)\le f_{N,+}^\ast(x)
\le Cf_{N,\triangledown}^\ast(x).$$

\item[(ii)] For any $a\in(0,\fz)$ and $\psi\in\cs(\rn)$, it is easy to see that
$\psi_{\triangledown,a}^\ast(f)\le C f_{N,\triangledown}^\ast$ pointwise,
where $C$ is a positive constant independent of $f$.
\end{enumerate}
\end{rem}

A distribution $f\in\cs'(\rn)$ is called a \emph{bounded distribution} if, for all
$\phi\in\cs(\rn)$, $f\ast \phi\in L^\fz(\rn)$. For a bounded distribution $f$,
its \emph{non-tangential maximal function},
with respect to {Poisson kernels} $\{P_t\}_{t>0}$, is defined by setting,
for all $x\in\rn$,
$$\cn (f)(x):=\sup_{t\in(0,\fz),\,|y-x|<t}\lf|f*P_t(y)\r|,$$
where, for all $x\in\rn$ and $t\in(0,\fz)$,
\begin{align}\label{1.23y}
P_t(x):=\frac{\Gamma([n+1]/2)}{\pi^{(n+1)/2}}
\frac{t}{(t^2+|x|^2)^{(n+1)/2}}
\end{align}
and $\Gamma$ denotes the Gamma function.

The following conclusion is the main result of this subsection, which
gives out the maximal function characterizations of the space $\whv$.

\begin{thm}\label{mthm1}
Let $p(\cdot)\in C^{\log}(\rn)$.
Suppose that $N\in(\frac{n}{\underline{p}}+n+1,\fz)$
is a positive integer, where $\underline{p}$ is as in \eqref{2.1y}.
Then the following items are equivalent:
\begin{enumerate}
\item[{\rm (i)}] $f\in\whv$, namely, $f\in\cs'(\rn)$ and $f_{N,+}^\ast\in\wlv$;

\item[{\rm (ii)}] $f$ is a bounded distribution and $\cn (f)\in\wlv$;

\item[{\rm (iii)}] $f\in\cs'(\rn)$ and there exists a $\psi\in\cs(\rn)$ with
$\int_{\rn}\psi(x)\,dx=1$ such that $\psi_{+}^\ast(f)\in\wlv$.
\end{enumerate}

Moreover, for any $f\in \whv$, it holds true that
$$\lf\|f_{N,+}^\ast\r\|_{\wlv}\sim\lf\|\cn(f)\r\|_{\wlv}
\sim\lf\|\psi_{+}^\ast(f)\r\|_{\wlv},$$
where the implicit equivalent positive constants are independent of $f$.
\end{thm}

\begin{proof}
{\bf STEP 1}: In this step, we show
${\rm (i)}\Rightarrow{\rm (ii)}\Rightarrow{\rm (iii)}.$

Suppose that ${\rm (i)}$ holds true, namely, $f\in\cs'(\rn)$ and $f_{N,+}^\ast\in\wlv$
with $N$ as in Theorem \ref{mthm1}.
To prove (ii), we first show that $f$ is a bounded distribution.
Indeed, by Remark \ref{r-m-fun}(i), we easily know that
there exists a positive constant $C_{(N)}$ such that,
for any $\phi\in\cs(\rn)$, $x\in\rn$ and $y\in B(x, 1)$,
$|f*{\phi}(x)|\le\frac1{C_{(N)}}f_{N,+}^\ast(y)$,
since
$$|f*{\phi}(x)|\ls f_{N,\triangledown}^\ast(y)\ls f_{N,+}^\ast(y).$$
Thus, for any $x\in\rn$, we have
$$B(x, 1)\subset\{y\in\rn:\ f_{N,+}^\ast(y)
\ge C_{(N)}|f*\phi(x)|\}=:\Omega_{f,x}.$$
By this and Remark \ref{r-vlp}(ii), we conclude that
\begin{align}\label{max-f3}
&\min\{\lf|f*\phi(x)\r|^{p_-}, \lf|f*\phi(x)\r|^{p_+}\}\noz\\
&\hs\le\min\{\lf|f*\phi(x)\r|^{p_-}, \lf|f*\phi(x)\r|^{p_+}\}
\frac{1}{|B(x,1)|}\int_{\rn}\chi_{\Omega_{f,x}}(y)\,dy\noz\\
&\hs\ls\min\{\lf|f*\phi(x)\r|^{p_-}, \lf|f*\phi(x)\r|^{p_+}\}
\int_{\Omega_{f,x}}\lf[\frac{1}{\|\chi_{\Omega_{f,x}}\|_{\lv}}\r]^{p(y)}
\|\chi_{\Omega_{f,x}}\|_{\lv}^{p(y)}\,dy\noz\\
&\hs\ls\min\{\lf|f*\phi(x)\r|^{p_-},\lf|f*\phi(x)\r|^{p_+}\}
\max\lf\{\lf\|\chi_{\Omega_{f,x}}\r\|_{\lv}^{p_-},
\lf\|\chi_{\Omega_{f,x}}\r\|_{\lv}^{p_+}\r\}\noz\\
&\hs\ls\max\lf\{\|f_{N,+}^\ast\|_{\wlv}^{p_-},\|f_{N,+}^\ast\|
_{\wlv}^{p_+}\r\}<\fz,
\end{align}
which implies that $f\ast \phi\in L^\fz(\rn)$.
Therefore, $f$ is a bounded distribution.

Next, we show that $\cn(f)\in \wlv$. By \cite[p.\,98]{stein93},
we know that, for all $x\in\rn$,
$$P_1(x)=\sum_{k=0}^{\fz}2^{-k}\psi_{2^k}^{(k)}(x),$$
where $\{\psi^{(k)}\}_{k\in\nn}\st\cs(\rn)$ have uniformly bounded seminorms
in $\cs(\rn)$ and $P_1$ is the Poisson kernel as in \eqref{1.23y} with $t=1$.
From this decomposition, it follows that, for any $t\in(0,\fz)$, $x\in\rn$ and
$y\in B(x,t)$,
\begin{align*}
\lf|f*P_t(y)\r|\le\sum_{k=0}^{\fz}2^{-k}\lf|f*{\psi}_{2^kt}^{(k)}(y)\r|
\le\sum_{k=0}^{\fz}2^{-k}(\psi^{(k)})_{\triangledown}^\ast (f)(x),
\end{align*}
which implies that, for all $x\in\rn$,
\begin{equation}\label{emeqp}
\cn(f)(x)\le \sum_{k=0}^{\fz}2^{-k}(\psi^{(k)})_{\triangledown}^\ast (f)(x).
\end{equation}
Since $\{\psi^{(k)}\}_{k\in\nn}$ have uniformly bounded seminorms in $\cs(\rn)$,
it follows, from \eqref{emeqp}, Remarks \ref{10.24.x1},
\ref{r-ar} and \ref{r-m-fun}, that
\begin{align*}
\lf\|\cn (f)\r\|_{\wlv}^v
&\le\lf\|\sum_{k=0}^{\fz}2^{-k}
(\psi^{(k)})_{\triangledown}^\ast (f)\r\|_{\wlv}^v\\
&\ls\sum_{k=0}^{\fz}2^{-kv}
\lf\|(\psi^{(k)})_{\triangledown}^\ast (f)\r\|_{\wlv}^v
\ls\lf\|f_{N,+}^\ast\r\|_{\wlv}^v,
\end{align*}
where $v$ is as in Remark \ref{r-ar}.
This shows that $\cn(f)\in\wlv$ and hence (ii) holds true.

Finally, assume that (ii) holds true, namely,
$f$ is a bounded distribution and $\cn (f)\in\wlv$.
Then, by \cite[p.\,99]{stein93}, we know that there exists $\psi\in\cs(\rn)$
with $\int_\rn\psi(x)\,dx=1$ such that, for all $x\in\rn$,
$\psi_{+}^\ast(f)(x)\ls \cn (f)(x)$.
Therefore, (iii) holds true, which completes the proof of STEP 1.

{\bf STEP 2}: In this step, we prove ${\rm (iii)}\Rightarrow{\rm (i)}$.

Assume that (iii) holds true, namely, $f\in\cs'(\rn)$ and there exists
$\psi\in\cs(\rn)$ with $\int_{\rn}\psi(x)\,dx=1$ such that
$\psi_{+}^\ast(f)\in\wlv$.
Since $N\in(\frac{n}{\underline{p}}+n+1,\fz)$, it follows that there exists a positive constant
$T>\frac{n}{\underline{p}}$ such that $N>T+n+1$.
From this and \cite[(3.1)]{cw14},
it follows that, for all $x\in\rn$,
\begin{equation}\label{emeq1}
f_{N,+}^\ast(x)\ls M_{\psi,T}(f)(x),
\end{equation}
where
$$M_{\psi,T}(f)(x):=\sup_{t\in(0,\fz),\,y\in\rn}|f*\psi_t(x-y)
|\lf(1+\frac{|y|}{t}\r)^{-T},\quad\forall\, x\in\rn.$$
On the other hand, by the proof of \cite[Theorem 2.1.4(c)]{Gra14},
we find that, for $q:=\frac{n}{T}$ and all $x\in\rn$,
$[M_{\psi,T}(f)(x)]^q\le\cm([\psi_{\triangledown}^\ast(f)]^q)(x).$
Thus, by the fact that $T>\frac n{\underline p}$,
Lemma \ref{mlmim} and Corollary \ref{mc1},
we conclude that
\begin{align}\label{max-f2}
\lf\|M_{\psi,T}(f)\r\|_{\wlv}
&=\lf\|\lf[M_{\psi,T}(f)\r]^q\r\|
_{WL^{p(\cdot)/q}(\rn)}^{1/q}
\le\lf\|\cm\lf(\lf[\psi_{\triangledown}^\ast(f)\r]^q\r)\r\|
_{WL^{p(\cdot)/q}(\rn)}^{1/q}\noz\\
&\ls\lf\|\lf[\psi_{\triangledown}^\ast(f)\r]^q\r\|
_{WL^{p(\cdot)/q}(\rn)}^{1/q}
\sim\lf\|\psi_{\triangledown}^\ast(f)\r\|_{\wlv}.
\end{align}
Now we claim that
\begin{align}\label{emeq3}
\lf\|\psi_{\triangledown}^\ast(f)\r\|_{\wlv}
\ls\lf\|\psi_{+}^\ast(f)\r\|_{\wlv}.
\end{align}
Assuming that this claim holds true for the time being, then,
due to \eqref{emeq1} and \eqref{max-f2}, we have
$$\lf\|f_{N,+}^\ast\r\|_{\wlv}\ls
\lf\|\psi_{+}^\ast(f)\r\|_{\wlv}<\fz,$$
which implies that (i) holds true.

To show the claim \eqref{emeq3}, we first assume that
$\psi_{\triangledown}^\ast(f)\in\wlv$, which will be proved later.
For any $\eta\in(0,\fz)$, let
$E:=\{x\in\rn:\ f_{N,+}^\ast(x)<\eta \psi_{\triangledown}^\ast(f)(x)\}.$
Then, by \eqref{emeq1} and \eqref{max-f2}, we know that there exists
a positive constant $C_0$ such that
\begin{align*}
\lf\|\psi_{\triangledown}^\ast(f)\chi_{E^{\com}}\r\|_{\wlv}
&\le\frac{1}{\eta}\lf\|f_{N,+}^\ast\chi_{E^{\com}}\r\|_{\wlv}\\
&\le\frac{1}{\eta}\lf\|f_{N,+}^\ast\r\|_{\wlv}
\le\frac{C_0}{\eta}\lf\|\psi_{\triangledown}^\ast(f)\r\|_{\wlv}.
\end{align*}
Thus, by Lemma \ref{mlm2}(iii), we find that
\begin{align*}
\lf\|\psi_{\triangledown}^\ast(f)\r\|_{\wlv}
&\le \widetilde{C}\lf[\lf\|\psi_{\triangledown}^\ast(f)\chi_E\r\|_{\wlv}
+\lf\|\psi_{\triangledown}^\ast(f)\chi_{E^{\com}}\r\|_{\wlv}\r]\\
&\le \widetilde{C}\lf\|\psi_{\triangledown}^\ast(f)\chi_E\r\|_{\wlv}
+\frac{\widetilde{C}C_0}{\eta}\lf\|\psi_{\triangledown}^\ast(f)\r\|_{\wlv},
\end{align*}
where $\widetilde{C}$ is a positive constant independent of $f$ and $\eta$.
By this and via choosing $\eta:=2\widetilde{C}C_0$, we conclude that
\begin{align}\label{emeq5}
\lf\|\psi_{\triangledown}^\ast(f)\r\|_{\wlv}
\le2\widetilde{C}\lf\|\psi_{\triangledown}^\ast(f)\chi_E\r\|_{\wlv}.
\end{align}
On the other hand, by \cite[p.\,96]{stein93},
we know that, for any $q\in(0,\underline{p})$ and all $x\in E$,
$$\psi_{\triangledown}^\ast(f)(x)
\ls\lf[\cm\lf(\lf[\psi_{+}^\ast(f)\r]^{q}\r)(x)\r]^{1/{q}},$$
which, combined with Lemma \ref{mlmim} and Corollary \ref{mc1}, implies that
\begin{align*}
\lf\|\psi_{\triangledown}^\ast(f)\chi_E\r\|_{\wlv}
&\ls\lf\|\lf[\cm\lf(\lf[\psi_{+}^\ast(f)\r]^{q}\r)\r]
^{1/{q}}\r\|_{\wlv}
\sim\lf\|\cm\lf(\lf[\psi_{+}^\ast(f)\r]^{q}\r)\r\|
_{WL^{p(\cdot)/{q}}(\rn)}^{1/{q}}\\
&\ls\lf\|\lf[\psi_{+}^\ast(f)\r]^{q}\r\|_
{WL^{p(\cdot)/{q}}(\rn)}^{1/{q}}\sim\lf\|\psi_{+}^\ast(f)\r\|_{\wlv}.
\end{align*}
From this and \eqref{emeq5}, we deduce that \eqref{emeq3} holds true.

Next we show that $\psi_{\triangledown}^\ast(f)\in\wlv$.
To this end, for any $\epsilon\in(0,\frac{1}{3}),\ L\in(0,\fz)$,
$T\in(\frac{n}{\underline p},\fz)$,
$\wz N\in\mathbb N$ and $x\in\rn$, let
$$M_{\psi}^{\epz, L}(f)(x):=\sup_{t\in(0,1/\epsilon)}\lf|f*\psi_t(x)\r|
\frac{t^L}{(t+\epsilon+\epsilon|x|)^L},$$
$$M_{\wz N}^{\epz, L}(f)(x):=\sup_{\psi\in\cf_{\wz N}(\rn)}
M_{\psi}^{\epz, L}(f)(x),$$
$$M_{\psi, 1}^{\epz, L}(f)(x):=
\sup_{t\in(0,1/\epz),\,|y-x|<t}\lf|f*{\psi}_t(y)\r|
\frac{t^L}{(t+\epz+\epz|y|)^L}$$
and
$$\bar {M}_{\psi, T}^{\epz, L}(f)(x):=
\sup_{t\in(0,1/\epz),\,y\in\rn}\lf|f*{\psi}_t(x-y)\r|
\lf(1+\frac{|y|}{t}\r)^{-T}\frac{t^L}{(t+\epz+\epz|x-y|)^L}.$$

By \cite[p.\,45]{Gra14}, we know that there exist positive constants $m$ and $l$
such that, for all $\epsilon\in(0,\frac{1}{3})$ and $x\in\rn$,
$$M_{\psi, 1}^{\epz, L}(f)(x)
\sim\sup_{t\in(0,1/\epz),\,|y-x|<t}\lf|f*{\psi}_t(y)\r|
\lf(\frac{t}{t+\epz}\r)^L\frac{1}{(1+\epz|y|)^L}
\ls\frac{C_{(f,\psi,\epsilon,n,l,m,L)}}{(1+\epsilon|x|)^{L-m}}.$$
Observe that, for all $x\in\rn$,
$$(1+\epsilon|x|)^{m-L}
\le\epsilon^{m-L}(1+|x|)^{m-L}\ls\epsilon^{m-L}
[\cm(\chi_{B(0,1)})(x)]^{\frac{L-m}{n}}.$$
By this via taking $L\in(m+\frac{n}{p_-},\fz)$, Lemma \ref{mlmim}
and Corollary \ref{mc1}, we conclude that, for all $\epsilon\in(0,\frac{1}{3})$,
$$\lf\|M_{\psi, 1}^{\epz, L}(f)\r\|_{\wlv}
\ls\lf\|\cm(\chi_{B(0,1)})\r\|
_{WL^{\frac{(L-m)p(\cdot)}{n}}(\rn)}^{\frac{L-m}{n}}
\ls\|\chi_{B(0,1)}\|_{\wlv}<\fz,$$
which implies that, for all $\epsilon\in(0,\frac{1}{3})$,
$M_{\psi, 1}^{\epz, L}(f)\in\wlv$. Moreover, by \cite[p.\,460]{cw14}, we know that
there exists $\wz N\in[T+L+n+1,\fz)$ such that, for all
$\epsilon\in(0,\frac1{3})$ and $x\in\rn$,
\begin{align}\label{emeqex}
M_{\wz N}^{\epz, L}(f)(x)\ls \bar {M}_{\psi, T}^{\epz, L}(f)(x).
\end{align}
For any $\lz\in(0,\fz)$, let
$F:=\{x\in\rn:\ M_{\wz N}^{\epz, L}(f)(x)<\lz M_{\psi, 1}^{\epz, L}(f)(x)\}$.
Then, by an argument similar to that used in the proof of
\cite[(3.7)]{cw14}, we know that, for all $x\in F$,
\begin{align}\label{10.1.x}
M_{\psi, 1}^{\epz, L}(f)(x)
\ls \lf\{\cm\lf(\lf[\psi_{+}^\ast(f)\r]
^{\underline{p}}\r)(x)\r\}^{1/{\underline{p}}}.
\end{align}

On the other hand, observe that, for all $\epsilon\in(0,1/3)$, $t\in(0,1/\epz)$
and $x,\ y\in\rn$, it holds true that, for all $z\in B(x-y,t)$,
$$\lf|\psi_t*f(x-y)\r|\frac{t^L}{(t+\epsilon+\epsilon|x-y|)^L}
\le M_{\psi, 1}^{\epz, L}(f)(z).$$
By this and the fact that $B(x-y,t)\subset B(x,|y|+t)$, we find that,
for $q=n/T$, $\epsilon\in(0,1/3)$, $t\in(0,1/\epz)$
and $x,\ y\in\rn$,
\begin{align*}
&\lf[\lf|\psi_t*f(x-y)\r|\frac{t^L}{(t+\epsilon+\epsilon|x-y|)^L}\r]^q\\
&\hs\le\frac{|B(x,|y|+t)|}{|B(x-y,t)|}\frac{1}{|B(x,|y|+t)|}
\int_{B(x,|y|+t)}\lf[M_{\psi, 1}^{\epz, L}(f)\r]^q(z)\,dz\\
&\hs\le\lf(1+\frac{|y|}{t}\r)^n\cm\lf(\lf[M_{\psi, 1}^{\epz, L}(f)\r]^q\r)(x),
\end{align*}
namely,
$$\lf[\lf|\psi_t*f(x-y)\r|\frac{t^L}{(t+\epsilon+\epsilon|x-y|)^L}
\lf(1+\frac{|y|}{t}\r)^{-T}\r]^q
\le\cm\lf(\lf[M_{\psi, 1}^{\epz, L}(f)\r]^q\r)(x),$$
which further implies that
$$\lf[\bar {M}_{\psi, T}^{\epz, L}(f)(x)\r]^q
\le\cm\lf(\lf[M_{\psi, 1}^{\epz, L}(f)\r]^q\r)(x).$$
From this, Lemma \ref{mlmim}, the fact that $q<\underline{p}$
and Corollary \ref{mc1}, we deduce that
\begin{align*}
\lf\|\bar {M}_{\psi, T}^{\epz, L}(f)\r\|_{\wlv}
&=\lf\|\lf[\bar {M}_{\psi, T}^{\epz, L}(f)\r]^q\r\|
_{WL^{p(\cdot)/q}(\rn)}^{1/q}
\le\lf\|\cm\lf(\lf[M_{\psi, 1}^{\epz, L}(f)\r]^q\r)\r\|
_{WL^{p(\cdot)/q}(\rn)}^{1/q}\\
&\ls\lf\|\lf[M_{\psi, 1}^{\epz, L}(f)\r]^q\r\|
_{WL^{p(\cdot)/q}(\rn)}^{1/q}
\sim\lf\|M_{\psi, 1}^{\epz, L}(f)\r\|_{\wlv},
\end{align*}
which, combined with \eqref{emeqex}, implies that
\begin{align}\label{emeqex1}
\lf\|M_{\wz N}^{\epz, L}(f)\r\|_{\wlv}
\ls\lf\|M_{\psi, 1}^{\epz, L}(f)\r\|_{\wlv}.
\end{align}

By \eqref{10.1.x}, \eqref{emeqex1}, the fact that
$M_{\psi, 1}^{\epz, L}(f)\in\wlv$ and an argument
similar to that used in the proof of \eqref{emeq3},
we conclude that, for all $\epz\in(0,1/3)$,
$$\lf\|M_{\psi, 1}^{\epz, L}(f)\r\|_{\wlv}
\ls \lf\|\psi_{+}^\ast(f)\r\|_{\wlv}<\fz$$
with the implicit positive constant independent of $\epz$.
By this, the fact that $M_{\psi, 1}^{\epz, L}(f)$ increases pointwise to
$\psi_{\triangledown}^\ast(f)$ as $\epz\rightarrow0$
for any $L\in(0,\fz)$ and Remark \ref{10.24.x1}, we find that
\begin{align*}
\lf\|\psi_{\triangledown}^\ast(f)\r\|_{\wlv}
&\le\liminf_{\epsilon\to0}\lf\|M_{\psi, 1}^{\epz, L}(f)\r\|_{\wlv}
\ls\lf\|\psi_{+}^\ast(f)\r\|_{\wlv}<\fz,\noz
\end{align*}
which implies that $\psi_{\triangledown}^\ast(f)\in\wlv$.
This finishes the proof of \eqref{emeq3} and hence STEP 2.
Thus, we complete the proof of Theorem \ref{mthm1}.
\end{proof}

By Remark \ref{r-m-fun} and Theorem \ref{mthm1}, we obtain the following
conclusion.

\begin{cor}\label{1.29.x1}
Let $p(\cdot)$ and $N$ be as in Theorem \ref{mthm1} and $a\in(0,\fz)$.
Then $f\in\whv$ if and only if one of the following items holds true:
\begin{enumerate}
\item[{\rm(i)}] $f\in\cs'(\rn)$ and there exists a $\psi\in\cs(\rn)$ with $\int_{\rn}\psi(x)\,dx=1$
such that $\psi_{\triangledown,a}^*(f)\in\wlv$;

\item[{\rm(ii)}] $f\in\cs'(\rn)$ and $f^*_{N,\triangledown}\in\wlv$.
\end{enumerate}

Moreover, for any $f\in\whv$, it holds true that
$$\|f\|_{\whv}\sim\|f^*_{N,\triangledown}\|_{\wlv}
\sim\|\psi_{\triangledown,a}^*(f)\|_{\wlv}$$
with the implicit equivalent positive constants independent of $f$.
\end{cor}

\section{Atomic characterizations of $\whv$\label{s-atom}}
\hskip\parindent
In this section, we establish the atomic characterization of $\whv$.
We begin with recalling the notion of $(p(\cdot),q,s)$-atoms introduced
by Nakai and Sawano in \cite[Definition 1.4]{ns12}.

\begin{defn}\label{atd1}
Let $p(\cdot)\in\cp(\rn)$,
$q\in(1,\fz]$ and
\begin{equation}\label{4.1.x}
s\in\lf(\frac{n}{p_-}-n-1,\fz\r)\cap{\zz}_+.
\end{equation}
A measurable function $a$ on $\rn$
is called a \emph{$(p(\cdot),q,s)$-atom} if there exists a ball $B$ such that
\begin{enumerate}
\item[{\rm (i)}] $\supp a \st B$;

\item[{\rm (ii)}] $\|a\|_{L^q(\rn)}\le \frac{|B|^{1/q}}{\|\chi_B\|_{\lv}}$;

\item[{\rm (iii)}] $\int_{\mathbb R^n}a(x)x^\az\,dx=0$ for all $\az\in{\zz}_+^n$
with $|\az|\le s$.
\end{enumerate}
\end{defn}

We now introduce the notion of the variable weak atomic Hardy space.

\begin{defn}\label{atd2}
Let $p(\cdot)\in C^{\log}(\rn)$, $q\in(1,\fz]$ and $s$ be
as in \eqref{4.1.x}.
The \emph{variable weak atomic Hardy space} $\wha$ is
defined as the space of all $f\in\cs'(\rn)$ which can be decomposed as
\begin{equation}\label{4.1.y}
f=\sum_{i\in\zz}\sum_{j\in\nn}\lij\aij\quad \mbox{in}\quad \cs'(\rn),
\end{equation}
where $\{\aij\}_{i\in\zz,j\in\nn}$ is a sequence of
$(p(\cdot),q,s)$-atoms, associated with balls $\{\Bij\}_{i\in\zz,\ j\in\nn}$,
satisfying that there exists a positive constant $c\in(0,1]$ such that,
for all $x\in\rn$ and $i\in\zz$,
$\sum_{j\in\nn}\chi_{c\Bij}(x)\le A$ with $A$ being a positive constant
independent of $x$ and $i$ and, for all $i\in\zz$ and $j\in\nn$,
$\lij:=\wz A 2^i\|\chi_\Bij\|_{\lv}$ with $\wz A$ being a positive constant
independent of $i$ and $j$.

Moreover, for any $f\in\wha$, define
$$\|f\|_{\wha}:=\inf\lf[{\sup_{i\in\zz}{\lf\|\lf\{\sum_{j\in\nn}
\lf[\frac{\lij\chi_\Bij}
{\|\chi_\Bij\|_{\lv}}\r]^{\underline{p}}\r\}^{1/{\underline{p}}}\r\|_{\lv}}}\r],$$
where the infimum is taken over all decompositions of $f$ as above.
\end{defn}

From the definition of $\wha$ and Remark \ref{2.5.y},
we easily deduce the following conclusion,
the details being omitted.

\begin{rem}\label{atomlm3}
Let $f\in\wha$. Then there exists a sequence $\{\aij\}_{i\in\zz,j\in\nn}$ of
$(p(\cdot),q,s)$-atoms, associated with balls $\{\Bij\}_{i\in\zz,j\in\nn}$,
satisfying that, for all $i\in\zz$ and $x\in\rn$,
$\sum_{j\in\nn}\chi_{c\Bij}(x)\le A$ with
$c\in(0,1]$ and $A$ being positive constants independent of $i$ and $x$
such that $f$ admits a decomposition as in \eqref{4.1.y}
with $\lij:=\wz A2^i\|\chi_\Bij\|_{\lv}$ for all
$i\in\zz$ and $j\in\nn$, where $\wz A$ is a positive
constant independent of $i$ and $j$, and
\begin{align}\label{atom-x1}
\|f\|_{\wha}
\sim\sup_{i\in\zz}{\lf\|\lf\{\sum_{j\in\nn}
\lf[\frac{\lij\chi_\Bij}
{\|\chi_\Bij\|_{\lv}}\r]^{\underline{p}}\r\}^{1/{\underline{p}}}
\r\|_{\lv}}
\end{align}
with the implicit equivalent positive constants independent of $f$.
Moreover, by the fact that $\sum_{j\in\nn}\chi_{c\Bij}\le A$
for all $i\in\zz$, the definition of $\{\lz_{i,j}\}_{i\in\zz,j\in\nn}$
and Remark \ref{2.5.y}, we further conclude that
\begin{align}
\|f\|_{\wha}&\sim\sup_{i\in\zz}2^i\lf\|\lf(\sum_{j\in\nn}
\chi_\Bij\r)^{1/{\underline{p}}}\r\|_{\lv}
\sim\sup_{i\in\zz}2^i\lf\|\lf(\sum_{j\in\nn}
\chi_{c\Bij}\r)^{1/{\underline{p}}}\r\|_{\lv}\noz\\
&\sim\sup_{i\in\zz}2^i\lf\|\sum_{j\in\nn}
\chi_{c\Bij}\r\|_{\lv}
\sim\sup_{i\in\zz}2^i\lf\|\sum_{j\in\nn}
\chi_{\Bij}\r\|_{\lv},\noz
\end{align}
where the implicit equivalent positive constants are independent of $f$.
\end{rem}

The main result of this section is stated as follows.

\begin{thm}\label{atthm1}
Let $p(\cdot)\in C^{\log}(\rn)$, $q\in(\max\{p_+,1\},\fz]$ with $p_+$
as in \eqref{2.1x} and $s$ be as in
\eqref{4.1.x}.
Then $\whv=\wha$ with equivalent quasi-norms.
\end{thm}

To prove Theorem \ref{atthm1}, we need the following useful technical lemma,
which is a variant of \cite[Lemma 4.1]{s10} and can be proved via
combining \cite[Lemma 4.1]{s10} and Lemma \ref{mlm1},
the details being omitted.

\begin{lem}\label{atlm2}
Let $p(\cdot)\in C^{\log}(\rn)$, $r\in(0,\underline{p}]$ and
$q\in[1,\fz]\cap(p_+,\fz]$. Then there exists
a positive constant $C$ such that, for all sequences
$\{B_j\}_{j\in\nn}$ of balls, numbers $\{\lz_j\}_{j\in\nn}\subset\mathbb C$
and measurable functions $\{a_j\}_{j\in\nn}$ satisfying that,
for each $j\in\nn$, $\supp a_j\st B_j$ and $\|a_j\|_{L^q(\rn)}\le |B_j|^{1/q}$,
it holds true that
$$\lf\|\lf(\sum_{j\in\nn}|\lz_j a_j|^{r}\r)^{\frac{1}{r}}
\r\|_{\lv}\le C\lf\|\lf(\sum_{j\in\nn}|\lz_j\chi_{B_j}|^{r}\r)
^{\frac{1}{r}}\r\|_{\lv}.$$
\end{lem}

In what follows, we use $\vec0_n$ to denote the origin of $\rn$
and, for any $\varphi\in\cs(\rn)$,
we use $\widehat{\varphi}$ to denote its
\emph{Fourier transform},
which is defined by setting, for all $\xi\in\rn$,
$$\widehat{\varphi}(\xi):=\int_{\rn}e^{-2\pi ix\xi}\varphi(x)\,dx.$$

We also need the following Calder\'on reproducing formula,
which was obtained by Calder\'on \cite[p.\,219]{c77}
(see also \cite[Lemma 4.1]{ct75}).

\begin{lem}\label{l-1.24x}
Let $\psi\in\cs(\rn)$ be such that $\supp \psi\st B(\vec0_n,1)$
and $\int_\rn\psi(x)\,dx=0$.
Then there exists a
function $\phi\in\cs(\rn)$ such that $\widehat{\phi}$
has compact support away from the origin and,
for all $x\in\rn\setminus\{\vec0_n\}$,
$$\int_0^\fz \widehat{\psi}(tx)\widehat{\phi}(tx)\,\frac{dt}t=1.$$
\end{lem}

Recall that, for any $d\in\zz_+$, $p(\cdot)\in\cp(\rn)$,
a locally integrable function $f$ on $\rn$ is
said to belong to the \emph{Campanato space}
$\cl_{1,p(\cdot),d}(\rn)$ if
$$\|f\|_{\cl_{1,p(\cdot),d}(\rn)}
:=\sup_{Q\subset\rn}\frac{1}{\|\chi_Q\|_{\lv}}
\int_{Q}|f(x)-P_Q^df(x)|\,dx<\fz,$$
where the supremum is taken over all cubes $Q$ of $\rn$ and
$P_Q^d$ denotes the \emph{unique polynomial} $P$ having degree at most $d$
and satisfies that, for any polynomial $R$ on $\rn$ with order at most
$d$, $\int_Q[f(x)-P(x)]R(x)\,dx=0$ (see \cite[Definition 6.1]{ns12}).

Now we turn to the proof of Theorem \ref{atthm1}.
For $N\in(\frac{n}{\underline{p}}+n+1,\fz)$
and $h\in\cs'(\rn)$, we denote $h_{N,+}^\ast$ simply by $h^\ast$.

\begin{proof}[Proof of Theorem \ref{atthm1}]
{\bf STEP 1:} In this step, we show that $\wha\st\whv$.

Let $f\in\wha$. Then, by Remark \ref{atomlm3}, we know that there exist
a sequence $\{\aij\}_{i\in\zz,j\in\nn}$ of $(p(\cdot),q,s)$-atoms,
associated with balls $\{\Bij\}_{i\in\zz,j\in\nn}$,
and $\{\lz_{i,j}\}_{i\in\zz,j\in\nn}\st\cc$ such that
\eqref{4.1.y} holds true in $\cs'(\rn)$ and
\begin{align}\label{atom-c1}
\|f\|_{\wha}\sim\sup_{i\in\zz}2^i\lf\|\sum_{j\in\nn}\chi_\Bij\r\|_{\lv}.
\end{align}
To prove $f\in\whv$, by the definition of $\whv$, it suffices to show that
$$\sup_{\az\in(0,\fz)}\az\lf\|\chi_{\{x\in\rn:\ |f^*(x)|>\az\}}
\r\|_{\lv}\ls\|f\|_{\wha}.$$
For any given $\az\in(0,\fz)$, we choose $i_0\in\zz$ such that
$2^{i_0}\le\az<2^{i_0+1}$ and write
$$f=\sum_{i=-\fz}^{i_0-1}\sum_{j\in\nn}\lij\aij
+\sum_{i=i_0}^{\fz}\sum_{j\in\nn}\lij\aij=:f_1+f_2.$$
Thus, by Remark \ref{r-vlp}(i), we find that
\begin{align}\label{eqatomsum}
&\lf\|\chi_{\{x\in\rn:\ f^*(x)>\az\}}\r\|_{\lv}\noz\\
&\hs\ls\lf\|\chi_{\{x\in\rn:\ f_1^*(x)
>\frac{\az}{2}\}}\r\|_{\lv}
+\lf\|\chi_{\{x\in A_{i_0}:\ f_2^*(x)
>\frac{\az}{2}\}}\r\|_{\lv}\noz\\
&\hs\hs\hs\hs+\lf\|\chi_{\{x\in(A_{i_0})^{\com}:\ f_2^*(x)
>\frac{\az}{2}\}}\r\|_{\lv}\noz\\
&\hs=:{\rm I}_1+{\rm I}_2+{\rm I}_3,
\end{align}
where $A_{i_0}:=\bigcup_{i={i_0}}^\fz\bigcup_{j\in\nn}(2\Bij)$.

For ${\rm I}_1$, it is easy to see that
\begin{align}\label{eqatomi1}
{\rm I}_1&\ls\lf\|\chi_{\{x\in\rn:\ \sum_{i=-\fz}
^{i_0-1}\sum_{j\in\nn}\lij(\aij)^{\ast}(x)\chi_{2\Bij}(x)
>\frac{\az}{4}\}}\r\|_{\lv}\noz\\
&\hs\hs\hs+\lf\|\chi_{\{x\in\rn:\ \sum_{i=-\fz}^{i_0-1}
\sum_{j\in\nn}\lij(\aij)^{\ast}(x)\chi_{(2\Bij)^\com}(x)
>\frac{\az}{4}\}}\r\|_{\lv}\noz\\
&=:{\rm I}_{1,1}+{\rm I}_{1,2}.
\end{align}

To estimate ${\rm I}_{1,1}$, for any $b\in(0,\underline{p})$,
let $\q1\in(1,\min\{\frac{q}{\max\{p_+,1\}},\frac{1}{b}\})$ and
$a\in(0,1-\frac1{\q1})$.
Then, by the H\"older inequality, we find that, for all $x\in\rn$,
\begin{align*}
&\sum_{i=-\fz}^{i_0-1}\sum_{j\in\nn}\lij(\aij)
^{\ast}(x)\chi_{2\Bij}(x)\\
&\hs\le \lf(\sum_{i=-\fz}^{i_0-1}2^{ia{\q1}'}\r)
^{1/{\q1}'}\lf\{\sum_{i=-\fz}^{i_0-1}2^{-ia{\q1}}
\lf[\sum_{j\in\nn}\lij(\aij)^{\ast}(x)\chi_{2\Bij}(x)\r]
^{\q1}\r\}^{1/{\q1}}\\
&\hs=\frac{2^{i_0a}}{(2^{a{\q1}'}-1)^{1/{\q1}'}}
\lf\{\sum_{i=-\fz}^{i_0-1}2^{-ia{\q1}}
\lf[\sum_{j\in\nn}\lij(\aij)^{\ast}(x)\chi_{2\Bij}(x)\r]
^{\q1}\r\}^{1/{\q1}},
\end{align*}
where ${\q1}'$ denotes the conjugate exponent of $\q1$, namely,
$\frac{1}{\q1}+\frac{1}{{\q1}'}=1$.
From this, the facts that $\q1b<1$ and $f^*(x)\ls\cm f(x)$
for all $x\in\rn$, Remark \ref{r-vlp}(i) and \cite[Theorem 2.61]{cfbook},
we deduce that
\begin{align*}
{\rm I}_{1,1}
&\le\lf\|\chi_{\{x\in\rn:\ \frac{2^{i_0a}}
{(2^{a{\q1}'}-1)^{1/{\q1}'}}
[\sum_{i=-\fz}^{i_0-1}2^{-ia{\q1}}
\{\sum_{j\in\nn}\lij(\aij)^{\ast}(x)
\chi_{2\Bij}(x)\}^{\q1}]^{1/{\q1}}>
2^{i_0-2}\}}\r\|_{\lv}\\
&\ls 2^{-i_0\q1(1-a)}\lf
\|\sum_{i=-\fz}^{i_0-1}2^{-ia{\q1}}\lf
[\sum_{j\in\nn}\lij(\aij)^{\ast}\chi_{2\Bij}\r]
^{\q1}\r\|_{\lv}\\
&\ls 2^{-i_0\q1(1-a)}\lf\|\sum_{i=-\fz}^{i_0-1}
2^{(1-a)i{\q1}b}\sum_{j\in\nn}\lf
[\|\chi_{\Bij}\|_{\lv}\cm(\aij)\chi_{2\Bij}\r]
^{\q1b}\r\|
^{\frac{1}{b}}_{L^{\frac{p(\cdot)}
b}(\rn)}\\
&\ls 2^{-i_0\q1(1-a)}\Bigg[\sum_{i=-\fz}
^{i_0-1}2^{(1-a)i{\q1}b}\\
&\quad\quad\quad\times\lf.\lf\|\lf\{\sum_{j\in\nn}
\lf[\|\chi_{\Bij}\|_{\lv}\cm(\aij)\chi_{2\Bij}\r]^
{\q1b}\r\}^{\frac{1}{b}}\r\|^{b}_{\lv}\r]^{\frac{1}{b}}.
\end{align*}
Now let $r:=\frac{q}{\q1}$. Then $r\in(1,\fz)$ and, by
the boundedness of $\cm$ on $L^r(\rn)$, we find that,
for all $i\in\zz$ and $j\in\nn$,
\begin{align*}
\lf\|\lf[\|\chi_{\Bij}\|_{\lv}\cm(\aij)\chi_{2\Bij}\r]^{\q1}\r\|_{L^r(\rn)}
&\ls\|\chi_{\Bij}\|_{\lv}^{\q1}\lf\|\cm(\aij)\chi_{2\Bij}\r\|^{\q1}_{L^q(\rn)}\noz\\
&\ls |\Bij|^{\frac{1}{r}}.
\end{align*}
Therefore, by Lemma \ref{atlm2}, Remark \ref{2.5.y} and \eqref{atom-c1},
we conclude that
\begin{align*}
{\rm I}_{1,1}&\ls 2^{-i_0\q1(1-a)}\lf
[\sum_{i=-\fz}^{i_0-1}2^{(1-a)i{\q1}b}
\lf\|\lf(\sum_{j\in\nn}\chi_{2\Bij}\r)
^{\frac{1}{b}}\r\|
^{b}_{\lv}\r]^
{\frac{1}{b}}\\
&\ls 2^{-i_0\q1(1-a)}\lf
[\sum_{i=-\fz}^{i_0-1}2^{(1-a)i{\q1}b}
\lf\|\lf(\sum_{j\in\nn}\chi_{c\Bij}\r)
^{\frac{1}{b}}\r\|
^{b}_{\lv}\r]^
{\frac{1}{b}}\\
&\ls 2^{-i_0\q1(1-a)}
\lf\{\sum_{i=-\fz}^{i_0-1}2^{[(1-a){\q1}-1]ib}
\r\}^{\frac{1}{b}}\sup_{i\in\zz}2^i\lf\|
\sum_{j\in\nn}\chi_{\Bij}\r\|_{\lv}
\ls\az^{-1}\|f\|_{\wha},
\end{align*}
which implies that
\begin{align}\label{eqatomi11}
\az{\rm I}_{1,1}\ls\|f\|_{\wha}.
\end{align}

To deal with ${\rm I}_{1,2}$, we need some estimates on $(\aij)^*$.
Let $\phi\in\cf_N(\rn)$ and, for any $i\in\zz$ and $j\in\nn$,
$B_{i,j}:=B(x_{i,j},r_{i,j})$ with some $x_{i,j}\in\rn$
and $r_{i,j}\in(0,\fz)$.
Then, from the vanishing moment condition of $\aij$, the Taylor remainder theorem and
the H\"older inequality,
we deduce that, for any $i\in\zz\cap[i_0,\fz)$, $j\in\nn$, $t\in(0,\fz)$ and
$x\in (2\Bij)^\com$,
\begin{align}\label{11.14.x2}
\hs\hs\lf|\aij * \phi_t(x)\r|
&=\lf|\int_\Bij \aij(y)
\lf[\phi\lf(\frac{x-y}{t}\r)
-\sum_{|\beta|\le s}\frac{D^\beta\phi
(\frac{x-\xij}{t})}{\beta!}
\lf(\frac{\xij-y}{t}\r)^\beta\r]\,\frac{dy}{t^n}\r|\noz\\
&\ls\int_\Bij|\aij(y)|\frac{|y-\xij|^{s+1}}
{|x-\xij|^{n+s+1}}\,dy\noz\\
&\ls\frac{(r_{i,j})^{s+1}}{|x-\xij|^{n+s+1}}
\lf[\int_\Bij|\aij(y)|^q\,dy\r]^{1/q}
\lf(\int_\Bij\,1\,dy\r)^{1/q'}\noz\\
&\ls\|\chi_{\Bij}\|_{\lv}^{-1}
\lf(\frac{\rij}{|x-\xij|}\r)^{n+s+1},
\end{align}
which implies that, for any $x\in (2\Bij)^\com$,
\begin{align}\label{eatomst}
(\aij)^*(x)&\ls\|\chi_{\Bij}\|_{\lv}^{-1}
\lf(\frac{\rij}{|x-\xij|}\r)^{n+s+1}\noz\\
&\ls\|\chi_{\Bij}\|_{\lv}^{-1}
\lf[\cm\lf(\chi_{\Bij}\r)(x)\r]^{\frac{n+s+1}{n}}.
\end{align}
By this, the H\"older inequality, Remark \ref{r-vlp}(i), Lemma \ref{mlm1} and
\eqref{atom-c1}, we find that, for any $b\in(0,\frac{n}{n+s+1})$,
$\Q1\in(\frac{n}{[n+s+1]b}, \frac{1}{b})$ and $a\in(0,1-\frac1{\Q1})$,
\begin{align*}
{\rm I}_{1,2}
&\le \lf\|
\chi_{\{x\in\rn:\ \frac{2^{i_0a}}{(2^{a{\Q1}'}-1)^{1/{\Q1}'}}
\{\sum_{i=-\fz}^{i_0-1}2^{-ia{\Q1}}
[\sum_{j\in\nn}\lij(\aij)^{\ast}(x)
\chi_{({2\Bij})^\com}(x)]^{\Q1}\}^
{1/{\Q1}}>2^{i_0-2}\}}\r\|_{\lv}\\
&\ls 2^{-i_0\Q1(1-a)}\lf
\|\sum_{i=-\fz}^{i_0-1}2^{-ia{\Q1}}\lf
[\sum_{j\in\nn}\lij(\aij)^{\ast}
\chi_{({2\Bij})^\com}\r]^{\Q1}\r\|_{\lv}\\
&\ls 2^{-i_0\Q1(1-a)}\lf\{
\sum_{i=-\fz}^{i_0-1}2^{(1-a)i{\Q1}b}
\lf\|\sum_{j\in\nn}\lf[\cm\lf(\chi_{\Bij}\r)\r]^
{\frac{(n+s+1)\Q1 b}{n}}\r\|_{L^{\frac{p(\cdot)}
{b}}(\rn)}\r\}^{\frac{1}{b}}\\
&\ls 2^{-i_0\Q1(1-a)}\lf
[\sum_{i=-\fz}^{i_0-1}2^{(1-a)i{\Q1}b}
\lf\|\sum_{j\in\nn}\chi_{\Bij}\r\|_
{L^{\frac{p(\cdot)}{b}}(\rn)}\r]^{\frac{1}{b}}\\
&\ls 2^{-i_0\Q1(1-a)}\lf\{\sum_{i=-\fz}
^{i_0-1}2^{[(1-a){\Q1}-1]ib}
2^{i{b}}\lf\|\lf(\sum_{j\in\nn}
\chi_{c\Bij}\r)^{\frac{1}{b}}\r\|_
{\lv}^{b}\r\}^{\frac{1}{b}}\\
&\ls\az^{-1}\sup_{i\in\zz}2^i\lf\|\sum_{j\in\nn}
\chi_{\Bij}\r\|_{\lv}\sim\az^{-1}\|f\|_{\wha},
\end{align*}
that is,
\begin{align}\label{eqatomi12}
\az{\rm I}_{1,2}\ls\|f\|_{\wha}.
\end{align}
By this, combined with \eqref{eqatomi1} and \eqref{eqatomi11}, we conclude that
\begin{align}\label{eqatomi}
\az{\rm I}_1\ls\|f\|_{\wha}.
\end{align}

For ${\rm I}_2$, we choose $r_1\in[\frac{1}{\underline{p}},\fz)$.
Then, by Remark \ref{2.5.y} and \eqref{atom-c1}, we conclude that
\begin{align*}
{\rm I}_2&\le\|\chi_{A_{i_0}}\|_{L^{p(\cdot)}(\rn)}
\le\lf\|\sum_{i={i_0}}^\fz\sum_{j\in\nn}
\chi_{2\Bij}\r\|_{L^{p(\cdot)}(\rn)}
\ls\lf\|\sum_{i={i_0}}^\fz\sum_{j\in\nn}
\chi_{\Bij}\r\|_{L^{p(\cdot)}(\rn)}\\
&\ls\lf[\sum_{i={i_0}}^\fz\lf\|\sum_{j\in\nn}\chi_{\Bij}
\r\|_{L^{p(\cdot)}(\rn)}^{\frac{1}{r_1}}\r]^{r_1}
\sim\lf\{\sum_{i={i_0}}^\fz2^{-\frac{i}{r_1}}
\lf[2^i\lf\|\sum_{j\in\nn}\chi_{\Bij}
\r\|_{L^{p(\cdot)}(\rn)}\r]^{\frac{1}{r_1}}\r\}^{r_1}\\
&\ls\sup_{i\in\zz}2^i\lf\|\sum_{j\in\nn}\chi_\Bij
\r\|_{\lv}\lf(\sum_{i={i_0}}^\fz
2^{-\frac{i}{r_1}}\r)^{r_1}
\ls\az^{-1}\|f\|_{\wha},
\end{align*}
which implies that
\begin{align}\label{eqatomii}
\az {\rm I}_2\ls\|f\|_{\wha}.
\end{align}

For ${\rm I}_3$, since $\underline{p}\in(\frac{n}{n+s+1},1]$, it follows that
there exists $r_2\in(0,\fz)$ such that
$r_2\in(\frac{n}{\underline{p}(n+s+1)},1)$.
By this, the value of $\lz_{i,j}$, Lemma \ref{mlm1}, \eqref{eatomst} and \eqref{atom-c1}, we find that
\begin{align*}
{\rm I}_3&\le\lf\|\chi_{\{x\in (A_{i_0})
^\complement:\ \sum_{i=i_0}^\fz\sum_{j\in\nn}
\lij(\aij)^*(x)>\az/2\}}\r\|_{\lv}\\
&\ls\az^{-r_2}\lf\|\sum_{i=i_0}
^\fz\sum_{j\in\nn}[\lij(\aij)^*]^{r_2}\chi_{(A_{i_0})^\com}
\r\|_{\lv}\noz\\
&\ls\az^{-r_2}\lf[\sum_{i=i_0}^
\fz\lf\|\lf\{\sum_{j\in\nn}[\lij(\aij)^*]^{r_2}\chi_{(A_{i_0})^\com}\r\}^
{\frac{n}{r_2(n+s+1)}}\r\|_{L^{\frac{r_2(n+s+1)}{n}p(\cdot)}(\rn)}\r]
^{\frac{r_2(n+s+1)}{n}}\\
&\ls\az^{-r_2}\lf\{\sum_{i=i_0}^\fz
2^{\frac{in}{n+s+1}}\lf\|\sum_{j\in\nn}\lf[\cm\lf(\chi_{\Bij}\r)\r]
^{\frac{r_2(n+s+1)}{n}}
\r\|_{L^{p(\cdot)}(\rn)}^{{\frac{n}{r_2(n+s+1)}}}\r\}
^{\frac{r_2(n+s+1)}{n}}\\
&\ls\az^{-r_2}\lf[\sum_{i=i_0}^\fz2^{\frac{in}{n+s+1}}\lf
\|\sum_{j\in\nn}\chi_{\Bij}\r\|
_{\lv}^\frac{n}{r_2(n+s+1)}\r]^{\frac{r_2(n+s+1)}{n}}\noz\\
&\ls\sup_{i\in\zz}2^i\lf\|\sum_{j\in\nn}\chi_\Bij
\r\|_{\lv}\az^{-r_2}\lf[\sum_{i=i_0}^\fz2^{\frac{in}{n+s+1}}
2^{-\frac{in}{(n+s+1)r_2}}\r]
^{\frac{r_2(n+s+1)}{n}}\noz\\
&\ls\az^{-1}\|f\|_{\wha},
\end{align*}
namely,
\begin{align}\label{eqatomiii}
\az {\rm I}_3\ls\|f\|_{\wha}.
\end{align}

Combining with \eqref{eqatomsum}, \eqref{eqatomi}, \eqref{eqatomii} and
\eqref{eqatomiii},
we obtain
\begin{align*}
\|f\|_{\whv}
&=\sup_{\az\in(0,\fz)}
\az\lf\|\chi_{\{x\in \rn:\ f^*(x)>\az\}}\r\|_{\lv}\\
&\ls\sup_{\az\in(0,\fz)}\az({\rm I}_1
+{\rm I}_2+{\rm I}_3)
\ls\|f\|_{\wha},
\end{align*}
which implies that $f\in\whv$. This finishes the proof of STEP 1.

{\bf STEP 2:} In this step, we prove that $\whv\st\wha$.

To complete the proof of STEP 2, it suffices to show $\whv\st\whz$,
since, due to the obvious fact that each $(p(\cdot),\fz,s)$-atom
is also a $(p(\cdot),q,s)$-atom for any $q\in(1,\fz)$, $\whz\st\wha$.

Let $\psi\in\cs(\rn)$ be such that $\supp \psi\st B(\vec0_n,1)$,
$\int_\rn\psi(x)x^\gz\,dx=0$ for all $\gz\in\zz_+^n$ with
$|\gz|\le s$.
Then, by Lemma \ref{l-1.24x}, we know that
there exists $\phi\in\cs(\rn)$ such that $\supp\wh \phi$ has compact support
away from the origin and,
for all $x\in\rn\setminus\{\vec0_n\}$,
$$\int_0^\fz\wh\psi(tx)\wh\phi(tx)\dt=1.$$
Define a function $\eta$ on $\rn$ by setting, for all $x\in\rn\setminus\{\vec0_n\}$,
$$\wh\eta(x):=\int_1^\fz\wh\psi(tx)\wh\phi(tx)\dt$$
and $\wh\eta(\vec0_n):=1$. Then, by \cite[p.\,219]{c77}, we find that $\eta$
is infinitely differentiable, has compact support and equals 1 near the origin.

Let $x_0:=(\overbrace{2,\,...\,,2}^{n\,{\rm times}})\in\rn$ and $f\in\whv$.
Following \cite{c77}, for all
$x\in\rn$ and $t\in(0,\fz)$, let $\wz \phi(x):=\phi(x-x_0)$,
$\wz \psi(x):=\psi(x+x_0)$,
$$
F(x,t):=f*\wz\phi_t(x)\quad \mbox{and}\quad G(x,t):=f*\eta_t(x).
$$
Then, by \cite[p.\,220]{c77}, we have
$$f(\cdot)=\int_0^{\fz}\int_{\rn}F(y,t)\wz \psi(\cdot-y)\frac{dy\,dt}{t}\quad  \mbox{in}\quad \cs'(\rn).$$
For all $x\in\rn$, let
$$M_{\triangledown} (f)(x):=\sup_{t\in(0,\fz),\,|y-x|\le
3(|x_0|+1)t}[|F(y,t)|+|G(y,t)|].$$
Then $M_{\triangledown} (f)$ is lower semi-continuous and,
due to Corollary \ref{1.29.x1}, $M_{\triangledown} (f)\in\wlv$;
moreover,
\begin{align}\label{maxeq}
\|M_{\triangledown} (f)\|_{\wlv}\sim\|f\|_{\whv}.
\end{align}
For all $i\in\zz$, let $\boz_i:=\{x\in\rn:\ M_{\triangledown} f(x)>2^i\}$.
Then $\boz_i$ is open. By \eqref{wvlp} and \eqref{maxeq}, we further find that
\begin{align}\label{9.16.x}
\sup_{i\in\zz}2^i\|\chi_{\boz_i}\|_{\lv}
\le\|M_{\triangledown} (f)\|_{\wlv}\ls\|f\|_{\whv}.
\end{align}
Since $\boz_i$ is a proper open subset of $\rn$,
by the Whitney decomposition (see, for example, \cite[p.\,463]{g09-1}),
we know that there exists a sequence $\{\qij\}_{j\in\nn}$ of cubes
such that, for all $i\in\zz$,
\begin{enumerate}
\item[{\rm (i)}] $\bigcup_{j\in\nn} \qij = \boz_i$ and
$\{\qij\}_{j\in\nn}$ have disjoint interiors;

\item[{\rm (ii)}] for all $j\in\nn$, $\sqrt{n}\,l_{\qij}
\le \dist(\qij,\, \boz_i^{\complement}) \le 4 \sqrt{n}\,l_{\qij }$,
where $l_{\qij }$ denotes the length of the cube $\qij$ and
$\dist(\qij,\, \boz_i^{\complement}):=\inf\{|x-y|:\ x\in\qij,\ y\in \boz_i^{\complement}\}$;

\item[{\rm (iii)}] for any $j,\, k\in\nn$, if the boundaries of
two cubes $\qij$ and $Q_{i,k}$ touch, then
$\frac14\le\frac{l_{\qij}}{ l_{Q_{i,k}}}\le 4;$

\item[{\rm (iv)}] for a given $j\in\nn$, there exist at most
$12 n$ different cubes $Q_{i,k}$ that touch $\qij$.
\end{enumerate}

Now, for any $\epsilon\in(0,\fz)$, $i\in\zz$, $j\in\nn$ and $x\in\rn$, let
$$\dist\lf(x,\boz_i^\com\r):=\inf\lf\{|x-y|:\ y\in \boz_i^\com\r\},$$
$$\wz\boz_i:=\lf\{(x,t)\in\rr^{n+1}_+:\ 0<2t(|x_0|+1)<\dist(x, \boz_i^\com)\r\},$$
$$\wqij:=\lf\{(x,t)\in\rr_+^{n+1}
:\ x\in\qij,\ (x,t)\in\wz\boz_i\bh\wz\boz_{i+1}\r\}$$
and
$$\bije(x):=\int_{\epsilon}^{\fz}\int_\rn\chi_{\wqij}(y,t) F(y,t)
\wz\psi_t(x-y)\frac{dy\,dt}{t}.$$
Then, by an argument similar to that used in \cite[pp.\,221-222]{c77}
(see also \cite[p.\,650]{lyj}),
we conclude that there exist positive constants $C_1$ and $C_2$
such that, for all $\ez\in(0,\fz)$, $i\in\zz$ and $j\in\nn$,
$\supp \bije\st C_1\qij$, $\|\bije\|_{L^\fz(\rn)}\le C_22^i$,
$\int_\rn\bije(x)x^\gz\,dx=0$ for all $\gz\in\zz_+^n$ satisfying
$|\gz|\le s$ and
$$f=\lim_{\ez\to0}\sum_{i\in\zz}\sum_{j\in\nn}\bije\ \mbox{ in }\ \cs'(\rn).$$
Moreover, by the argument used in \cite[p.\,650]{lyj}, we find that
there exist $\{\bij\}_{i\in\zz,j\in\nn}\st L^{\fz}(\rn)$ and
a sequence $\{\ez_k\}_{k\in\nn}\st(0,\fz)$ such that
$\ez_k\to 0$ as $k\to \fz$ and,
for any $i\in\zz$, $j\in\nn$ and $g\in L^1(\rn)$,
\begin{equation}\label{10.7.x4}
\lim_{k\to\fz}\la b_{i,j}^{\ez_k},g\ra = \la \bij,g\ra,
\end{equation}
$\supp \bij\st C_1\qij$, $\|\bij\|_{L^{\fz}(\rn)}\le  C_22^i$ and,
for all $\gz\in\zz_+^n$ with $|\gz|\le s$,
$$\int_\rn\bij(x)x^\gz\,dx
=\la b_{i,j}, x^\gz\chi_{C_1\qij}\ra
=\lim_{k\to\fz}\int_\rn b_{i,j}^{\ez_k}(x)x^\gz\,dx=0.$$

Next we show that
\begin{align}\label{10.7.x1}
\lim_{k\to\fz}\sum_{i\in\zz}\sum_{j\in\nn} b_{i,j}^{\ez_k}
=\sum_{i\in\zz}\sum_{j\in\nn} b_{i,j} \ \mbox{ in }\ \cs'(\rn).
\end{align}
Since $\|b_{i,j}^{\ez_k}\|_{L^\fz(\rn)}\ls2^i$,
$\|\bij\|_{L^\fz(\rn)}\ls2^i$ and, for all $\gz\in\zz_+^n$ with $|\gz|\le s$,
$$\int_\rn b_{i,j}^{\ez_k}(x)x^\gz\,dx=0=\int_\rn\bij(x)x^\gz\,dx,$$
it follows, from Remarks \ref{r-vlp}(iv) and \ref{r-hlb} and \eqref{9.16.x},
that, for all $\zeta\in\cs(\rn)$ and $k,\ N\in\nn$,
\begin{align*}
&\sum_{|i|>N}\sum_{j\in\nn}
[|\la b_{i,j}^{\ez_k},\zeta\ra|+|\la\bij,\zeta\ra|]\\
&\hs=\sum_{i=-\fz}^{-N-1}\sum_{j\in\nn}
[|\la b_{i,j}^{\ez_k},\zeta\ra|+|\la\bij,\zeta\ra|]
+\sum_{i=N+1}^{\fz}\sum_{j\in\nn}\lf\{
\lf|\int_{C_1\qij}b_{i,j}^{\ez_k}(x)
[\zeta(x)-P_{C_1\qij}^s\zeta(x)]\,dx\r|\r.\\
&\quad\quad+\lf.\lf|\int_{C_1\qij}\bij(x)
[\zeta(x)-P_{C_1\qij}^s\zeta(x)]\,dx\r|\r\}\\
&\hs\ls\sum_{i=-\fz}^{-N-1}2^i\int_{\rn}|\zeta(x)|\,dx
+\sum_{i=N+1}^{\fz}\sum_{j\in\nn}2^i
\int_{C_1\qij}|\zeta(x)-P_{C_1\qij}^s\zeta(x)|\,dx\\
&\hs\ls2^{-N}\|\zeta\|_{L^1(\rn)}+\sum_{i=N+1}^{\fz}\sum_{j\in\nn}
2^i\|\chi_{\qij}\|_{L^{\frac{p(\cdot)}{r}}(\rn)}
\|\zeta\|_{\cl_{1,\frac{p(\cdot)}{r},s}(\rn)}\\
&\hs\ls2^{-N}\|\zeta\|_{L^1(\rn)}
+\|\zeta\|_{\cl_{1,\frac{p(\cdot)}{r},s}(\rn)}
\sum_{i=N+1}^{\fz}2^i\|\chi_{\boz_i}\|_{L^{\frac{p(\cdot)}{r}}(\rn)}\\
&\hs\ls2^{-N}\|\zeta\|_{L^1(\rn)}
+\|\zeta\|_{\cl_{1,\frac{p(\cdot)}{r},s}(\rn)}
\lf[\sup_{i\in\zz}2^i\|\chi_{\boz_i}\|_{\lv}\r]
^{r}\sum_{i=N+1}^{\fz}2^{-i(r-1)}\\
&\hs\ls2^{-N}\|\zeta\|_{L^1(\rn)}
+2^{-N(r-1)}\|\zeta\|_{\cl_{1,\frac{p(\cdot)}{r},s}(\rn)}\|f\|_{\whv}^r,
\end{align*}
which tends to $0$ as $N\to\fz$,
where $r$ is chosen such that $r\in(\max\{p_+,1\},\fz)$ and, in the last inequality,
we used the fact that,  for any $\zeta\in\cs(\rn)$,
$\|\zeta\|_{\cl_{1,\frac{p(\cdot)}{r},s}(\rn)}$ is finite
(see \cite[Lemma 2.8]{zyl}). Similarly, we have
$$\sum_{|i|\le N}\sum_{j\in\nn}[|\la b_{i,j}
^{\ez_k},\zeta\ra|+|\la\bij,\zeta\ra|]<\fz.$$
Therefore, by the argument same as that used in \cite[p.\,651]{lyj},
we conclude that \eqref{10.7.x1} holds true.

For all $i\in\zz$ and $j\in\nn$,
let $B_{i,\,j}$ be the ball having the same center
as $Q_{i,\,j}$ with the radius $5\sqrt n C_1 l_{Q_{i,\,j}}$,
$$\aij:=\frac{\bij}{C_2 2^i \|\chi_{\Bij}\|_{\lv}}\quad\mbox{and}
\quad\lij:=C_2 2^i \|\chi_{\Bij}\|_{\lv}.$$
Then, from the properties of $\bij$, it follows that
$\aij$ is a $(p(\cdot),\fz,s)$-atom associated with the ball
$\Bij$ and, due to \eqref{10.7.x1},
$f=\sum_{i\in\zz}\sum_{j\in\nn}\lij\aij$ in $\cs'(\rn)$.
Moreover, by Definition \ref{atd2}, Remark \ref{2.5.y} and \eqref{9.16.x},
we find that
\begin{align*}
\|f\|_{\whz}&\ls\sup_{i\in\zz}2^i\lf\|\sum_{j\in\nn}\chi_\Bij\r\|_{\lv}
\ls\sup_{i\in\zz}2^i\lf\|\sum_{j\in\nn}
\chi_{\qij}\r\|_{\lv}\\
&\ls\sup_{i\in\zz}2^i
\|\chi_{\boz_i}\|_{\lv}\ls \|f\|_{\whv}.
\end{align*}
Thus, $f\in\whz$ and $\|f\|_\whz\ls \|f\|_\whv$. This shows
$$\whv\st\whz,$$
which completes the proof of Theorem \ref{atthm1}.
\end{proof}

\section{Molecular characterizations of $\whv$\label{s-mole}}
\hskip\parindent
  In this section, we establish the molecular characterization of $\whv$
and begin with the following definition of molecules.

\begin{defn}\label{mod1}
Let $p(\cdot)\in C^{\log}(\rn)$, $q\in (1,\fz]$, $s\in{\zz}_+$ and
$\epsilon\in (0, \fz)$. A measurable function $m$ is called a
$(p(\cdot), q, s,\epsilon)$-\emph{molecule} associated
with some ball $B\subset\rn$ if
\begin{enumerate}
\item[{\rm (i)}] for each $j\in\nn$, $\|m\|_{L^q(U_j(B))}\le 2^{-j\epsilon}
|U_j(B)|^{\frac{1}{q}}\|\chi_B\|_{\lv}^{-1}$,
where $U_0(B):=B$ and, for all $j\in\nn$,
$U_j(B):=(2^j B) \backslash (2^{j-1} B)$;

\item[{\rm (ii)}] $\int_{\rn}m(x)x^\beta dx=0$ for all
$\beta\in\zz_+^n$ with $|\beta|\leq s$.
\end{enumerate}
\end{defn}

\begin{defn}\label{mod2}
Let $p(\cdot)\in C^{\log}(\rn)$, $q\in(1,\fz]$,
$s\in(\frac{n}{p_-}-n-1,\fz)\cap{\zz}_+$ with $p_-$ as in \eqref{2.1x},
and $\epsilon\in (0, \fz)$.
The \emph{variable weak molecular Hardy space}
$\whm$ is defined as the space of all $f\in\cs'(\rn)$
which can be decomposed as  $f=\sum_{i\in\zz}\sum_{j\in\nn}\lij\mij$
in $\cs'(\rn)$, where $\{\mij\}_{i\in\zz,j\in\nn}$ is a sequence of
$(p(\cdot),q,s,\epsilon)$-molecules associated with balls
$\{\Bij\}_{i\in\zz,j\in\nn}$,
$\{\lij\}_{i\in\zz,j\in\nn}:=\{\wz A2^i\|\chi_{\Bij}\|_{\lv}\}_{i\in\zz,j\in\nn}$ with
$\wz A$ being a positive constant independent of $i,\ j$, and there exist
positive constants $A$ and $C$ such that, for all $i\in\zz$ and
$x\in\rn, \sum_{j\in\nn}\chi_{C\Bij}(x)\le A$.

Moreover, for any $f\in\whm$, define
$$\|f\|_{\whm}:=\inf\lf[\sup_{i\in\zz}{\lf\|\lf\{\sum_{j\in\nn}
\lf[\frac{\lij\chi_\Bij}
{\|\chi_\Bij\|_{\lv}}\r]^{\underline{p}}\r\}^{1/{\underline{p}}}\r\|
_{\lv}}\r],\noz\\$$
where the infimum is taken over all decompositions of $f$ as above.
\end{defn}

\begin{thm}\label{mothm1}
Let $p(\cdot)\in C^{\log}(\rn)$, $q\in(\max\{p_+,1\},\fz]$,
$s\in(\frac{n}{p_-}-n-1,\fz)\cap{\zz}_+$ and
$\epz\in (n+s+1, \fz)$, where $p_+$ and $p_-$ are as in \eqref{2.1x}.
Then $\whv=\whm$ with equivalent quasi-norms.
\end{thm}

\begin{proof}
Notice that a $(p(\cdot), \fz,s)$-atom
 is also a $(p(\cdot), q, s,\epsilon)$-molecule.
Thus, we have
$$\whv\subset\whz\subset\whm.$$
Therefore, to prove this theorem, it suffices to show $\whm\subset\whv$.

Let $m$ be any fixed $(p(\cdot),q,s,\epsilon)$-molecule
associated with some ball $B:=B(x_B,r_B),$
where $x_B\in\rn$ and $r_B\in(0,\fz)$.
Then we claim that $m$ is an infinite linear combination of
$(p(\cdot),q,s)$-atoms.

To prove this, for all $k\in\zz_+$, let $m_k:=m\chi_{U_k(B)}$ with
$U_k(B)$ as in Definition \ref{mod1}(i),
and $\cp_k$ be the {linear vector space} generated by the set
$\{x^\gamma\chi_{U_k(B)}\}_{|\gamma|\le s}$ of ``polynomials". It is well known
(see, for example, \cite{tg80}) that, for any given $k\in\zz_+$,
there exists a unique {polynomial}
$P_k\in\cp_k$ such that, for all multi-indices $\beta$ with $|\beta|\le s$,
\begin{align}\label{11.9.x1}
\int_{\rn}x^\beta[m_k(x)-P_k(x)]\,dx=0,
\end{align}
where $P_k$ is given by
\begin{align}\label{11.9.x5}
P_k:=\sum_{\beta\in\zz_+^n,|\beta|\le s}
\lf[\frac{1}{|U_k(B)|}\int_{\rn}x^\beta m_k(x)\,dx\r]Q_{\beta,k}
\end{align}
and $Q_{\beta,k}$ is the unique polynomial in $\cp_k$ satisfying
that, for all multi-indices $\beta$ with $|\beta|\le s$ and the
\emph{Kronecker delta} $\delta_{\gamma,\beta}$,
\begin{align}\label{11.9.x3}
\int_{\rn}x^\gamma Q_{\beta,k}(x)\,dx=|U_k(B)|\delta_{\gamma,\beta},
\end{align}
where, when $\gamma=\beta$, $\delta_{\gamma,\beta}:=1$ and, when $\gamma\neq \beta$,
$\delta_{\gamma,\beta}:=0$.

Recall that it was proved in \cite[p.\,83]{tg80} that, for all $k\in\zz_+$,
\begin{align*}
\sup_{x\in U_k(B)}|P_k(x)|\ls \frac{1}{|U_k(B)|}\|m_k\|_{L^1(\rn)}.
\end{align*}
From this and the H\"older inequality, we deduce that, for all $k\in\zz_+$,
\begin{align}\label{11.9.x2}
\|m_k-P_k\|_{L^q(U_k(B))}
&\le\|m_k\|_{L^q(U_k(B))}+\|P_k\|_{L^q(U_k(B))}\le \wz C\|m_k\|_{L^q(U_k(B))}\noz\\
&\le \wz C2^{-k\epsilon}|2^kB|^{1/q}\|\chi_{B}\|_{\lv}^{-1},
\end{align}
where $\wz C$ is a positive constant independent of $m$, $B$ and $k$.
Obviously, for all $k\in\zz_+$, $\supp (m_k-P_k)\st U_k(B)$.
By this and \eqref{11.9.x1}, for any given $k\in\zz_+$, if we let
$$a_k:=\frac{2^{k\epsilon}\|\chi_{B}\|_{\lv}(m_k-P_k)}{\wz C\|\chi_{2^kB}\|_{\lv}},$$
we then conclude that $a_k$ is a $(p(\cdot),q,s)$-atom. Therefore,
\begin{align}\label{11.9.y2}
\sum_{k=0}^{\fz}(m_k-P_k)=\sum_{k=0}^{\fz}\mu_ka_k
\end{align}
is an infinite linear combination of $(p(\cdot),\fz,s)$-atoms,
where, for any $k\in\zz_+$,
$$\mu_k:=\wz C2^{-k\epsilon}\|\chi_{2^kB}\|_{\lv}/\|\chi_{B}\|_{\lv}.$$

Next we prove that $\sum_{k=0}^\fz P_k$ can be divided
into an infinite linear combination of $(p(\cdot),\fz,s)$-atoms.
For any $j\in\zz_+$ and $\ell\in\zz_+^n$, let
$$N_\ell^j:=\sum_{k=j}^\fz\int_{U_k(B)}m_k(x)x^\ell\,dx.$$
Then, for any $\ell\in\zz_+^n$ with $|\ell|\le s$, it holds true that
\begin{align}\label{11.9.x6}
N_\ell^0=\sum_{k=0}^\fz\int_{U_k(B)}m_k(x)x^\ell\,dx
=\int_{\rn}m(x)x^\ell\,dx=0.
\end{align}
From this and \eqref{11.9.x5}, we further deduce that
\begin{align}\label{11.9.x7}
\sum_{k=0}^{\fz}P_k&=\sum_{\ell\in\zz_+^n,|\ell|\le s}
\sum_{k=0}^{\fz}|U_k(B)|^{-1}Q_{\ell,k}\int_{\rn}m_k(x)x^\ell\,dx\noz\\
&=\sum_{\ell\in\zz_+^n,|\ell|\le s}\sum_{k=0}^{\fz}N_\ell^{k+1}
\lf[|U_{k+1}(B)|^{-1}Q_{\ell,k+1}\chi_{U_{k+1}(B)}(x)-|U_k(B)|^{-1}Q_{\ell,k}\chi_{U_k(B)}(x)\r]\noz\\
&=:\sum_{\ell\in\zz_+^n,|\ell|\le s}\sum_{k=0}^{\fz}b_{\ell}^k.
\end{align}
By an argument similar to that used in the proof of \cite[(4.35)]{hyy}, we conclude that,
for any $k\in\zz_+$ and $\ell\in\zz_+^n$ with $|\ell|\le s$,
\begin{equation}\label{5.7x}
\|b_\ell^k\|_{L^\fz(\rn)}\ls 2^{-k\epsilon}\|\chi_B\|_{\lv}^{-1}\quad
{\rm and}\quad \supp b_\ell^k\subset 2^{k+1}B;
\end{equation}
moreover, for all $\gamma\in\zz_+^n$ with $|\gamma|\le s$,
$\int_\rn b_\ell^k(x)x^\gamma\,dx=0$.
For all $k\in\zz_+$ and $\ell\in\zz_+^n$ with $|\ell|\le s$, let
$$\mu_\ell^k:=2^{-k\epsilon}\frac{\|\chi_{2^{k+1}B}\|_{\lv}}{\|\chi_{B}\|_{\lv}}
\quad{\rm and}\quad
a_\ell^k:=2^{k\epsilon}b_{\ell}^k\frac{\|\chi_{B}\|_{\lv}}{\|\chi_{2^{k+1}B}\|_{\lv}}.$$
Then, for any $k\in\zz_+$ and $\ell\in\zz_+^n$ with $|\ell|\le s$,
by \eqref{5.7x} and the definition of $a_\ell^k$, we find that $a_\ell^k$
is a $(p(\cdot),\fz,s)$-atom supported on $2^{k+1}B$ up to a positive constant multiple.
Thus,
\begin{equation}\label{1.29x}
\sum_{k=0}^{\fz}P_k=\sum_{\ell\in\zz_+^n,|\ell|\le s}\sum_{k=0}^{\fz}
\mu_\ell^ka_\ell^k
\end{equation}
forms an infinite linear combination of $(p(\cdot),\fz,s)$-atoms.

Combining \eqref{11.9.y2} and \eqref{1.29x}, we conclude that
\begin{align}\label{11.14x1}
m=\sum_{k=0}^{\fz}m_k=\sum_{k=0}^{\fz}(m_k-P_k)+\sum_{k=0}^{\fz}P_k=\sum_{k=0}^{\fz}\mu_ka_k+
\sum_{\ell\in\zz_+^n,|\ell|\le s}\sum_{k=0}^{\fz}\mu_\ell^ka_\ell^k.
\end{align}
This shows that a $(p(\cdot),q,s,\epsilon)$-molecule can be divided into an infinite
linear combination of $(p(\cdot),q,s)$-atoms. Therefore, the above claim holds true.

Now we prove $\whm\subset\whv$. Let $f\in\whm$. Then, by an argument
similar to that used in Remark \ref{atomlm3}, we know that there exist $\{\lij\}_{i\in\zz,j\in\nn}$ and a sequence $\{\mij\}_{i\in\zz,j\in\nn}$ of $(p(\cdot), q, s,\epsilon)$-molecules
associated with balls $\{\Bij\}_{i\in\zz,j\in\nn}$
such that, for all $i\in\zz$, $\sum_{j\in\nn}\chi_{c\Bij}\ls 1$ with
$c\in(0,1]$ being a positive constant independent of $i$, and
$f=\sum_{i\in\zz}\sum_{j\in\nn}\lij\mij$ in $\cs'(\rn)$, moreover,
\begin{align}\label{11.17.x1}
\|f\|_{\whm}\sim\sup_{i\in\zz}2^i\lf\|\sum_{j\in\nn}\chi_\Bij\r\|_{\lv}.
\end{align}
Next we prove that, for any $\az\in (0, \fz)$,
\begin{align}\label{10.11.z}
\az\lf\|\chi_{\{x\in\rn: f^*(x)>\az\}}\r\|
_{\lv}\ls\|f\|_{\whm},
\end{align}
where the implicit positive constant is independent of $\az$
and $f^*:=f^*_{N,+}$ with $N$ as in Theorem \ref{mthm1}.

For any $\az\in(0,\fz)$, let $i_0\in\zz$
such that $2^{i_0}\le\az<2^{i_0+1}$. Then we have
$$f=\sum_{i=-\fz}^{i_0-1}\sum_{j\in\nn}\lij\mij
+\sum_{i=i_0}^{\fz}\sum_{j\in\nn}\lij\mij=:f_1+f_2$$
and
\begin{align}\label{emodep1}
\lf\|\chi_{\{x\in\rn:\ f^*(x)>\az\}}\r\|_{\lv}
&\ls\lf\|\chi_{\{x\in\rn:\ f_1^*(x)>\frac{\az}{2}\}}\r\|_{\lv}
+\lf\|\chi_{\{x\in\rn:\ f_2^*(x)>\frac{\az}{2}\}}\r\|_{\lv}\noz\\
&=:{\rm I}_1+{\rm I}_2.
\end{align}

We first estimate ${\rm I}_1$.
To this end, we need another estimate for $(\mij)^*$.  From \eqref{11.14x1}, we deduce that,
for all $i\in\zz$ and $j\in\nn$, there exists a sequence of multiples
of $(p(\cdot),q,s)$-atoms, $\{a_{i,j}^l\}_{l\in\zz_+}$, associated with balls
$\{2^{l+1}\Bij\}_{l\in\zz_+}$ such that
$$\|a_{i,j}^l\|_{L^q(\rn)}\ls
\frac{2^{-l\epsilon}|2^{l+1}\Bij|^{1/q}}{\|\chi_{\Bij}\|_{\lv}}$$
and $\mij=\sum_{l\in\zz_+}a_{i,j}^l$ almost everywhere in $\rn$.
Then, for all $i\in\zz\cap(-\fz,i_0-1]$ and $j\in\nn$, we have
\begin{align}\label{10.11.x}
(\mij)^*\le\sum_{l\in\zz_+}(a_{i,j}^l)^*
=\sum_{l\in\zz_+}\sum_{k\in\zz_+}(a_{i,j}^l)^*
\chi_{U_k(2^l\Bij)}=:\sum_{l\in\zz_+}\sum_{k=0}^2{\rm J}_{l,k}
+\sum_{l\in\zz_+}\sum_{k=3}^{\fz}{\rm J}_{l,k},
\end{align}
where $U_k(2^l\Bij)$ is defined as in Definition \ref{mod1}(i) with $B$
replaced by $2^l\Bij$.
Thus, it follows that
\begin{align}\label{11.16.x3}
{\rm I}_1&=\lf\|\chi_{\{x\in\rn:\ f_1^*(x)>\frac{\az}{2}\}}\r\|_{\lv}
\le\lf\|\chi_{\{x\in\rn:\ \sum_{i=-\fz}^{i_0-1}\sum_{j\in\nn}\lij(\mij)^*(x)
>\frac{\az}{2}\}}\r\|_{\lv}\noz\\
&\ls\lf\|\chi_{\{x\in\rn:\ \sum_{i=-\fz}^{i_0-1}
\sum_{j\in\nn}\sum_{l\in\zz_+}
\sum_{k=0}^2\lij{\rm J}_{l,k}>\frac{\az}{4}\}}\r\|_{\lv}\noz\\
&\quad+\lf\|\chi_{\{x\in\rn:\ \sum_{i=-\fz}^{i_0-1}\sum_{j\in\nn}
\sum_{l\in\zz_+}\sum_{k=3}^{\fz}\lij{\rm J}_{l,k}
>\frac{\az}{4}\}}\r\|_{\lv}\noz\\
&=:{\rm I}_{1,1}+{\rm I}_{1,2}.
\end{align}

For ${\rm I}_{1,1}$, by an argument similar to that used in the proof of \eqref{eqatomi11},
we conclude that
\begin{align}\label{11.15.x1}
\az{\rm I}_{1,1}\ls\|f\|_{\whm}.
\end{align}

On the other hand, by an argument similar to that used in the proof of \eqref{11.14.x2},
we conclude that, for any $i\in\zz$, $j\in\nn$, $l\in\zz_+$, $k\in[3,\fz)\cap\zz_+$,
$x\in U_k(2^l\Bij)$ and $y\in2^{l+1}\Bij$,
\begin{align}\label{mostar}
{\rm J}_{l,k}
&\ls\frac{|y-\xij|^{s+1}}{|x-\xij|^{n+s+1}}
\int_{2^{l+1}\Bij} |a_{i,j}^l(y)|\,dy\chi_{U_k(2^l\Bij)}(x)\noz\\
&\ls\frac{|y-\xij|^{s+1}}{|x-\xij|^{n+s+1}}\|a_{i,j}^l\|
_{L^q(\rn)}|2^{l+1}\Bij|^{1/{q'}}\chi_{U_k(2^l\Bij)}(x)\noz\\
&\ls\frac{2^{-l(n+\epsilon)-k(n+s+1)}}{\rij^n\|\chi_{\Bij}\|_{\lv}}
|2^{l+1}\Bij|\chi_{U_k(2^l\Bij)}(x)\noz\\
&\ls\frac{2^{-l\epsilon-k(n+s+1)}}
{\|\chi_{\Bij}\|_{\lv}}\chi_{U_k(2^l\Bij)}(x),
\end{align}
which, combined with \eqref{10.11.x}, \eqref{11.17.x1}, Remark \ref{2.5.y},
the fact that $\epsilon\in(n+s+1,\fz)$ and via choosing
$r\in(\frac{n}{n+s+1},\underline p)$, implies that
\begin{align*}
\az{\rm I}_{1,2}
&\ls\az^{1-1/r}\lf\|\sum_{i=-\fz}^{i_0-1}\sum_{j\in\nn}\sum_{l\in\zz_+}
\sum_{k=3}^{\fz}2^i2^{-l\epsilon}2^{-k(n+s+1)}
\chi_{U_k(2^l\Bij)}\r\|_{L^{\frac{p(\cdot)}{r}}(\rn)}^{1/r}\\
&\ls\az^{1-1/r}\lf[\sum_{l\in\zz_+}\sum_{k=3}^{\fz}2^{-l\epsilon}2^{-k(n+s+1)}
\sum_{i=-\fz}^{i_0-1}2^i\lf\|\sum_{j\in\nn}
\chi_{U_k(2^l\Bij)}\r\|_{L^{\frac{p(\cdot)}{r}}(\rn)}\r]^{1/r}\\
&\ls\az^{1-1/r}\lf[\sum_{l\in\zz_+}
\sum_{k=3}^{\fz}2^{-l\epsilon}2^{-k(n+s+1)}2^{\frac{n(k+l)}{r}}
\sum_{i=-\fz}^{i_0-1}2^i\lf\|\sum_{j\in\nn}
\chi_{\Bij}\r\|_{L^{\frac{p(\cdot)}{r}}(\rn)}\r]^{1/r}\\
&\ls\az^{1-1/r}\lf[\sum_{i=-\fz}^{i_0-1}2^i\lf\|\sum_{j\in\nn}
\chi_{\Bij}\r\|_{L^{\frac{p(\cdot)}{r}}(\rn)}\r]^{1/r}\\
&\ls\az^{1-1/r}\sup_{i\in\zz}2^i\lf\|\sum_{j\in\nn}\chi_{\Bij}\r\|_{\lv}
\lf[\sum_{i=-\fz}^{i_0-1}2^{i(1-r)}\r]^{1/r}
\sim\|f\|_{\whm}.
\end{align*}
From this, \eqref{11.16.x3} and \eqref{11.15.x1}, we deduce that
\begin{align}\label{11.16.x4}
\az{\rm I}_1\ls\|f\|_{\whm}.
\end{align}

We next estimate ${\rm I_2}$. By \eqref{10.11.x}, we know that
\begin{align}\label{11.16.x1}
{\rm I}_2&\ls\lf\|\chi_{\{x\in\rn:\ \sum_{i=i_0}^{\fz}
\sum_{j\in\nn}\sum_{l\in\zz_+}
\sum_{k=0}^2\lij{\rm J}_{l,k}>\frac{\az}{4}\}}\r\|_{\lv}\noz\\
&\quad\hs+\lf\|\chi_{\{x\in\rn:\ \sum_{i=i_0}^{\fz}\sum_{j\in\nn}
\sum_{l\in\zz_+}\sum_{k=3}^{\fz}\lij{\rm J}_{l,k}
>\frac{\az}{4}\}}\r\|_{\lv}\noz\\
&=:{\rm I}_{2,1}+{\rm I}_{2,2}.
\end{align}

Let $\q1\in(1,\min\{\frac{q}{\max\{p_+,1\}},\frac{1}{b}\})$ and
$a\in(1-\frac1{\q1},\fz)$ for any $b\in(0,\underline{p})$,
where $p_+$ and $\underline{p}$ are as in \eqref{2.1x}, respectively, \eqref{2.1y}.
Then, with an argument similar to that used in the proof of \eqref{11.15.x1}, we obtain
\begin{align*}
{\rm I}_{2,1}&\ls 2^{-i_0\q1(1-a)}
\lf\{\sum_{i=i_0}^{\fz}2^{[(1-a){\q1}-1]ib}
\r\}^{\frac{1}{b}}\sup_{i\in\zz}2^i\lf\|
\lf(\sum_{j\in\nn}\chi_{\Bij}\r)
^{\frac{1}{b}}\r\|_{\lv}\\
&\ls\az^{-1}\|f\|_{\whm},
\end{align*}
which shows that
\begin{align}\label{11.16.x5}
\az{\rm I}_{2,1}\ls\|f\|_{\whm}.
\end{align}

On the other hand, let $a\in(\frac{1}{\underline p},\fz)$ and
$b\in(1-\frac{1}{a},\fz)$. By the H\"older inequality,
we find that, for all $x\in\rn$,
\begin{align}\label{10.11.y}
\sum_{i=i_0}^{\fz}\sum_{j\in\nn}
\sum_{l\in\zz_+}\sum_{k=3}^{\fz}\lij{\rm J}_{l,k}
&\le\lf(\sum_{i=i_0}^{\fz}2^{iba'}\r)^{1/a'}\lf
[\sum_{i=i_0}^{\fz}2^{-iba}\lf(\sum_{j\in\nn}
\sum_{l\in\zz_+}\sum_{k=3}^{\fz}\lij{\rm J}_{l,k}\r)^a\r]^{1/a}\noz\\
&=\frac{2^{i_0b}}{(2^{ba'}-1)^{1/a'}}\lf[\sum_{i=i_0}^{\fz}
2^{-iba}\lf(\sum_{j\in\nn}\sum_{l\in\zz_+}
\sum_{k=3}^{\fz}\lij{\rm J}_{l,k}\r)^a\r]^{1/a}.
\end{align}
Choose $d\in(\frac{n}{n+s+1},1)$. Then, by \eqref{10.11.y},
\eqref{mostar} and Remark \ref{2.5.y},
we find that
\begin{align*}
\hs\hs{\rm I}_{2,2}
&\le\lf\|\chi_{\{x\in\rn:\ \frac{2^{i_0b}}{(2^{ba'}-1)^{1/a'}}
\lf[\sum_{i=i_0}^{\fz}2^{-iba}\lf(\sum_{j\in\nn}\sum_{l\in\zz_+}
\sum_{k=3}^{\fz}\lij{\rm J}_{l,k}\r)^a\r]^{1/a}>2^{i_0-2}\}}\r\|_{\lv}\\
&\ls 2^{-i_0a(1-b)}\lf\|\sum_{i=i_0}^{\fz}2^{-iba}\lf(\sum_{j\in\nn}
\sum_{l\in\zz_+}\sum_{k=3}^{\fz}\lij{\rm J}_{l,k}\r)^a\r\|_{\lv}\\
&\ls2^{-i_0a(1-b)}\lf\|\sum_{i=i_0}^{\fz}2^{i(1-b)}
\sum_{l\in\zz_+}\sum_{k=3}^{\fz}2^{-l\epsilon}2^{-k(n+s+1)}
\sum_{j\in\nn}\chi_{U_k(2^l\Bij)}\r\|_{L^{ap(\cdot)(\rn)}}^a\\
&\ls2^{-i_0a(1-b)}\lf[\sum_{i=i_0}^{\fz}2^{i(1-b)}
\sum_{l\in\zz_+}\sum_{k=3}^{\fz}2^{-l\epsilon}2^{-k(n+s+1)}
\lf\|\sum_{j\in\nn}\chi_{U_k(2^l\Bij)}\r\|_{L^{ap(\cdot)(\rn)}}\r]^a\\
&\ls2^{-i_0a(1-b)}\lf[\sum_{i=i_0}^{\fz}2^{i(1-b)}
\lf\{\sum_{l\in\zz_+}\sum_{k=3}^{\fz}2^{-l\epsilon}2^{-k(n+s+1)}2^{\frac{n(k+l)}{d}}\r\}
\lf\|\sum_{j\in\nn}\chi_{\Bij}\r\|_{L^{ap(\cdot)(\rn)}}\r]^a\\
&\ls2^{-i_0a(1-b)}\lf[\sum_{i=i_0}^{\fz}2^{i(1-b)}
\lf\|\sum_{j\in\nn}\chi_{\Bij}\r\|_{L^{ap(\cdot)(\rn)}}\r]^a\\
&\ls2^{-i_0a(1-b)}\lf[\sum_{i=i_0}^{\fz}2^{i(1-b-\frac{1}{a})}\r]^a
\sup_{i\in\zz}2^i\lf\|\sum_{j\in\nn}\chi_{\Bij}\r\|_{\lv}\\
&\ls\az^{-1}\sup_{i\in\zz}2^i\lf\|\sum_{j\in\nn}\chi_{\Bij}\r\|_{\lv}
\sim\az^{-1}\|f\|_{\whm},
\end{align*}
which, combined with \eqref{11.16.x5} and \eqref{11.16.x1},
implies that $\az{\rm I}_2\ls\|f\|_{\whm}$.
This, together with \eqref{emodep1} and \eqref{11.16.x4},
shows that \eqref{10.11.z} holds true and hence
finishes the proof of Theorem \ref{mothm1}.
\end{proof}

\section{Littlewood-Paley function characterizations of $\whv$\label{s-palay}}
\hskip\parindent
In this section, we establish Littlewood-Paley
function characterizations of $\whv$ in Theorems
\ref{lpthm1}, \ref{lpthm2} and \ref{10.10.x}
as an application of the atomic characterization of $\whv$ in Theorem \ref{atthm1}.

Let $\phi\in\mathcal{S}(\mathbb R^n)$ be a radial function satisfying
\begin{align}\label{12.7.x1}
\supp\phi\subset\{x\in\mathbb R^n:\ |x|\le 1\},
\end{align}
\begin{align}\label{12.7.x2}
\int_{\mathbb R^n}\phi(x)x^\gz \,dx=0\ \ \mbox{for all}\ \
\gz\in\zz_+^n\ \ \mbox{with}\ \ |\gz|
\le\max\lf\{\lf\lfloor\frac{n}{p_-}-n-1\r\rfloor,0\r\}
\end{align}
and
\begin{align}\label{12.7.x3}
\int_0^\infty|\widehat{\phi}(\xi t)|^2\,\frac{dt}{t}=1\ \ \mbox{for all}
\ \ \xi\in\rn\setminus\{\vec0_n\}.
\end{align}
Here and hereafter, the \emph{symbol $\lfloor s \rfloor$} for any $s\in\rr$ denotes the maximal integer not larger than
$s$. Recall that, for all $f\in\mathcal{S}'(\mathbb R^n)$,
the \emph{Littlewood-Paley $g$-function},
the \emph{Lusin area function} and the \emph{Littlewood-Paley $g_\lz^*$-function}
of $f$ with $\lz\in(0,\fz)$ are defined,
respectively, by setting, for all $x\in\rn$,
\begin{align*}
 g(f)(x):=\lf[\int_0^\infty\lf|f\ast\phi_t(x)\r|^2\dt\r]^{1/2},
\end{align*}
\begin{align*}
 S(f)(x):=\lf[\int_{\Gamma(x)}
 |f\ast\phi_t(y)|^2\,\frac{\,dy\,dt}{t^{n+1}}\r]^{1/2}
\end{align*}
and
\begin{align*}
 g_\lz^*(f)(x):=\lf[\int_0^\infty\int_{\mathbb  R^n}
 \lf(\frac{t}{t+|x-y|}\r)^{\lz
 n}\lf|f\ast\phi_t(y)\r|^2\,\frac{\,dy\,dt}{t^{n+1}}\r]^{1/2},
\end{align*}
where, for any $x\in\rn$, $\Gamma(x):=\{(y,t)\in\rn\times(0,\fz):\ |y-x|<t\}$
and, for any $t\in(0,\fz)$, $\phi_t(\cdot):=\frac{1}{t^n}\phi(\frac{\cdot}{t})$.

Recall that $f\in\mathcal{S}'(\mathbb R^n)$ is said to
\emph{vanish weakly at infinity} if, for every $\phi\in\mathcal{S}(\mathbb R^n)$,
$f\ast\phi_t\to 0$ in $\mathcal{S}'(\mathbb R^n)$ as $t\to \infty$
(see, for example, \cite[p.\,50]{fs82}).

The main results of this section are stated as follows.

\begin{thm}\label{lpthm1}
Let $p(\cdot)\in C^{\log}(\rn)$.
Then $f\in\whv$ if and only if
$f\in\cs'(\rn)$, $f$ vanishes weakly at infinity and $S(f)\in\wlv$.
Moreover, for all $f\in\whv$,
$$C^{-1}\|S(f)\|_{\wlv}\le\|f\|_{\whv}\le C\|S(f)\|_{\wlv},$$
where $C$ is a positive constant independent of $f$.
\end{thm}

\begin{thm}\label{lpthm2}
Let $p(\cdot)\in C^{\log}(\rn)$.
Then $f\in\whv$ if and only if
$f\in\cs'(\rn)$, $f$ vanishes weakly at infinity
and $g(f)\in\wlv$. Moreover, for all $f\in\whv$,
$$C^{-1}\|g(f)\|_{\wlv}\le\|f\|_{\whv}\le C\|g(f)\|_{\wlv},$$
where $C$ is a positive constant independent of $f$.
\end{thm}

\begin{thm}\label{10.10.x}
Let $p(\cdot)\in C^{\log}(\rn)$ and $\lambda\in(1+\frac{2}{{\min\{p_-,2\}}}, \fz)$.
Then $f\in\whv$ if and only if $f\in\cs'(\rn)$, $f$ vanishes weakly at infinity and
$g_\lambda^{\ast}(f)\in\wlv$. Moreover, for all $f\in\whv$,
$$C^{-1}\|g_\lz^\ast(f)\|_{\wlv}\le\|f\|_{\whv}\le C\|g_\lz^\ast(f)\|_{\wlv},$$
where $C$ is a positive constant independent of $f$.
\end{thm}

\begin{rem}
In \cite[Theorem 4.13]{lyj}, Liang et al. established the $g_\lz^\ast$-function
characterization of the weak Hardy space $W\!H^p(\rn)$ with constant exponent
$p\in(0,1]$, as a special
case of the weak Musielak-Orlicz Hardy space $W\!H^\vz(\rn)$, with the best known range
for $\lz\in(2/p,\fz)$. However, it is still unclear whether or not the
$g_\lz^\ast$-function, when
$\lz\in(\frac{2}{\min\{p_-,2\}},1+\frac{2}{\min\{p_-,2\}}]$,
can characterize $\whv$, since the method used in the proof of
Theorem \ref{10.10.x} does not work in this case, while the method used in
\cite[Theorem 4.13]{lyj} strongly depends on the properties of uniformly
Muckenhoupt weights, which are not satisfied by $t^{p(\cdot)}$
with $p(\cdot)\in C^{\log}(\rn)$ (see Remark \ref{1.29.x2}(iii)).
\end{rem}

To prove Theorems \ref{lpthm1}, \ref{lpthm2} and \ref{10.10.x},
we need some technical lemmas.

\begin{lem}\label{lppro1}
Let $p(\cdot)\in C^{\log}(\rn)$. If $f\in\whv$,
then $f$ vanishes weakly at infinity.
\end{lem}

\begin{proof}
Let $f\in\whv$. Then, by Remark \ref{r-m-fun}(i), we know that,
for any $\phi\in\cs(\rn)$, $t\in(0, \fz)$, $x\in\rn$ and $y\in B(x, t)$,
$|f*{\phi_t}(x)|\ls f_{N,\triangledown}^\ast(y)\ls f_{N,+}^\ast(y)$,
where $N\in(\frac{n}{\underline{p}}+n+1,\fz)$ with $\underline{p}$ as in \eqref{2.1y}.
Thus, there exists a positive constant
$C_0$, independent of $x$, $t$ and $f$, such that
$$B(x, t)\subset\{y\in\rn:\
f_{N,+}^\ast(y)\ge C_0|f*{\phi_t}(x)|\}.$$
By this, Remark \ref{r-vlp}(ii) and an argument similar to that used in the
proof of \eqref{max-f3}, we conclude that, for all $x\in\rn$,
\begin{align*}
&\min\{\lf|f*{\phi_t}(x)\r|^{p_-}, \lf|f*{\phi_t}(x)\r|^{p_+}\}\\
&\hs\ls\frac{1}{|B(x,t)|}\max\lf\{\|f\|_{\whv}^{p_-},\|f\|
_{\whv}^{p_+}\r\}\rightarrow0
\end{align*}
as $t\rightarrow\fz$, which implies that $f$ vanishes weakly at infinity,
where $p_+$ and $p_-$ are as in \eqref{2.1x}.
This finishes the proof of Lemma \ref{lppro1}.
\end{proof}

The following inequality of the Lusin area function on classical
Lebesgue spaces is well known,
whose proof can be found, for example, in \cite[Chapter 7]{fs82} (see also \cite{st89}).

\begin{lem}\label{lm-9.22}
Let $q\in(1,\fz)$. Then there exists a positive constant $C$ such that,
for all $f\in L^q(\rn)$,
$$C^{-1}\|f\|_{L^q(\rn)}\le\|S(f)\|_{L^q(\rn)}\le C\|f\|_{L^q(\rn)}.$$
\end{lem}

\begin{proof}[Proof of Theorem \ref{lpthm1}]
We first prove the sufficiency.
Let $f\in\cs'(\rn)$, $f$ vanish
weakly at infinity and $S(f)\in\wlv$. Then
we prove that $f\in \wha$ for some $q$ and $s$ as in
Theorem \ref{atthm1} and
\begin{equation}\label{area-4}
\|f\|_{\whv}\sim\|f\|_{\wha}\ls\|S(f)\|_{\wlv}.
\end{equation}

Denote by $\cq$ the set of all dyadic cubes in $\rn$.
For any $i\in\zz$, let $$\Omega_i:=\{x\in\rn:\ S(f)(x)>2^i\}$$
and
$$\cq_i:=\lf\{Q\in\cq:\ |Q\cap\Omega_i|\ge\frac{|Q|}{2}\ {\rm and}\
|Q\cap\Omega_{i+1}|<\frac{|Q|}{2}\r\}.$$
For any $i\in\zz$, we use $\{\qij\}_{j}$ to denote the maximal dyadic cubes in $\cq_i$, namely,
there does not exist $Q\in\cq_i$ such that $\qij\subsetneqq Q$. For any $Q\in\cq$, let
$\ell(Q)$ denote its side length and
$$Q^+:=\{(y,t)\in\mathbb R^{n+1}_+:\ y\in Q,\, \sqrt n{\ell(Q)}<t\le
 2\sqrt n\ell(Q)\}$$
and, for all $i\in\zz$ and $j$, let
$$\Bij:=\bigcup_{Q\in\cq_i,\,Q\st Q_{i,j}}Q^+.$$
Here we point out that $Q^+$ for different $Q\in \cq_i$ and $Q\subset Q_{i,j}$
are mutually disjoint.
Then, by the proof of \cite[Theorem 4.5]{lyj}, we find that
\begin{equation}\label{area-3}
f=:\sum_{i\in\zz}\sum_{j}\lz_{i,j}\aij\quad \mbox{in}\quad \cs'(\rn),
\end{equation}
where, for any $i\in\zz$, $j$ and $x\in\rn$,
$\lz_{i,j}:=2^i\|\chi_{4\sqrt nQ_{i,j}}\|_{\lv}$,
\begin{align*}
\aij(x):=&\frac1{\lz_{i,j}}\int_{\Bij}f*\phi_t(y)\phi_t(x-y)\frac{dy\,dt}{t}\\
=&\frac1{\lz_{i,j}}\sum_{Q\in\cq_i,Q\st \qij}\int_{Q^+}f*\phi_t(y)
\phi_t(x-y)\frac{\,dy\,dt}{t}
=:\frac1{\lz_{i,j}}\sum_{Q\in\cq_i,Q\st \qij}e_Q(x),
\end{align*}
and $\phi$ is as in \eqref{12.7.x1}, \eqref{12.7.x2}
and \eqref{12.7.x3}.
Moreover, for any $i\in\zz$ and $j$,
$\supp\aij\subset\wz{Q}_{i,j}:=4{\sqrt n}\qij$ and
$$\int_{\rn}\aij(x)x^{\beta} dx=0\quad\mbox{for}\quad|\beta|\le s.$$

Next we estimate $\|\aij\|_{L^q(\rn)}$ for all $i\in\zz$ and $j$.
Let
$$\wz\boz_i:=\lf\{x\in\rn:\ \cm(\chi_{\boz_i})(x)\ge \frac{1}{2}\r\}.$$
Observe that $|Q\cap \Omega_i|\ge \frac{|Q|}2$ for any $Q\in \cq_i$, which implies
that $Q\st \wz \Omega_i$. From this and the fact that
$|Q\cap \Omega_{i+1}|<\frac{|Q|}2$, we deduce that, for all $i\in\zz$
and $x\in Q\in\cq_i$,
\begin{equation}\label{area-1}
\cm\lf(\chi_{Q\cap(\wz\boz_i\setminus\boz_{i+1})}\r)(x)\ge \frac{\chi_Q(x)}{2}.
\end{equation}
For any $Q\in\cq_i$, let
$$c_Q:=\lf[\int_{Q^+}|\phi_t*f(y)|^2\frac{dy\,dt}{t^{n+1}}\r]^{1/2}.$$
Notice that, for all $i\in\zz$, $j$ and $x\in\rn$,
\begin{equation*}
S(a_{i,j})(x)\ls \frac1{\lz_{i,j}}
\lf\{\sum_{Q\in\cq_i,Q\st\qij}\lf[\cm(c_Q\chi_Q)(x)\r]^2
\r\}^{\frac12}
\end{equation*}
(see \cite[(4.9)]{lyj}). From this, Lemma \ref{lm-9.22}, \eqref{area-1}
and the Fefferman-Stein
vector-valued inequality (see, for example, \cite[p.\,51, Theorem 1(c)]{stein93}),
we deduce that, for all $i\in\zz$ and $j$,
\begin{align}\label{area-2}
\|\aij\|_{L^q(\rn)}
&\ls\|S(\aij)\|_{L^q(\rn)}
\ls\frac1{\lz_{i,j}}\lf\|\lf\{\sum_{Q\in\cq_i,Q\st \qij}
\lf[\cm(c_Q\chi_Q)\r]^2\r\}^{1/2}
\r\|_{L^q(\rn)}\noz\\
&\ls\frac1{\lz_{i,j}}\lf\|\lf[\sum_{Q\in\cq_i,Q\st \qij}
(c_Q)^2\chi_Q\r]^{1/2}
\r\|_{L^q(\rn)}\noz\\
&\ls\frac1{\lz_{i,j}}\lf\|\lf\{\sum_{Q\in\cq_i,Q\st \qij}
\lf[c_Q^2\cm\lf(\chi_{Q\cap(\wz\boz_i\setminus\boz_{i+1})}\r)\r]^2\r\}^{1/2}
\r\|_{L^q(\rn)}\noz\\
&\ls\frac1{\lz_{i,j}}\lf\|\lf[\sum_{Q\in\cq_i,Q\st \qij}
(c_Q)^2\chi_{Q\cap(\wz\boz_i\setminus\boz_{i+1})}\r]^{1/2}
\r\|_{L^q(\rn)}.
\end{align}
Since, for all $x\in Q\st\qij$, if $(y,t)\in Q^+$, then
$|x-y|<\sqrt n \ell(Q)\le t$, it follows that
$Q^+\st \Gamma(x)$, which, combined with the fact that
$\{Q^+:\ Q\in \cq_i,\ Q\st \qij\}_{i\in\zz,j}$ are disjoint,
further implies that, for all $i\in\zz$, $j$ and $x\in\rn$,
\begin{align*}
\sum_{Q\in\cq_i,Q\st \qij}
(c_Q)^2\chi_{Q\cap(\wz\boz_i\setminus\boz_{i+1})}(x)
&=\sum_{Q\in\cq_i,Q\st \qij}\int_{Q^+}|\phi_t*f(y)|^2\frac{dy\,dt}{t^{n+1}}
\chi_{Q\cap(\wz\boz_i\setminus\boz_{i+1})}(x)\\
&\le [S(f)(x)]^2\chi_{\qij\cap(\wz\boz_i\setminus\boz_{i+1})}(x)
\ls 2^{2i}\chi_{\qij}(x).
\end{align*}
Thus, by this and \eqref{area-2}, we conclude that, for all $i\in\zz$ and $j$,
\begin{equation*}
\|a_{i,j}\|_{L^q(\rn)}\ls\frac1{\lz_{i,j}}\lf\|2^i\chi_{\qij}\r\|_{L^q(\rn)}
\ls\frac{|\wz{Q}_{i,j}|^{\frac{1}{q}}}{\|\chi_{\wz{Q}_{i,j}}\|_{\lv}}.
\end{equation*}
Therefore, for any $i\in\zz$ and $j$, $a_{i,j}$ is a $(p(\cdot),q,s)$-atom
up to a harmless constant multiple
and hence \eqref{area-3} forms an atomic decomposition of $f$.

On the other hand, by Remark \ref{2.5.y},
$\lf|\qij\cap\boz_i\r|\ge\frac{\lf|\qij\r|}{2}$,
Lemma \ref{zhuolemma} and the fact that $\{\qij\}_{j}$
have disjoint interiors, we find that, for any $i\in\zz$,
\begin{align*}
&\lf\|\lf\{\sum_{j}\lf[\frac{\lz_{i,j}\chi_{\wz{Q}_{i,j}}}
{\|\chi_{\wz{Q}_{i,j}}\|_{\lv}}\r]
^{\underline{p}}\r\}^{\frac1{\underline{p}}}\r\|_{\lv}\\
&\hs\sim2^i\lf\|\lf(\sum_{j}
\chi_{\wz{Q}_{ij}}\r)^{\frac1{\underline{p}}}\r\|_{\lv}
\ls2^i\lf\|\lf(\sum_{j}
\chi_{\qij}\r)^{\frac1{\underline{p}}}\r\|_{\lv}\\
&\hs\ls 2^i\lf\|\lf(\sum_{j}
\chi_{\qij\cap\boz_i}\r)^{\frac1{\underline{p}}}\r\|_{\lv}
\ls2^i\|\chi_{\boz_i}\|_{\lv}\ls\|S(f)\|_{\wlv},
\end{align*}
which, together with Theorem \ref{atthm1}, implies that $f\in\wha=\whv$
and \eqref{area-4} holds true.
This finishes the proof of the sufficiency of Theorem \ref{lpthm1}.

Next we prove the necessity of Theorem \ref{lpthm1}.
Let $f\in \whv$. Obviously, by Lemma \ref{lppro1}, we know that
$f$ vanishes weakly at infinity. Due to Theorem \ref{atthm1},
we can decompose $f$ as follows
$$f=\sum_{i=-\fz}^{i_0-1}\sum_{j\in\nn}\lij\aij
+\sum_{i=i_0}^{\fz}\sum_{j\in\nn}\lij\aij=:f_1+f_2,$$
where $\{\lz_{i,j}\}_{i\in\zz,j\in\nn}$ and $\{a_{i,j}\}_{i\in\zz,j\in\nn}$
are as in Theorem \ref{atthm1} satisfying \eqref{4.1.y}.
Thus, we obtain
\begin{align}\label{lydecom}
&\lf\|\chi_{\{x\in\rn:\ S(f)(x)>\az\}}\r\|_{\lv}\noz\\
&\hs\ls\lf\|\chi_{\{x\in\rn:\ S(f_1)(x)>\frac{\az}{2}\}}\r\|_{\lv}
+\lf\|\chi_{\{x\in A_{i_0}:\ S(f_2)(x)>\frac{\az}{2}\}}\r\|_{\lv}\noz\\
&\hs\hs\hs+\lf\|\chi_{\{x\in(A_{i_0})^{\com}:\ S(f_2)(x)
>\frac{\az}{2}\}}\r\|_{\lv}\noz\\
&\hs=:{\rm I}_1+{\rm I}_2+{\rm I}_3,
\end{align}
where $A_{i_0}:=\bigcup_{i={i_0}}^\fz\bigcup_{j\in\nn}(4\Bij)$
and $\{\Bij\}_{i\in\zz,j\in\nn}$ are the balls as in Theorem \ref{atthm1}.

It is easy to see that
\begin{align}\label{lydecom1}
{\rm I}_1&\ls\lf\|\chi_{\{x\in\rn:\ \sum_{i=-\fz}
^{i_0-1}\sum_{j\in\nn}\lij S(\aij)(x)\chi_{4\Bij}(x)
>\frac{\az}{4}\}}\r\|_{\lv}\noz\\
&\hs\hs+\lf\|\chi_{\{x\in\rn:\ \sum_{i=-\fz}^{i_0-1}
\sum_{j\in\nn}\lij S(\aij)(x)\chi_{(4\Bij)^\com}(x)
>\frac{\az}{4}\}}\r\|_{\lv}\noz\\
&=:{\rm I}_{1,1}+{\rm I}_{1,2}.
\end{align}

For ${\rm I}_{1,1}$, by Lemmas \ref{lm-9.22} and \ref{atlm2}, Remark \ref{2.5.y}
and an argument similar to that
used in the proof of \eqref{eqatomi11}, we conclude that
\begin{align}\label{area-5}
{\rm I}_{1,1}\ls\az^{-1}\|f\|_{\wha}.
\end{align}

For ${\rm I}_{1,2}$, we first write that, for any $i\in\zz$, $j\in\nn$ and $x\in\rn$,
\begin{align}\label{lysaij}
\hs\hs\hs\lf[S(\aij)(x)\r]^2
&=\int_0^{\frac{|x-\xij|}{4}}
\int_{|y-x|<t}|\aij\ast\phi_t(y)|^2\dytn
+\int_{\frac{|x-\xij|}{4}}^\infty\int_{|y-x|<t}\cdots\noz\\
&=:\mj_1+\mj_2,
\end{align}
where $\xij$ denotes the center of $\Bij$.
From the Taylor remainder theorem, we deduce that, for
all $i\in\zz$, $j\in\nn$, $N\in\zz_+$, $t\in(0, \fz)$,
$x\in (4\Bij)^{\com}$, $|y-x|<t$ and $z\in\Bij$,
\begin{align}\label{lyphi}
\lf|\phi\lf(\frac{y-z}{t}\r)-\sum_{|\alpha|\le s}
\frac{\partial^\alpha\phi(\frac{y-\xij}{t})}{\alpha!}
\lf(\frac{\xij-z}{t}\r)^\alpha\r|\ls
\lf(\frac{t}{|\xi|}\r)^{N}\frac{|z-\xij|^{s+1}}{t^{s+1}},
\end{align}
where $\xi=(y-\xij)+\theta(\xij-z)$ for some $\theta\in[0,1]$.

When $x\in (4\Bij)^{\com}$ and $|y-x|<t\le\frac{|x-\xij|}4$, we have
$|y-\xij|\sim|x-\xij|$ and, in this case,
$|\xi|\ge |y-\xij|-|\xij-z|\ge \frac12|x-\xij|$.
Thus, by this and the vanishing moment condition of $\aij$,
\eqref{lyphi} in the case that $N=n+s+2$ and the H\"older inequality,
we know that, for all $i\in\zz$, $j\in\nn$, $x\in (4\Bij)^{\com}$
and $|y-x|<t\le\frac{|x-\xij|}4$,
\begin{align*}
\lf|\aij\ast\phi_t(y)\r|
&\ls t\int_\Bij|\aij(z)|\frac{|z-\xij|^{s+1}}
{|x-\xij|^{n+s+2}}\,dz
\ls\frac{t(\rij)^{s+1}}{|x-\xij|^{n+s+2}}\|\aij\|
_{L^q(\rn)}|\Bij|^{1/{q'}}\\
&\ls\frac{t}{|x-\xij|}\lf(\frac{\rij}{|x-\xij|}\r)^{n+s+1}\|
\chi_{\Bij}\|_{\lv}^{-1}.
\end{align*}
From this, we further deduce that, for all $i\in\zz$, $j\in\nn$
and $x\in(4\Bij)^{\com}$,
\begin{align}\label{lyesti}
\mj_1&\ls
\|\chi_{\Bij}\|_{\lv}^{-2}\lf(\frac{\rij}{|x-\xij|}\r)^{2(n+s+1)}
\frac{1}{|x-\xij|^2}\int_0^{\frac{|x-\xij|}{4}}t\,dt\noz\\
&\sim\|\chi_{\Bij}\|_{\lv}^{-2}
\lf(\frac{\rij}{|x-\xij|}\r)^{2(n+s+1)}\noz\\
&\ls\|\chi_{\Bij}\|_{\lv}^{-2}
\lf[\cm\lf(\chi_{\Bij}\r)(x)\r]^{\frac{2(n+s+1)}{n}}.
\end{align}

When $t\geq \frac{|x-\xij|}{4}$, by \eqref{lyphi}
in the case that $N=0$ and the H\"older inequality,
together with the vanishing moment condition of $\aij$,
we also find that,
for all $i\in\zz$, $j\in\nn$, $x\in (4\Bij)^{\complement}$
and $|y-x|<t$,
\begin{align*}
\lf|\aij\ast\phi_t(y)\r|
\ls \lf(\frac{\rij}{t}\r)^{n+s+1}\|\chi_{\Bij}\|
_{\lv}^{-1},
\end{align*}
which implies that
\begin{align}\label{lyesti2}
\hs\hs\mj_2&\ls\|\chi_{\Bij}\|_{\lv}^{-2}(\rij)^{2(n+s+1)}
 \int_{\frac{|x-\xij|}{4}}^\infty t^{-2(n+s+1)-1}\,dt \noz\\
&\sim\|\chi_{\Bij}\|_{\lv}^{-2}\lf(\frac{\rij}
{|x-\xij|}\r)^{2(n+s+1)}\noz\\
&\ls\|\chi_{\Bij}\|_{\lv}^{-2}
\lf[\cm\lf(\chi_{\Bij}\r)(x)\r]^{\frac{2(n+s+1)}{n}}.
\end{align}
Thus, by \eqref{lysaij}, \eqref{lyesti} and \eqref{lyesti2},
we conclude that, for all
$i\in\zz$, $j\in\nn$ and $x\in(4\Bij)^\com$,
\begin{align}\label{lysaij1}
|S(\aij)(x)|\ls\|\chi_{\Bij}\|_{\lv}^{-1}
\lf[\cm\lf(\chi_{\Bij}\r)(x)\r]^{\frac{n+s+1}{n}}.
\end{align}
From this, the H\"older inequality, Remark \ref{r-vlp}(i),
Lemma \ref{mlm1} and an argument similar to that used in
the proof of \eqref{eqatomi12}, we deduce that
${\rm I}_{1,2}\ls\az^{-1}\|f\|_{\wha}$.
Combining this, \eqref{lydecom1} and \eqref{area-5}, we further conclude that
\begin{align}\label{lye1}
\az{\rm I}_1\ls\|f\|_{\wha}.
\end{align}

By an argument similar to that used in
the proof of \eqref{eqatomii}, we also find that
\begin{align}\label{lyeerr}
{\rm I}_2
\ls\|\chi_{A_{i_0}}\|_{\lv}\ls\az^{-1}\|f\|_{\wha}.
\end{align}
Let $r_2\in(\frac{n}{\underline{p}(n+s+1)},1)$. Then, by \eqref{lysaij1},
Lemma \ref{mlm1} and an argument similar to that used
in the proof of \eqref{eqatomiii}, we know that
\begin{align}\label{lyeez}
{\rm I}_3&
=\lf\|\chi_{\{x\in{(A_{i_0})^{\com}}:\
S(f_2)(x)>\frac{\az}{2}\}}\r\|_{\lv}\noz\\
&\ls\lf\|\lf[\frac{\sum_{i=i_0}^
\fz\sum_{j\in\nn}\lij S(\aij)}
{\az}\r]^{r_2}\chi_{(A_{i_0})^\com}\r\|_{\lv}\noz\\
&\ls\az^{-r_2}\lf[\sum_{i=i_0}^\fz2^{ir_2}\r.\noz\\
&\hs\hs\times\lf.\lf
\|\lf\{\sum_{j\in\nn}\lf[\cm\lf(\chi_{\Bij}\r)\r]
^{\frac{r_2(n+s+1)}{n}}\r\}^{\frac{n}{r_2(n+s+1)}}
\r\|_{L^{\frac{r_2(n+s+1)}{n}p(\cdot)}(\rn)}\r]
^{\frac{r_2(n+s+1)}{n}}\noz\\
&\ls\az^{-1}\|f\|_{\wha}.
\end{align}

Finally, combining \eqref{lydecom}, \eqref{lye1},
\eqref{lyeerr} and \eqref{lyeez}, we conclude that
$$\|S(f)\|_{\wlv}\ls\|f\|_{\whv},$$
which completes the proof of the necessity and hence the proof of Theorem \ref{lpthm1}.
\end{proof}

By an argument similar to that used in the proof
of the necessity part of Theorem \ref{lpthm1}, we obtain the following
boundedness of the Littlewood-Paley $g$-function from
$\whv$ to $\wlv$, the details being omitted.

\begin{prop}\label{lppro2}
Let $p(\cdot)\in C^{\log}(\rn)$.
If $f\in\whv$, then $g(f)\in\wlv$ and
$$\|g(f)\|_{\wlv}\le C\|f\|_{\whv},$$
where $C$ is a positive constant independent of $f$.
\end{prop}

To prove Theorem \ref{lpthm2}, we borrow some ideas from Ullrich \cite{u12}
and begin with the following notation.
For any $\phi\in\cs(\rn)$ and $f\in\cs'(\rn)$,
we let, for all $t,\ a\in(0,\fz)$ and $x\in\rn$,
$$(\phi^*_tf)_a(x):=\sup_{y\in\rn}\frac{|\phi_t*f(x+y)|}{(1+|y|/t)^a}$$
and
\begin{equation}\label{1.21-x}
g_{a,*}(f)(x):=\lf\{\int_{0}^{\fz}[(\phi^*_tf)_a(x)]^2\frac{dt}{t}\r\}^{1/2},
\end{equation}
where $\phi_t(\cdot):=\frac{1}{t^n}\phi(\frac{\cdot}{t})$.

The following estimate is a special case of \cite[Lemma 3.5]{lsuyy},
which is further traced back to \cite[(2.66)]{u12} and the argument
used in the proof of \cite[Theorem 2.6]{u12}.

\begin{lem}\label{lm-12.7}
Let $\phi\in\cs(\rn)$ satisfy
\eqref{12.7.x1}, \eqref{12.7.x2} and \eqref{12.7.x3} and $N_0\in\nn$. Then, for all
$t\in[1,2]$, $a\in(0,N_0]$, $l\in\zz$, $f\in\cs'(\rn)$ and $x\in\rn$, it holds true that
\begin{equation}\label{1.22x}
\lf[(\phi^*_{2^{-l}t}f)_a(x)\r]^r\le C_{(N_0,r)}\sum_{k=0}^{\fz}2^{-kN_0r}
2^{(k+l)n}\int_{\rn}\frac{|(\phi_{2^{-(k+l)}})_t*f(y)|^r}{(1+2^l|x-y|)^{ar}}\,dy,
\end{equation}
where $r$ is an arbitrary fixed positive number and $C_{(N_0,r)}$ a positive
constant independent of $\phi$, $l$, $t$, $f$ and $x$, but may depend on $N_0$ and $r$.
\end{lem}

\begin{proof}
By \eqref{12.7.x2}, we know that, for all $\gz\in\zz_+^n$ with
$|\gz|\le\max\{\lfloor\frac{n}{p_-}-n-1\rfloor,0\}$,
$$D^{\gz}\widehat\phi(\vec0_n)=\int_{\rn}(-2\pi i\xi)^{\gz}\phi(\xi)\,d\xi
=(-2\pi i)^{|\gz|}\int_{\rn}\xi^{\gz}\phi(\xi)\,d\xi=0,$$
where $p_-$ is as in \eqref{2.1x}.

On the other hand, since $\phi$ satisfies \eqref{12.7.x3}, it follows
that there exists $\xi_0\in\rn\setminus\{\vec0_n\}$ such that
$|\widehat{\phi}(\xi_0)|>0$. From this and
the continuity of $\widehat\phi$, we deduce that
there exists $\delta\in(0,\fz)$ such that
$\vec0_n\notin B(\xi_0,\delta)$ and,
for all $\xi\in B(\xi_0,\delta)$, $|\widehat{\phi}(\xi)|>0$.
By this, combined with the fact that $\widehat{\phi}$
is radial due to $\phi$ being radial,
we further conclude that,
$$\mbox{for\ all}\quad \xi\in B(\vec0_n,|\xi_0|+\delta)\backslash B(\vec0_n,|\xi_0|-\delta),\quad
|\widehat{\phi}(\xi)|>0.$$
Thus, if a radial Schwartz function $\phi$ satisfies
\eqref{12.7.x1}, \eqref{12.7.x2} and \eqref{12.7.x3},
then $\phi$ also satisfies the assumptions in \cite[Lemma 3.5]{lsuyy}.
Therefore, by \cite[Lemma 3.5]{lsuyy}, we find that \eqref{1.22x} holds true,
which completes the proof of Lemma \ref{lm-12.7}.
\end{proof}

We now prove Theorem \ref{lpthm2}.

\begin{proof}[Proof of Theorem \ref{lpthm2}]
For any $f\in\whv$, by Lemma \ref{lppro1} and Proposition \ref{lppro2},
we know that $f\in\cs'(\rn)$, $f$ vanishes weakly at infinity and $g(f)\in\wlv$.
Thus, to prove Theorem \ref{lpthm2},
by Theorem \ref{lpthm1}, it suffices to show that, for any $f\in\cs'(\rn)$, which
vanishes weakly at infinity, it holds true that
\begin{align}\label{lyine}
\|S(f)\|_{\wlv}\ls\|g(f)\|_{\wlv}.
\end{align}
To this end, let $f\in\cs'(\rn)$ vanish weakly at infinity. It is easy to know that,
for any $a\in(0,\fz)$ and almost every $x\in\rn$,
$S(f)(x)\ls g_{a,*}(f)(x)$.
Thus, to prove \eqref{lyine}, it suffices to show that
\begin{align}\label{elypo}
\lf\|g_{a,*}(f)\r\|_{\wlv}\ls\lf\|g(f)\r\|_{\wlv}
\end{align}
holds true for some $a\in(\frac{n}{\min\{p_-,2\}},\fz)$. We now prove \eqref{elypo}.
Since $a\in(\frac{n}{\min\{p_-,2\}},\fz)$,
it follows that there exists $r\in\lf(0,{\min\{p_-,2\}}\r)$
such that $a\in(\frac{n}{r},\fz)$. Choosing $N_0$ sufficiently large,
then, by Lemma \ref{lm-12.7} and the
Minkowski integral inequality, we find that, for all $x\in\rn$,
\begin{align*}
g_{a,*}(f)(x)
&=\lf\{\sum_{j\in\zz}\int_1^2[(\phi_{2^{-j}t}^*f)_a(x)]^2\frac{dt}{t}\r\}^{1/2}\\
&\ls\lf\{\sum_{j\in\zz}\int_1^2\lf[\sum_{k=0}^{\fz}2^{-kN_0r}2^{(k+j)n}
\int_{\rn}\frac{|(\phi_{2^{-(k+j)}})_t*f(y)|^r}{(1+2^j|x-y|)^{ar}}\,dy\r]
^{\frac{2}{r}}\frac{dt}{t}\r\}^{1/2}\\
&\ls\lf[\sum_{j\in\zz}\lf\{\sum_{k=0}
^{\fz}2^{-kN_0r}2^{(k+j)n}\int_{\rn}
\frac{[\int_1^2|\lf(\phi_{2^{-(k+j)}}\r)_t*f(y)|^2\frac{dt}{t}]
^{\frac{r}{2}}}{(1+2^j|x-y|)^{ar}}\,dy\r\}^{\frac{2}{r}}\r]^{1/2},
\end{align*}
which, together with Lemma \ref{mlmim} and Remarks \ref{r-ar} and \ref{10.24.x1},
implies that
\begin{align*}
&\lf\|g_{a,*}(f)\r\|_{\wlv}^{rv}\\
&\hs\ls\lf\|\sum_{k=0}
^{\fz}2^{-k(N_0r-n)}\lf[\sum_{j\in\zz}2^{j\frac{2n}{r}}
\lf\{\int_{\rn}\frac{[\int_1^2|\lf(\phi_{2^{-(k+j)}}\r)_t*f(y)|
^2\frac{dt}{t}]^{\frac{r}{2}}}{(1+2^j|\cdot-y|)^{ar}}\,dy\r\}
^{\frac{2}{r}}\r]^{\frac{r}{2}}\r\|
_{WL^{\frac{p(\cdot)}{r}}(\rn)}^v\\
&\hs\ls\sum_{k=0}
^{\fz}2^{-kv(N_0r-n)}\lf\|\lf[\sum_{j\in\zz}2^{j\frac{2n}{r}}
\lf\{\int_{\rn}\frac{[\int_1^2|\lf(\phi_{2^{-(k+j)}}\r)_t*f(y)|
^2\frac{dt}{t}]^{\frac{r}{2}}}{(1+2^j|\cdot-y|)^{ar}}\,dy\r\}
^{\frac{2}{r}}\r]^{\frac{r}{2}}\r\|
_{WL^{\frac{p(\cdot)}{r}}(\rn)}^v\\
&\hs\ls\sum_{k=0}
^{\fz}2^{-kv(N_0r-n)}\\
&\hs\hs\times\!\lf\|\lf[\sum_{j\in\zz}2^{j\frac{2n}{r}}
\lf\{\sum_{i=0}^{\fz}2^{-iar}
\int_{|\cdot-y|\sim2^{i-j}}\lf[\int_1^2|\lf(\phi_{2^{-(k+j)}}\r)_t*f(y)|
^2\frac{dt}{t}\r]^{\frac{r}{2}}\,dy\r\}
^{\frac{2}{r}}\r]^{\frac{r}{2}}\r\|
_{WL^{\frac{p(\cdot)}{r}}(\rn)}^v,
\end{align*}
where $v$ is as in Remark \ref{r-ar} and
$|\cdot-y|\sim2^{i-j}$ means that $|x-y|<2^{-j}$ when $i=0$,
or $2^{i-j-1}\le|x-y|<2^{i-j}$ when $i\in\nn$. Applying the Minkowski
series inequality, Proposition \ref{mlmveq}, and Remarks \ref{r-ar}
and \ref{10.24.x1}, we conclude that
\begin{align*}
\lf\|g_{a,*}(f)\r\|_{\wlv}^{rv}
&\ls\sum_{k=0}^{\fz}2^{-kv(N_0r-n)}\lf\|\sum_{i=0}^{\fz}2^{-iar+in}\r.\\
&\hs\times\lf.\lf\{\sum_{j\in\zz}
\lf[\cm\lf(\lf[\int_1^2|\lf(\phi_{2^{-(k+j)}}\r)_t*f|
^2\frac{dt}{t}\r]^{\frac{r}{2}}\r)\r]
^{\frac{2}{r}}\r\}^{\frac{r}{2}}\r\|
_{WL^{\frac{p(\cdot)}{r}}(\rn)}^v\\
&\ls\sum_{k=0}^{\fz}2^{-kv(N_0r-n)}\\
&\hs\times\sum_{i=0}^{\fz}2^{(-iar+in)v}
\lf\|\lf\{\sum_{j\in\zz}
\lf[\int_1^2|\lf(\phi_{2^{-(k+j)}}\r)_t*f|
^2\frac{dt}{t}\r]\r\}^{\frac{r}{2}}\r\|
_{WL^{\frac{p(\cdot)}{r}}(\rn)}^v\\
&\ls\lf\|g(f)\r\|_{\wlv}^{rv},
\end{align*}
which implies that \eqref{elypo} holds true.
This finishes the proof of Theorem \ref{lpthm2}.
\end{proof}

Applying Theorems \ref{lpthm1} and \ref{lpthm2}, we now prove Theorem \ref{10.10.x}.
\begin{proof}[Proof of Theorem \ref{10.10.x}]
To prove this theorem, we only need to show the necessity,
since the sufficiency is easy because of Theorem \ref{lpthm1} and
the obvious fact that, for all $f\in\cs'(\rn)$ and
$x\in\rn$, $S(f)(x)\le g_\lz^{\ast}(f)(x)$.

To show the necessity, for any $f\in\whv$, by Lemma \ref{lppro1}, we know that
$f$ vanishes weakly at infinity. From the fact that
$\lambda\in(1+\frac{2}{{\min\{p_-,2\}}}, \fz)$,
we deduce that there exists
$a\in(\frac{n}{{\min\{p_-,2\}}},\fz)$ such that $\lz\in(1+\frac{2a}{n},\fz)$
and, for all $x\in\rn$,
\begin{align*}
g_\lambda^{\ast}(f)(x)&=\lf[\int_0^\infty\int_{\mathbb  R^n}
\lf(\frac{t}{t+|x-y|}\r)^{\lz n}
\lf|f\ast\phi_t(y)\r|^2\,\frac{\,dy\,dt}{t^{n+1}}\r]^{1/2}\\
&\ls\lf\{\int_0^{\fz}[(\phi_t^*f)_a(x)]^2\int_{\rn}\lf(1+\frac{|x-y|}{t}\r)
^{2a-\lz n}\frac{dy\,dt}{t^{n+1}}\r\}^{1/2}\\
&\sim\lf\{\int_0^{\fz}[(\phi_t^*f)_a(x)]^2
\frac{dt}{t}\r\}^{1/2}\sim g_{a,*}(f)(x),
\end{align*}
which, together with \eqref{elypo} and Theorem \ref{lpthm2},
implies that
$$\|g_\lambda^{\ast}(f)\|_{\wlv}\ls\|f\|_{\whv}.$$
This finishes the proof of Theorem \ref{10.10.x}.
\end{proof}

\section{Boundedness of Calder\'on-Zygmund operators\label{s-bou}}
\hskip\parindent
In this section, as an application of the variable weak Hardy space $\whv$, we
establish
the boundedness of the Calder\'on-Zygmund operators from the variable Hardy
space $\hv$ to $\whv$. We begin with recalling the definition of
the variable Hardy space $\hv$ (see \cite[Definition 1.1]{ns12}).

\begin{defn}
Let $p(\cdot)\in C^{\log}(\rn)$ and $N\in(\frac{n}{p_-}+n+1,\fz)\cap\nn$
with $p_-$ as in \eqref{2.1x}.
The \emph{variable Hardy space} $\hv$ is defined to be the set of all
$f\in\cs'(\rn)$ such that the (quasi-)norm
$$\|f\|_{\hv}:=\|f_{N,+}^\ast\|_{\lv}<\fz,$$
where $f_{N,+}^\ast$ is as in \eqref{2.8x}.
\end{defn}

Recall that the \emph{variable atomic Hardy space}
$\ha$ is defined as
the space of all $f\in\cs'(\rn)$ such that
$f=\sum_{j\in\nn}\lz_j a_j$ in $\cs'(\rn)$, where $\{\lz_j\}_{j\in\nn}$
is a sequence of non-negative numbers, $\{a_j\}_{j\in\nn}$ is a sequence
of $(p(\cdot),q,s)$ atoms, associated with balls $\{B_j\}_{j\in\nn}$ of $\rn$,
satisfying that
\begin{enumerate}
\item[{\rm (i)}] $\supp a_j \st B_j$;

\item[{\rm (ii)}] $\|a_j\|_{L^q(\rn)}\le\frac{|B_j|^{1/q}}{\|\chi_{B_j}\|_{\lv}}$;

\item[{\rm (iii)}] $\int_{\mathbb R^n}a_j(x)x^\az\,dx=0$ for all $\az\in{\zz}_+^n$
with $|\az|\le s$.
\end{enumerate}

Moreover, for any $f\in\ha$, let
$$\|f\|_{\ha}=\inf\lf\{\lf\|\lf\{\sum_{j\in\nn}
\lf[\frac{\lz_j\chi_{B_j}}{\|\chi_{B_j}\|_{\lv}}\r]
^{\underline{p}}\r\}^{\frac{1}{\underline{p}}}\r\|_{\lv}\r\},$$
where $\underline{p}$ is as in \eqref{2.1y} and the infimum is taken over all
admissible decompositions of $f$ as above.

The following lemma is just \cite[Theorem 1.1]{s10}.

\begin{lem}\label{blm1}
Let $p(\cdot)\in C^{\log}(\rn)$, $q\in [1,\fz]\cap (p_+,\fz]$ and
$s\in(\frac{n}{p_-}-n-1,\fz)\cap{\zz}_+$, where $p_+$ and $p_-$ are as in \eqref{2.1x}.
Then $\hv=\ha$ with equivalent quasi-norms.
\end{lem}

\begin{rem}\label{07-22}
Let $p(\cdot)\in C^{\log}(\rn)$ and $p_+\in(0,1]$. Then, from the proof of
\cite[Theorem 4.5]{ns12}, we deduce that the subspace $H^{p(\cdot)}(\rn)\cap L^2(\rn)$
is dense in $H^{p(\cdot)}(\rn)$.
\end{rem}

Recall that, for any given $\delta\in(0,1]$, a \emph{convolutional $\delta$-type
Calder\'on-Zygmund operator} $T$ means that: $T$ is a linear bounded operator on
$L^2(\rn)$ with kernel $k\in\cs'(\rn)$ coinciding with a locally integrable function on
$\rn\backslash \{\vec0_n\}$ and satisfying that, for all $x$, $y\in\rn$ with $|x|>2|y|$,
$$|k(x-y)-k(x)|\le C\frac{|y|^\delta}{|x|^{n+\delta}}$$
and, for all $f\in L^2(\rn)$, $Tf(x)=k*f(x).$

The first main result of this section reads as follows.

\begin{thm}\label{bdnthm2}
Let $p(\cdot):\ \rn\to(0,1]$ belong to $C^{\log}(\rn)$ and $\delta\in(0,1]$.
Let $T$ be a convolutional $\delta$-type
Calder\'on-Zygmund operator.
If $p_-\in[\frac{n}{n+\delta},1]$ with $p_-$ as in \eqref{2.1x}, then
$T$ has a unique extension on $H^{p(\cdot)}(\rn)$ and, moreover,
for all $f\in\hv$,
$$\|Tf\|_{\whv}\le C\|f\|_{H^{p(\cdot)}(\rn)},$$
where $C$ is a positive constant independent of $f$.
\end{thm}

The proof of Theorem \ref{bdnthm2} is given below.

\begin{rem}\label{3.4.re1}
If $p(\cdot)\equiv p\in(0,1]$, then $\whv=W\!H^p(\rn)$. In this case,
Theorem \ref{bdnthm2} indicates that, if $\delta\in(0,1]$, $p=\frac{n}{n+\delta}$
and $T$ is a convolutional $\delta$-type Calder\'on-Zygmund operator, then
$T$ is bounded from $H^{\frac{n}{n+\delta}}(\rn)$ to $W\!H^{\frac{n}{n+\delta}}(\rn)$,
which is just \cite[Theorem 1]{liu91} (see also \cite[Theorem 5.2]{lyj}).
Here $\frac n{n+\delta}$ is called the \emph{critical index}. Thus,
the boundedness of the Calder\'on-Zygmund operator from $H^{p(\cdot)}(\rn)$ to
$\whv$ obtained in Theorem \ref{bdnthm2} includes the critical case.
\end{rem}

Recall that, for any given $\gamma\in(0,\fz)$, a linear operator $T$ is called
a \emph{$\gamma$-order Calder\'on-Zygmund operator} if $T$ is bounded on $L^2(\rn)$
and its kernel
$$k:\ (\rn\times\rn)\backslash\{(x,x):\ x\in\rn\}\to\mathbb C$$
satisfies that there exists a positive constant $C$ such that,
for any $\az\in\zz^n_+$ with $|\az|\le\lceil\gamma\rceil$
and $x$, $y$, $z\in\rn$ with $|x-y|>2|y-z|$,
\begin{align}\label{3.1.y}
\lf|\partial_x^\az k(x,y)-\partial_x^\az k(x,z)\r|\le
C\frac{|y-z|^{\gamma-\lceil\gamma\rceil}}{|x-y|^{n+\gamma}}
\end{align}
and, for any $f\in L^2(\rn)$ having compact support and $x\notin\supp f$,
$$Tf(x)=\int_{\supp f}k(x,y)f(y)\,dy.$$
Here and hereafter, for any $\gz\in(0,\fz)$, $\lceil \gamma\rceil$ denotes the maximal integer
less than $\gz$.

For $m\in\nn$, an operator $T$ is said to have the \emph{vanishing moment condition up to
order $m$} if, for any $a\in L^2(\rn)$ having compact support and satisfying that,
for all $\beta\in\zz^n_+$ with $|\beta|\le m$, $\int_{\rn}a(x)x^\beta\,dx=0$,
it holds true that $\int_{\rn}x^\beta Ta(x)\,dx=0$.

The second main result of this section is stated as follows.

\begin{thm}\label{bdnthm3}
Let $p(\cdot):\ \rn\to(0,1]$ belong to $C^{\log}(\rn)$ and $\gamma\in(0,\fz)$.
Let $T$ be a $\gamma$-order Calder\'on-Zygmund operator and have
the vanishing moment condition up to order $\lceil \gamma\rceil$.
If $\lceil \gamma\rceil\le n(\frac{1}{p_-}-1)\le\gamma$ with $p_-$
as in \eqref{2.1x}, then $T$ has a unique extension on
$H^{p(\cdot)}(\rn)$ and, moreover, for all $f\in\hv$,
$\|Tf\|_{\whv}\le C\|f\|_{H^{p(\cdot)}(\rn)}$,
where $C$ is a positive constant independent of $f$.
\end{thm}

The proof of Theorem \ref{bdnthm3} is presented below.

\begin{rem}\label{3.4.re2}
Recall that, for $\delta\in(0,1]$, a
\emph{non-convolutional $\delta$-type Calder\'on-Zygmund operator}
$T$ means that: $T$ is a linear bounded operator on $L^2(\rn)$
and there exist a kernel $k$ on $(\rn\times\rn)\backslash\{(x,x):\ x\in\rn\}$
and a positive constant $C$ such that, for any $x$, $y$, $z\in\rn$
with $|x-y|>2|y-z|$,
$$|k(x,y)-k(x,z)|\le C\frac{|y-z|^\delta}{|x-y|^{n+\delta}}$$
and, for any $f\in L^2(\rn)$ having compact support and $x\notin\supp f$,
$$Tf(x)=\int_{\supp f}k(x,y)f(y)\,dy.$$
Notice that, when $\gamma:=\delta\in(0,1]$, the operator $T$
in Theorem \ref{bdnthm3} is just a non-convolutional
$\delta$-type Calder\'on-Zygmund operator.
Thus, the operators of Theorem \ref{bdnthm3} conclude the non-convolutional
$\delta$-type Calder\'on-Zygmund operators as special cases. From this,
it is easy to see that the critical index
of $\gamma$-order Calder\'on-Zygmund operators is $\frac n{n+\gamma}$.
\end{rem}

To prove Theorems \ref{bdnthm2} and \ref{bdnthm3},
we need the following proposition.

\begin{prop}\label{1.14.x3}
Let $r\in(1,\fz)$ and $p(\cdot)\in C^{\log}(\rn)$ with $p_-\in[1,\fz)$,
where $p_-$ is as in \eqref{2.1x}.
Then there exists a positive constant $C$ such that, for any sequence
$\{f_j\}_{j\in\nn}$ of measurable functions and $\az\in(0,\fz)$,
$$\az\lf\|\chi_{\{x\in\rn:\ \{\sum_{j\in\nn}
[\cm(f_j)(x)]^r\}^{\frac1{r}}>\az\}}\r\|_{\lv}
\le C\lf\|\lf(\sum_{j\in\nn}|f_j|^r\r)^{\frac1{r}}\r\|_{\lv}.$$
\end{prop}

The proof of Proposition \ref{1.14.x3} depends on the following
extrapolation theorem corresponding to the {Muckenhoupt weight class} $A_1(\rn)$,
which is just \cite[Theorem 5.24]{cfbook} and a weaker version can also
be found in \cite[Theorem 1.3]{cf06}.
Recall that a locally integrable function $w$,
which is positive almost everywhere on $\rn$, is said to belong to
$A_1(\rn)$ if
$$[w]_{A_1(\rn)}:=\mathop\mathrm{ess\,sup}_{x\in\rn}\frac{\cm(w)(x)}{w(x)}<\fz,$$
where $\cm$ denotes the Hardy-Littlewood maximal function as in \eqref{2.2x}.

\begin{lem}\label{1.14.x2}
Suppose that the family $\cf$ is a set of all pairs of functions
$(F,G)$ satisfying that there exists $p_0\in[1,\fz)$ such that,
for every $w\in A_1(\rn)$,
$$\int_\rn [F(x)]^{p_0}w(x)\,dx\le C_{(p_0,[w]_{A_1(\rn)})}\int_\rn [G(x)]^{p_0}w(x)\,dx,$$
where $C_{(p_0,[w]_{A_1(\rn)})}$ is a positive constant independent of $F$ and $G$,
but may depend on $p_0$ and $[w]_{A_1(\rn)}$.
Let $p(\cdot)\in\cp(\rn)$ be such that
$p_0\le p_-\le p_+<\fz$ with $p_-$ and $p_+$ as in \eqref{2.1x}.
If the maximal operator $\cm$ is bounded on $L^{(p(\cdot)/p_0)^*}(\rn)$,
where, for all $x\in\rn$,
$$\frac{1}{(p(x)/p_0)^*}+\frac{1}{p(x)/p_0}=1,$$
then there exists a positive constant $C$ such that, for all $(F,G)\in\cf$,
$$\|F\|_{\lv}\le C\|G\|_{\lv}.$$
\end{lem}

We also need the following weak-type weighted Fefferman-Stein
vector-valued inequality of the Hardy-Littlewood maximal operator $\cm$ in \eqref{2.2x}
from \cite[Theorem 3.1(a)]{aj80}.

\begin{lem}\label{1.14.x1}
Let $r\in(1,\fz)$ and $w\in A_1(\rn)$. Then there
exists a positive constant $C$, depending on $n$, $r$ and $[w]_{A_1(\rn)}$,
such that, for all $\az\in(0,\fz)$
and sequences $\{f_j\}_{j\in\nn}$ of measurable functions on $\rn$,
$$\az\, w\lf(\lf\{x\in\rn:\ \lf(\sum_{j\in\nn}
[\cm(f_j)(x)]^r\r)^{\frac1{r}}>\az\r\}\r)
\le C\int_{\rn}\lf[\sum_{j\in\nn}|f_j(x)|^r\r]^{\frac1{r}}w(x)\,dx.$$
\end{lem}

\begin{proof}[Proof of Proposition \ref{1.14.x3}]
For any $r\in(1,\fz)$,  $\az\in(0,\fz)$ and
any sequence $\{f_j\}_{j\in\nn}$ of measurable functions on $\rn$,
let $\cf_\az$ be the set of all pairs $(F_\az,\ G)$,
where, for all $x\in\rn$,
$$F_\az(x):=\az\chi_{\{x\in\rn:\ \{\sum_{j\in\nn}
[\cm(f_j)(x)]^r\}^{\frac1{r}}>\az\}}(x)
\quad\mbox{and}\quad
G(x):=\lf[\sum_{j\in\nn}|f_j(x)|^r\r]^{\frac1{r}}.$$
Then, by Lemma \ref{1.14.x1}, we know that, for every $w\in A_1(\rn)$,
\begin{align}\label{1.17x}
&\int_{\rn}F_\az(x)w(x)\,dx\noz\\
&\hs=\az\, w\lf(\lf\{x\in\rn:\ \lf(\sum_{j\in\nn}
[\cm(f_j)(x)]^r\r)^{\frac1{r}}>\az\r\}\r)
\ls\int_{\rn}G(x)w(x)\,dx.
\end{align}
From $p(\cdot)\in C^{\log}(\rn)$ and $p_-\in[1,\fz)$, it follows that
$\cm$ is bounded on $L^{p^*(\cdot)}(\rn)$ (see, for example, \cite[Theorem 3.16]{cfbook}).
Thus, by this, \eqref{1.17x} and via applying Lemma \ref{1.14.x2} with $p_0:=1$
and $\cf:=\cf_\az$, we conclude that, for all $\az\in(0,\fz)$,
\begin{align*}
\az\lf\|\chi_{\{x\in\rn:\ \{\sum_{j\in\nn}
[\cm(f_j)(x)]^r\}^{\frac1{r}}>\az\}}\r\|_{L^{p(\cdot)}(\rn)}
=\lf\|F_\az\r\|_{\lv}\ls \|G\|_{\lv}.
\end{align*}
Therefore, we find that, for all $\az\in(0,\fz)$ and sequences $\{f_j\}_{j\in\nn}$
of measurable functions,
$$\az\lf\|\chi_{\{x\in\rn:\ \{\sum_{j\in\nn}
[\cm(f_j)(x)]^r\}^{\frac1{r}}>\az\}}\r\|_{\lv}
\ls\lf\|\lf[\sum_{j\in\nn}|f_j(x)|^r\r]^{\frac1{r}}\r\|_{\lv},$$
which completes the proof of Proposition \ref{1.14.x3}.
\end{proof}

We next prove Theorem \ref{bdnthm2}.

\begin{proof}[Proof of Theorem \ref{bdnthm2}]
Let $p(\cdot)$ and $s$ be as in Lemma \ref{blm1} and $f\in H^{p(\cdot)}(\rn)\cap L^2(\rn)$.
Then, by Lemma \ref{blm1} and its proof (see also the proof of \cite[Theorem 4.5]{ns12}), we know that
there exist sequences $\{\lz_j\}_{j\in\nn}$ of positive constants and
$\{a_j\}_{j\in\nn}$ of $(p(\cdot),2,s)$-atoms supported on balls $\{B_j\}_{j\in\nn}:=\{B(x_j,r_j):\
x_j\in\rn\ \mbox{and}\ r_j\in(0,\fz)\}_{j\in\nn}$ such that
\begin{equation}\label{7.2x}
f=\sum_{j\in\nn}\lz_ja_j\quad {\rm in} \quad L^2(\rn)
\end{equation}
and
$$\lf\|\lf\{\sum_{j\in\nn}\lf[\frac{\lz_j\chi_{B_j}}
{\|\chi_{B_j}\|_{\lv}}\r]
^{\underline{p}}\r\}^{\frac{1}{\underline{p}}}\r\|_{\lv}
\ls\|f\|_{H^{p(\cdot)}(\rn)},$$
where $\underline{p}$ is as in \eqref{2.1y}.
Since the operator $T$ is bounded on $L^2(\rn)$, it follows from \eqref{7.2x} that
\begin{equation*}
Tf=\sum_{j\in\nn}\lz_jTa_j
\end{equation*}
holds true in $L^2(\rn)$, namely, $Tf$ is well defined for any $f\in H^{p(\cdot)}(\rn)\cap
L^2(\rn)$. Let $\phi\in\cs(\rn)$ satisfy $\int_{\rn}\phi(x)\,dx\neq0$.
Then, to prove this theorem, by Theorem \ref{mthm1}, we only need to show that
\begin{align}\label{7.1x}
\lf\|\phi_{+}^\ast(Tf)\r\|_{\wlv}\ls\|f\|_{H^{p(\cdot)}(\rn)},
\end{align}
where $\phi_{+}^\ast(Tf)$ is as in \eqref{3.5x} with $f$ replaced by $Tf$.
For any $\az\in(0,\fz)$, by Remark \ref{r-vlp}, we have
\begin{align*}
\quad&\az\lf\|\chi_{\{x\in\rn:\ \phi_{+}^\ast(Tf)(x)>\az\}}\r\|_{\lv}\\
&\hs\le\az\lf\|\chi_{\{x\in\rn:\ \sum_{j\in\nn}\lz_j
\phi_{+}^\ast(Ta_j)(x)>\az\}}\r\|_{\lv}\\
&\hs\ls\az\lf\|\chi_{\{x\in\rn:\ \sum_{j\in\nn}
\lz_j\phi_{+}^\ast(Ta_j)(x)\chi_{4B_j}(x)
>\frac{\az}2\}}\r\|_{\lv}\\
&\quad\quad+\az\lf\|\chi_{\{x\in\rn:\ \sum_{j\in\nn}\lz_j
\phi_{+}^\ast(Ta_j)(x)\chi_{(4B_j)^{\complement}}(x)
>\frac{\az}2\}}\r\|_{\lv}\\
&\hs\ls\lf\|\sum_{j\in\nn}\lz_j\phi_{+}^\ast(Ta_j)\chi_{4B_j}\r\|_{\lv}
+\az\lf\|\chi_{\{x\in\rn:\ \sum_{j\in\nn}\lz_j
\phi_{+}^\ast(Ta_j)(x)\chi_{(4B_j)^{\complement}}(x)
>\frac{\az}2\}}\r\|_{\lv}\\
&\hs=:{\rm I}+{\rm II}.
\end{align*}

Observe that $\phi_{+}^\ast(Ta_j)\ls\cm(Ta_j)$ and $a_j\in L^2(\rn)$.
From the facts that $\cm$ is bounded on $L^r(\rn)$ with $r\in(1,\fz]$ and that
$T$ is bounded on $L^2(\rn)$, we conclude that
$$\|\phi_{+}^\ast(Ta_j)\|_{L^2(\rn)}\ls\|\cm(Ta_j)\|_{L^2(\rn)}
\ls\|Ta_j\|_{L^2(\rn)}\ls\|a_j\|_{L^2(\rn)}
\ls\frac{|B_j|^{1/2}}{\|\chi_{B_j}\|_{\lv}},$$
which, combined with Lemmas \ref{atlm2} and \ref{mlm1},
implies that
\begin{align}\label{1.26.x1}
{\rm I}&\ls\lf\|\lf\{\sum_{j\in\nn}
\lf[\frac{\lz_j\chi_{4B_j}}{\|\chi_{B_j}\|_{\lv}}\r]
^{\underline{p}}\r\}^{\frac{1}{\underline{p}}}\r\|_{\lv}\noz\\
&\ls\lf\|\lf\{\sum_{j\in\nn}
\lf[\frac{\lz_j\chi_{B_j}}{\|\chi_{B_j}\|_{\lv}}\r]
^{\underline{p}}\r\}^{\frac{1}{\underline{p}}}\r\|_{\lv}
\ls\|f\|_{H^{p(\cdot)}(\rn)}.
\end{align}

Next, we deal with ${\rm II}$. To this end, by an argument similar to that
used in the proof of \cite[(5.4), (5.5),(5.6) and (5.7)]{lyj},
we conclude that for any $j\in\nn$,
\begin{equation*}
\phi_+^\ast(Ta_j)\chi_{(4B_j)^\complement}(x)\ls\frac{{r_j}^{n+\delta}}{|x-x_j|^{n+\delta}}
\frac{1}{\|\chi_{B_j}\|_{\lv}}
\ls\lf[\cm(\chi_{B_j})(x)\r]^{\frac{n+\delta}{n}}
\frac{1}{\|\chi_{B_j}\|_{\lv}}
\end{equation*}
Thus, by Proposition \ref{1.14.x3}, we know that
\begin{align}\label{1.26.x2}
{\rm II}&\ls\az\lf\|\chi_{\{x\in\rn:\ \sum_{j\in\nn}
\frac{\lz_j}{\|\chi_{B_j}\|_{\lv}}
[\cm(\chi_{B_j})(x)]^{\frac{n+\delta}{n}}
>\frac{\az}2\}}\r\|_{\lv}\noz\\
&\ls\frac{\az}{2}\lf\|\chi_{\{x\in\rn:\ [\sum_{j\in\nn}\frac{\lz_j}
{\|\chi_{B_j}\|_{\lv}}\{\cm(\chi_{B_j})(x)\}^{\frac{n+\delta}{n}}]
^{\frac{n}{n+\delta}}>(\frac{\az}{2})^{\frac{n}{n+\delta}}\}}\r\|
_{L^{\frac{(n+\delta)p(\cdot)}{n}}(\rn)}^{\frac{n+\delta}{n}}\noz\\
&\ls\lf\|\lf[\sum_{j\in\nn}\frac{\lz_j\chi_{B_j}}
{\|\chi_{B_j}\|_{\lv}}\r]^{\frac{n}{n+\delta}}\r\|
_{L^{\frac{(n+\delta)p(\cdot)}{n}}(\rn)}^{\frac{n+\delta}{n}}\noz\\
&\ls\lf\|\lf\{\sum_{j\in\nn}
\lf[\frac{\lz_j\chi_{B_j}}{\|\chi_{B_j}\|_{\lv}}\r]
^{\underline{p}}\r\}^{\frac{1}{\underline{p}}}\r\|_{\lv}
\ls\|f\|_{H^{p(\cdot)}(\rn)}.
\end{align}
Finally, combining \eqref{1.26.x1} and \eqref{1.26.x2}, we conclude that,
for any $\az\in(0,\fz)$,
$$\az\lf\|\chi_{\{x\in\rn:\ \phi_{+}^\ast(Tf)(x)>\az\}}\r\|_{\lv}
\ls\|f\|_{H^{p(\cdot)}(\rn)},$$
namely, \eqref{7.1x} holds true.
This, together with Remark \ref{07-22} and a dense argument, finishes the proof of Theorem \ref{bdnthm2}.
\end{proof}

We finally prove Theorem \ref{bdnthm3}.

\begin{proof}[Proof of Theorem \ref{bdnthm3}]
By an argument similar to that used in the proof of Theorem \ref{bdnthm2},
it suffices to show that, for all $\az\in(0,\fz)$ and $f\in\hv\cap L^2(\rn)$,
\begin{align}\label{3.4.x1}
\az\lf\|\chi_{\{x\in\rn:\ \sum_{j\in\nn}\lz_j
\phi_{+}^\ast(Ta_j)(x)\chi_{(4B_j)^{\complement}}(x)
>\frac{\az}2\}}\r\|_{\lv}\ls\|f\|_{\hv},
\end{align}
where $\lz_j$, $a_j$ and $B_j$ are as in the proof of Theorem \ref{bdnthm2}.
To this end, we need some finer estimates about $\phi_{+}^\ast(Ta_j)$.
By the vanishing moment condition of $T$
and the fact that $\lceil \gamma\rceil\le n(\frac{1}{p_-}-1)\le s$,
we know that, for all $j\in\nn$, $t\in(0,\fz)$ and $x\in(4B_j)^{\com}$,
\begin{align}\label{3.1.y2}
\lf|\phi_t*Ta_j(x)\r|
&\le\frac{1}{t^n}\int_{\rn}
\lf|\phi\lf(\frac{x-y}{t}\r)-\sum_{|\beta|\le\lceil\gamma\rceil}
\frac{D^\beta\phi\lf(\frac{x-x_j}{t}\r)}{\beta!}
\lf(\frac{y-x_j}{t}\r)^\beta\r||Ta_j(y)|\,dy\noz\\
&=\frac{1}{t^n}\lf(\int_{|y-x_j|<2r_j}+\int_{2r_j\le|y-x_j|<\frac{|x-x_j|}{2}}
+\int_{|y-x_j|\ge\frac{|x-x_j|}{2}}\r)\noz\\
&\quad\quad\times\lf|\phi\lf(\frac{x-y}{t}\r)-\sum_{|\beta|\le\lceil\gamma\rceil}
\frac{D^\beta\phi\lf(\frac{x-x_j}{t}\r)}{\beta!}
\lf(\frac{y-x_j}{t}\r)^\beta\r||Ta_j(y)|\,dy\noz\\
&=:{\rm A}_1+{\rm A}_2+{\rm A}_3.
\end{align}

For ${\rm A}_1$, by the Taylor remainder theorem, we find that, for any $j\in\nn$ and $y\in\rn$
with $|y-x_j|<2r_j$, there exists $\xi_1(y)\in 2B_j$ such that
\begin{align*}
{\rm A}_1&=\frac{1}{t^n}\int_{|y-x_j|<2r_j}
\lf|\phi\lf(\frac{x-y}{t}\r)-\sum_{|\beta|\le\lceil\gamma\rceil}
\frac{D^\beta\phi\lf(\frac{x-x_j}{t}\r)}{\beta!}
\lf(\frac{y-x_j}{t}\r)^\beta\r||Ta_j(y)|\,dy\\
&\le\frac{1}{t^n}\int_{|y-x_j|<2r_j}
\lf|\sum_{|\beta|=\lceil\gamma\rceil+1}\partial^\beta\phi
\lf(\frac{x-\xi_1(y)}{t}\r)\r|\lf(\frac{|y-x_j|}{t}\r)^{\lceil \gamma\rceil+1}|Ta_j(y)|\,dy,
\end{align*}
which, combined with the H\"older inequality and the fact that $T$ is bounded on $L^2(\rn)$,
implies that, for all $t\in(0,\fz)$ and $x\in(4B_j)^{\com}$,
\begin{align}\label{3.1.y3}
{\rm A}_1
&\ls\frac{1}{t^n}\int_{|y-x_j|<2r_j}
\frac{t^{n+\lceil \gamma\rceil+1}}{|x-x_j|^{n+\lceil \gamma\rceil+1}}
\frac{|y-x_j|^{\lceil \gamma\rceil+1}}{t^{\lceil \gamma\rceil+1}}|Ta_j(y)|\,dy\noz\\
&\ls\frac{r_j^{\lceil \gamma\rceil+1}}
{|x-x_j|^{n+\lceil \gamma\rceil+1}}\|Ta_j\|_{L^2(\rn)}|B_j|^{1/2}
\ls\frac{r_j^{n+\lceil \gamma\rceil+1}}{|x-x_j|^{n+\lceil \gamma\rceil+1}}
\frac{1}{\|\chi_{B_j}\|_{\lv}}.
\end{align}

For ${\rm A}_2$, by the Taylor remainder theorem, the vanishing
moment condition of $a_j$, the fact that
$\lceil \gamma\rceil\le n(\frac{1}{p_-}-1)\le s$,
\eqref{3.1.y} and the H\"older inequality,
we conclude that, for all $z\in B_j$, there exists $\xi_2(z)\in B_j$ such that,
for all $t\in(0,\fz)$ and $x\in(4B_j)^{\com}$,
\begin{align}\label{3.1.y4}
{\rm A}_2&=\frac{1}{t^n}\int_{2r_j\le|y-x_j|<\frac{|x-x_j|}{2}}
\lf|\phi\lf(\frac{x-y}{t}\r)-\sum_{|\beta|\le\lceil \gamma\rceil}
\frac{D^\beta\phi\lf(\frac{x-x_j}{t}\r)}{\beta!}
\lf(\frac{y-x_j}{t}\r)^\beta\r||Ta_j(y)|\,dy\noz\\
&\ls\int_{2r_j\le|y-x_j|<\frac{|x-x_j|}{2}}
\frac{|y-x_j|^{\lceil \gamma\rceil+1}}{|x-x_j|^{n+\lceil \gamma\rceil+1}}\noz\\
&\quad\quad\times\lf[\int_{B_j}|a_j(z)|\lf|k(y,z)-
\sum_{|\beta|\le\lceil \gamma\rceil}\frac{\partial_y^\beta
k(y,x_j)}{\beta!}(z-x_j)^\beta\r|\,dz\r]\,dy\noz\\
&\ls\frac{1}{|x-x_j|^{n+\lceil \gamma\rceil+1}}\int_{2r_j\le|y-x_j|<\frac{|x-x_j|}{2}}
|y-x_j|^{\lceil \gamma\rceil+1}\noz\\
&\quad\quad\times\int_{B_j}|a_j(z)|\lf|\sum_{|\beta|=\lceil \gamma\rceil}\frac{\partial_y^\beta
k(y,x_j)-\partial_y^\beta k(y,\xi_2(z))}{\beta!}(z-x_j)^\beta\r|\,dz\,dy\noz\\
&\ls\frac{1}{|x-x_j|^{n+\lceil \gamma\rceil+1}}\int_{2r_j\le|y-x_j|<\frac{|x-x_j|}{2}}
|y-x_j|^{\lceil \gamma\rceil+1}\int_{B_j}|a_j(z)|\frac{|z-x_j|^\gamma}
{|y-x_j|^{n+\gamma}}\,dz\,dy\noz\\
&\ls\frac{r_j^\gamma}{|x-x_j|^{n+\lceil \gamma\rceil+1}}\int_{2r_j\le|y-x_j|<\frac{|x-x_j|}{2}}
\frac{1}{|y-x_j|^{n+\gamma-\lceil \gamma\rceil-1}}\,dy\|a_j\|_{L^2(\rn)}|B_j|^{1/2}\noz\\
&\ls\frac{r_j^{n+\gamma}}{|x-x_j|^{n+\gamma}}\frac{1}{\|\chi_{B_j}\|_{\lv}}.
\end{align}

For ${\rm A}_3$, by the vanishing moment condition of $a_j$, the fact that
$\lceil \gamma\rceil\le n(\frac{1}{p_-}-1)\le s$,
\eqref{3.1.y} and the H\"older inequality,
we find that, for all $z\in B_j$, there exists $\xi_3(z)\in B_j$ such that,
for all $t\in(0,\fz)$ and $x\in(4B_j)^{\com}$,
\begin{align}\label{3.1.y5}
{\rm A}_3
&\le\int_{|y-x_j|\ge\frac{|x-x_j|}{2}}\lf|\frac{1}{t^n}
\lf[\phi\lf(\frac{x-y}{t}\r)-\sum_{|\beta|\le\lceil \gamma\rceil}
\frac{D^\beta\phi\lf(\frac{x-x_j}{t}\r)}{\beta!}\lf(\frac{y-x_j}{t}\r)^\beta\r]\r|\noz\\
&\quad\quad\times\lf\{\int_{B_j}|a_j(z)|\lf|k(y,z)-\sum_{|\beta|\le\lceil \gamma\rceil}
\frac{\partial_y^\beta k(y,x_j)}{\beta!}(z-x_j)^\beta\r|\,dz\r\}\,dy\noz\\
&\ls\int_{|y-x_j|\ge\frac{|x-x_j|}{2}}\lf|\frac{1}{t^n}
\lf[\phi\lf(\frac{x-y}{t}\r)-\sum_{|\beta|\le\lceil \gamma\rceil}
\frac{D^\beta\phi\lf(\frac{x-x_j}{t}\r)}{\beta!}\lf(\frac{y-x_j}{t}\r)^\beta\r]\r|\noz\\
&\quad\quad\times\int_{B_j}|a_j(z)|\lf|
\sum_{|\beta|=\lceil \gamma\rceil}\frac{\partial_y^\beta
k(y,x_j)-\partial_y^\beta k(y,\xi_3(z))}{\beta!}(z-x_j)^\beta\r|\,dz\,dy\noz\\
&\ls\int_{|y-x_j|\ge\frac{|x-x_j|}{2}}|\phi_t(x-y)|\int_{B_j}|a_j(z)|
\frac{|z-x_j|^\gamma}{|y-x_j|^{n+\gamma}}\,dz\,dy\noz\\
&\quad\quad+\int_{|y-x_j|\ge\frac{|x-x_j|}{2}}\lf|\frac{1}{t^n}
\sum_{|\beta|\le\lceil \gamma\rceil}
\frac{D^\beta\phi\lf(\frac{x-x_j}{t}\r)}{\beta!}
\lf(\frac{y-x_j}{t}\r)^\beta\r|\noz\\
&\quad\quad\times\int_{B_j}|a_j(z)|
\frac{|z-x_j|^\gamma}{|y-x_j|^{n+\gamma}}\,dz\,dy\noz\\
&\ls\frac{r_j^\gamma}{|x-x_j|^{n+\gamma}}\|a_j\|_{L^2(\rn)}|B_j|^{1/2}
\int_{|y-x_j|\ge\frac{|x-x_j|}{2}}\lf|\phi_t(x-y)\r|\,dy\noz\\
&\quad\quad+\sum_{|\beta|\le\lceil \gamma\rceil}
r_j^\gamma\|a_j\|_{L^2(\rn)}|B_j|^{1/2}\noz\\
&\quad\quad\times\int_{|y-x_j|\ge\frac{|x-x_j|}{2}}\frac{1}{t^n}
\frac{t^{n+|\beta|}}{|x-x_j|^{n+|\beta|}}
\frac{|y-x_j|^{|\beta|}}{|t|^{|\beta|}}\frac{1}{|y-x_j|^{n+\gamma}}\,dy\noz\\
&\ls\frac{r_j^{n+\gamma}}{|x-x_j|^{n+\gamma}}\frac{1}{\|\chi_{B_j}\|_{\lv}}.
\end{align}

Combining \eqref{3.1.y2}, \eqref{3.1.y3}, \eqref{3.1.y4} and \eqref{3.1.y5},
we conclude that, for all $x\in(4B_j)^{\com}$,
\begin{align*}
|\phi_{+}^\ast(Ta_j)(x)|&=\sup_{t\in(0,\fz)}\lf|\phi_t*Ta_j(x)\r|
\ls\frac{r_j^{n+\gamma}}{|x-x_j|^{n+\gamma}}\frac{1}{\|\chi_{B_j}\|_{\lv}}\\
&\ls\lf[\cm(\chi_{B_j})(x)\r]^{\frac{n+\gamma}{n}}
\frac{1}{\|\chi_{B_j}\|_{\lv}},
\end{align*}
which implies that
$$\phi_{+}^\ast(Ta_j)(x)\chi_{(4B_j)^{\com}}(x)
\ls\lf[\cm(\chi_{B_j})(x)\r]^{\frac{n+\gamma}{n}}
\frac{1}{\|\chi_{B_j}\|_{\lv}}.$$
Therefore, by Proposition \ref{1.14.x3} and an argument similar to that
used in the proof of \eqref{1.26.x2}, we know that
\begin{align*}
&\az\lf\|\chi_{\{x\in\rn:\ \sum_{j\in\nn}\lz_j
\phi_{+}^\ast(Ta_j)(x)\chi_{(4B_j)^{\complement}}(x)
>\frac{\az}2\}}\r\|_{\lv}\\
&\hs\ls\frac{\az}{2}\lf\|\chi_{\{x\in\rn:\ [\sum_{j\in\nn}\frac{\lz_j}
{\|\chi_{B_j}\|_{\lv}}\{\cm(\chi_{B_j})(x)\}^{\frac{n+\gamma}{n}}]
^{\frac{n}{n+\gamma}}>(\frac{\az}{2})^{\frac{n}{n+\gamma}}\}}\r\|
_{L^{\frac{(n+\gamma)p(\cdot)}{n}}(\rn)}^{\frac{n+\gamma}{n}}\\
&\hs\ls\lf\|\lf[\sum_{j\in\nn}\frac{\lz_j\chi_{B_j}}
{\|\chi_{B_j}\|_{\lv}}\r]^{\frac{n}{n+\gamma}}\r\|
_{L^{\frac{(n+\gamma)p(\cdot)}{n}}(\rn)}^{\frac{n+\gamma}{n}}
\ls\|f\|_{H^{p(\cdot)}(\rn)},
\end{align*}
which shows \eqref{3.4.x1} holds true and hence completes the proof of Theorem \ref{bdnthm3}.
\end{proof}

\textbf{ Acknowledgement.} The authors would like to express their
deep thanks to the referee for his very careful reading and several useful
comments which improve the presentation of this article.

\bigskip

\noindent Xianjie Yan, Dachun Yang (Corresponding author), Wen Yuan
and Ciqiang Zhuo

\medskip

\noindent  School of Mathematical Sciences, Beijing Normal University,
Laboratory of Mathematics and Complex Systems, Ministry of
Education, Beijing 100875, People's Republic of China

\smallskip

\noindent {\it E-mails}: \texttt{xianjieyan@mail.bnu.edu.cn} (X. Yan)

\hspace{0.98cm} \texttt{dcyang@bnu.edu.cn} (D. Yang)

\hspace{0.98cm} \texttt{wenyuan@bnu.edu.cn} (W. Yuan)

\hspace{0.98cm} \texttt{cqzhuo@mail.bnu.edu.cn} (C. Zhuo)

\end{document}